\documentclass{article}
\usepackage{cite}
\usepackage[colorlinks=true, linkcolor=blue, citecolor=red, urlcolor=blue]{hyperref}

\usepackage[utf8]{inputenc}

\usepackage{cite}

\usepackage[pdftex]{graphicx}
\usepackage[pdftex,dvipsnames,usenames]{color}

\usepackage{enumitem}
\usepackage{mathtools}

\usepackage{cancel}

\usepackage[english]{babel}

\hypersetup{
  colorlinks   = true, % Colours links instead of ugly boxes
  urlcolor     = blue, % Colour for external hyperlinks
  linkcolor    = blue, % Colour of internal links
  citecolor   = red % Colour of citations
}

\usepackage{amsmath,amssymb,amsfonts,mathrsfs, amsthm}

\usepackage{graphicx}
\usepackage[margin=0.75in]{geometry}

\usepackage{soul}

\usepackage{pdfpages}

\usepackage{lipsum}
\usepackage{csquotes}
\usepackage{thmtools}

\usepackage{amsfonts,amsmath,amstext,amssymb,amsthm}
\usepackage{enumitem}
\usepackage[english]{babel}
\usepackage{fancybox}
\usepackage{xspace}
\usepackage{amscd}
\usepackage[all]{xy}
\usepackage[svgnames]{xcolor}
\usepackage{calc}

\usepackage{esint}

\usepackage{enumitem}

\usepackage{cases}
\usepackage{eqnarray}

\usepackage[compact]{titlesec}

\usepackage{accents}

\newcommand{\B}{\mathcal{B}}
\def\XXint#1#2#3{{\setbox0=\hbox{$#1{#2#3}{\int}$ }
\vcenter{\hbox{$#2#3$ }}\kern-.56\wd0}}

\newcommand{\Qb}{\mathcal{Q}}

\newcommand{\N}{\mathbb{N}}
\newcommand{\R}{\mathbb{R}}
\newcommand{\Z}{\mathbb{Z}}
\newcommand{\bS}{\mathbb{S}}
\newcommand{\p}{\mathfrak p}

\newcommand{\cH}{\mathcal{H}}

\newcommand{\Sp}{\mathbb{S}}

\newcommand{\ep}{\varepsilon}
\newcommand{\conv}[1]{\xrightarrow{\,#1\,}}

\newcommand{\s}{\hspace{7pt}}
\newcommand{\vsp}{\vspace{4.5pt}}
\newcommand{\ov}[1]{\widetilde{#1}}

\newcommand{\sg}{\sigma}
\newcommand{\inj}{\textnormal{inj}}

\newtheorem{theorem}{Theorem}[section]
\newtheorem{proposition}[theorem]{Proposition}
\newtheorem{lemma}[theorem]{Lemma}
\newtheorem{corollary}[theorem]{Corollary}

\theoremstyle{definition}

\newtheorem{definition}[theorem]{Definition}

\newtheorem{remark}[theorem]{Remark}

\usepackage{url}

\makeatletter
\renewenvironment{abstract}{
  \par\smallskip
  \noindent\textsc{Abstract.}\ %
}{
  \par\medskip
}
\makeatother

\makeatletter
\newcommand{\addressa}[1]{\gdef\@addressa{#1}}
\newcommand{\emaila}[1]{\gdef\@emaila{\url{#1}}}

\newcommand{\@endstuff}{\par\vspace{\baselineskip}\noindent
\begin{tabular}{@{}l}\scshape\@addressa\\\textit{E-mail address:} \@emaila\end{tabular} 

}

\AtEndDocument{\@endstuff}
\makeatother

\begin{document}
\setlength{\abovedisplayskip}{6pt}          % space above displayed equations
\setlength{\belowdisplayskip}{6pt}          % space below displayed equations
\setlength{\abovedisplayshortskip}{6pt}     % space above short equations
\setlength{\belowdisplayshortskip}{6pt}     % space below short equations

\setlength{\parskip}{0pt plus 1pt}
\setlength{\parindent}{15pt}
\setlength{\abovedisplayskip}{5pt plus 1pt minus 2pt}
\setlength{\belowdisplayskip}{5pt plus 1pt minus 2pt}

\setlist[itemize]{itemsep=2pt, topsep=3pt, partopsep=1pt, parsep=1pt}
\setlist[enumerate]{noitemsep, topsep=2pt, partopsep=0pt, parsep=0pt}

\titlespacing{\section}{0pt}{10pt plus 2pt minus 2pt}{6pt plus 1pt minus 1pt}
\titlespacing{\subsection}{0pt}{8pt plus 2pt minus 2pt}{5pt plus 1pt minus 1pt}

\title{\huge Weyl Law and convergence in the classical limit for min-max nonlocal minimal surfaces}
\author{Enric Florit-Simon}
\date{}
\maketitle

\addressa{Enric Florit-Simon \\ Department of Mathematics, ETH Z\"{u}rich \\ Rämistrasse 101, 8092 Zürich, Switzerland}
\emaila{enric.florit@math.ethz.ch}

\abstract{We study nonlocal minimal surfaces as a new approximation theory for the area functional, and more specifically in the context of Yau's conjecture on the existence of minimal surfaces in closed three-dimensional manifolds. This programme offers an alternative to the Almgren--Pitts and Allen--Cahn approaches, with advantageous features both from the existence and regularity viewpoints.

We build on recent work in which the author and collaborators constructed infinitely many nonlocal $s$-minimal hypersurfaces (via min-max methods) on any closed $n$-dimensional Riemannian manifold, establishing a full analogue of Yau's conjecture for $s\in(0,1)$.

The present article first proves a Weyl-type Law for the fractional perimeters of these hypersurfaces. The rest---and main part---of the article is devoted to obtaining uniform estimates (in the classical limit $s\to 1$) for min-max $s$-minimal surfaces in closed three-manifolds, eventually establishing their convergence to smooth, classical minimal surfaces. We recover in particular recent results on existence, generic density and equidistribution of minimal surfaces, which are a strong form of Yau's conjecture in this setting.}

\tableofcontents

\section{Introduction and main results}
In \cite{CFS}, the author and collaborators obtained the existence (for any fixed $s\in(0,1)$) of infinitely many nonlocal $s$-minimal (hyper)surfaces $\{E_{\p}^s\}_{\p\in\N}$ in any closed Riemannian manifold $M$ of arbitrary dimension $n$, which in particular proves the nonlocal version of a well-known conjecture (recently proved by Song \cite{Song}) of Yau \cite{Yau82} on the existence of infinitely many classical minimal surfaces on closed three-dimensional manifolds. A partial statement of the result in \cite{CFS} is the following:
\begin{theorem}[\textbf{Fractional Yau Conjecture \cite{CFS}}]\label{FracYau3}
    Let $(M^n, g)$ be an $n$-dimensional, closed Riemannian manifold. Then, for every $s\in(0,1)$ and every natural number $\p\in\N^+$, there exists an $s$-minimal surface $E_{\p}^s$ with Morse index at most $\p$ (in the sense of Definition \ref{WeakMorseDef}) and fractional perimeter
    \begin{equation}\label{PerGrowth}
    \frac{c(M)}{1-s}\p^{s/n}\le {\rm Per}_s(\partial E_{\p}^s)=: l_s(\p,M)\le \frac{C(M)}{1-s}\p^{s/n},
\end{equation}
    where $c(M)$ and $C(M)$ are uniform as $s\to 1$.\\
    In particular, $M$ contains infinitely many $s$-minimal surfaces. In addition,  if  $s\in(s_0,1)$ where $s_0\in(0,1)$ is a universal constant, if $n=3$ or $n=4$ then $\partial E_{\p}^s$ is a smooth hypersurface, whereas if $n\geq 5$ then $\partial E_{\p}^s$ is smooth outside of a set of Hausdorff dimension at most $n-5$.
\end{theorem}
Here $l_s(\p,M)$ corresponds to a (limiting) min-max value arising from considering certain $\p$-parameter families of functions on $M$, and is defined in Section \ref{sec:minmax} in the present article. We remark that we will be employing slightly different min-max families than those in \cite{CFS} --- all the relevant setting and necessary modifications to obtain Theorem \ref{FracYau3} with our definitions are given in Section \ref{Setting}.\\

The main motivation behind the present article is to study nonlocal minimal surfaces as a {\it new approximation theory for the area functional}, more specifically in the context of min-max problems. The first result in this work is that, perhaps surprisingly, a Weyl-type Law holds for the fractional perimeters of these $s$-minimal surfaces, already for a fixed $s$:
\begin{theorem}[\textbf{Weyl Law for nonlocal minimal surfaces}]\label{thmweyl}
Let $(M,g)$ be a closed manifold of dimension $n$, and let $s\in(0,1)$. There exists a universal constant $\tau(n,s)>0$, depending only on $n$ and the fractional parameter $s$, such that
\begin{equation}\label{weyllaw}
    \lim_{\p\to\infty} \p^{-\frac{s}{n}}l_{s}(\p,M) = \tau(n,s)\textnormal{vol}(M,g)^{\frac{n-s}{n}}\, .
\end{equation}
\end{theorem}
This gives a nonlocal analogue of the Weyl-type Law for classical minimal hypersurfaces, which was conjectured by Gromov \cite{Gromov} in general codimension and proved by Liokumovich--Marques--Neves \cite{LMN} in the case of minimal hypersurfaces obtained via the Almgren--Pitts min-max procedure. Soon after that, Gaspar--Guaraco \cite{GG2} showed that the Weyl Law holds for limits (as $\ep\to 0$) of min-max Allen--Cahn solutions as well, with a similar proof to the one in \cite{LMN}. Interestingly, their Weyl Law can hold only for the limits obtained as $\ep\to 0$; on the other hand, Theorem \ref{FracYau3} holds already for any fixed $s\in(0,1)$. We emphasise here that $s$-minimal hypersurfaces are geometric objects on their own, with a canonical definition on Riemannian manifolds (see Section \ref{sec:defssetting} and \cite{FracSobPaper}), and moreover the parameter $s$ is ``dimensionless''; in the Allen--Cahn case, $\ep$ has ``length units'', thus being an Allen--Cahn solution for a certain $\ep$ is not a scaling invariant notion anymore, and solutions with larger energy (i.e. more transitions) can only exist when making $\ep$ smaller.\\

The rest (and main part) of the article focuses on the classical limit $s\to 1^-$. For a smooth open set $E\subset M$, it will follow from our results that
$$\lim_{s\to 1^-} (1-s){\rm Per}_s(\partial E)=\gamma_n{\rm Per}(\partial E)\,,$$
with $\gamma_n$ defined in \eqref{gammadef}. This compels us to define
\begin{equation}\label{deflp1}
    l_1(\p,M):=\liminf_{s\to 1}[(1-s)l_s(\p,M)]=\liminf_{s\to 1}[(1-s)\text{Per}_s(\partial E_{\p}^s)]\, ,
\end{equation}
and we will obtain a Weyl Law for these quantities as well:
\begin{theorem}[\textbf{Weyl Law in the classical limit}]\label{clasweylthm} Let $(M,g)$ be a closed manifold of dimension $n$. There exists a universal constant $\tau(n,1)>0$, depending only on $n$, such that
\begin{equation}\label{clasweyl}
    \lim_{\p\to\infty} \p^{-\frac{1}{n}}l_{1}(\p,M) = \tau(n,1)\textnormal{vol}(M,g)^{\frac{n-1}{n}}\, .
\end{equation}
\end{theorem}
One would expect for $l_{1}(\p,M)$ to correspond to the area of a classical minimal hypersurface $\Sigma_{\p}$ of optimal regularity, obtained as a limit of the $\partial E_{\p}^s$ after letting $s\to 1$. For this reason, the main bulk of the article is devoted to studying the convergence, as $s\to 1$, of sequences of $s$-minimal surfaces with bounded Morse index, focusing on the case $n=3$. We will exhibit their convergence to a smooth classical minimal surface in a remarkably clean way, which could be interpreted as saying that stability assumptions actually confer $s$-minimal hypersurfaces a length scale $\sqrt{1-s}$ as $s\to 1$, thus effectively becoming a regularisation of the area functional. A partial statement is the following:
\begin{theorem}[\textbf{Compactness of $s$-minimal surfaces}]\label{FMIConv}
Let $M$ be a closed Riemannian manifold of dimension $3$. Let $\{E_{s_i}\}_{i}$ be a sequence of $C^2$ $s_i$-minimal surfaces with Morse index bounded by $\p\in\N$, and suppose that $s_i\to 1^-$. Assume that
$$
\sup_i (1-s_i)\textnormal{Per}_{s_i}(E_{s_i})<\infty\, .
$$
Then, up to passing to a subsequence (not relabeled), the following hold:\\
There exists a collection of smooth, connected, disjoint classical minimal surfaces $\Sigma_1,...,\Sigma_m$, together with positive integers $n_1,...,n_m$, such that the $\partial E_{s_i}$ converge in the sense of varifolds to $\sum_{k=1}^m n_k[\Sigma_k]$, where $[\Sigma_k]$ denotes the integer rectifiable varifold associated to the smooth surface $\Sigma_k$. Moreover, 
$$\lim_{i\to\infty}(1-s_i)\textnormal{Per}_{s_i}(E_{s_i})=\gamma_n\sum_{k=1}^m n_k\textnormal{Per}(\Sigma_k)\,.
$$
Regarding the convergence, there exists a (possibly empty) collection of points $q_1,...,q_{l}$, with $l\leq\p$, such that: in a neighbourhood of every point in $\Sigma_k\setminus\{q_1,...,q_{l}\}$, for every $\alpha\in(0,1)$ the $\partial E_{s_i}$ converge in $C^{2,\alpha}$ as $n_k$ ordered normal graphs with separation of order at least $\sqrt{1-s_i}$.\\
Finally, the Morse index is lower semicontinuous, in the sense that 
$$\sum_{k=1}^m \textnormal{index}(\Sigma_k)\leq \p\,.
$$
\end{theorem}
Applying this result to the $E_\p^s$ from Theorem \ref{FracYau3}, which have index bounded by $\p$, we get:
\begin{corollary}[\textbf{Attainability of min-max widths}]\label{FracYau4}
    Let $(M^n, g)$ be a $3$-dimensional, closed Riemannian manifold. Then, for every natural number $\p\in\N^+$, there exists a collection of smooth, connected, disjoint classical minimal surfaces $\Sigma_1,...,\Sigma_{m_\p}$, together with positive integers $n_1,...,n_{m_\p}$, such that
    $$l_{1}(\p,M)=\gamma_n\sum_{k=1}^{m_\p} n_k\textnormal{Per}(\Sigma_k)\,.$$
    Moreover, $\sum_{i=k}^{m_\p} \textnormal{index}(\Sigma_k)\leq \p$.
\end{corollary}
Combining Corollary \ref{FracYau4} with the asymptotics for the $l_1(\p,M)$ in Theorem \ref{clasweylthm}, this gives a Weyl-type Law for classical minimal surfaces like the one by Liokumovich--Marques--Neves \cite{LMN}. With the Weyl Law at hand, we can in particular recover the results by Irie--Marques--Neves \cite{IMN} and Marques--Neves--Song \cite{MNS} on the density and equidistribution of minimal surfaces on closed manifolds with a generic metric:
\begin{theorem}[\textbf{Density}]\label{Density}
Let $M$ be a closed manifold of dimension $3$.
For a $C^\infty$-generic Riemannian metric $g$ on $M$, the union of all closed, smooth, embedded minimal hypersurfaces in $(M,g)$ is dense (and in particular, there are infinitely many such hypersurfaces).
\end{theorem}
\begin{theorem}[\textbf{Equidistribution}]\label{Equi}
    Let $M$ be a closed manifold of dimension $3$.
    For a $C^\infty$-generic Riemannian metric $g$ on $M$, there exists a sequence $\{\Sigma_i\}_{i\in\N}$ of closed, smooth, embedded minimal hypersurfaces in $(M,g)$ which is equidistributed. More precisely, for any $f\in C^\infty(M)$,
    \begin{equation*}
        \lim_{k\to\infty} \frac{1}{\sum_{i=1}^k {\rm Per}(\Sigma_i)}\sum_{i=1}^k \int_{\Sigma_i} f\,d\Sigma_i = \frac{1}{{\rm Vol}_g(M)} \int_M f\,dM\, .
    \end{equation*}
\end{theorem}
We refer the reader to \cite{Alm2, CM, MN, Pitts, SS, Zhou} as well as the references \cite{IMN, LMN, MNS, Song} already mentioned in this introduction for an overview of some of the main developments that have led to the proof of Yau's conjecture.

\section{Overview of the article and further directions}
The proofs of the Weyl Law(s) and of Theorems \ref{Density} and \ref{Equi} follow the ideas of \cite{LMN} and of \cite{IMN, MNS}. The main difficulties in the nonlocal case stem from the presence of nonlocal interactions and the fact that the fractional Sobolev energy on a closed manifold $M$ is canonically defined through a singular kernel $K_s(p,q)$ which depends on the heat kernel of the entire manifold (see \eqref{wethiowhoihw2}): Indeed, the proof of the Weyl Law is based on comparison between $M$ and its subdomains, and between subdomains of $M$ and subdomains of Euclidean space. Likewise, the proofs of Theorems \ref{Density} and \ref{Equi} require the precise knowledge of how changing the metric of $M$ affects the behaviour of the problem. Therefore, the right setting needs to be selected carefully. Our chosen approach, which we believe to be of independent interest, is to first prove a Weyl Law for a related min-max problem, where an analogous fractional Sobolev energy is defined through a different kernel which depends explicitly on the distance function of $M$. We then show how this implies the Weyl Law for the canonical, original problem as well. The study of the precise behaviour of the singular kernel $K_s$, which was initiated in \cite{FracSobPaper} and is continued in the present paper, is used in an essential way in the aforementioned analysis as well as in the rest of the paper.\\

The proof of the compactness of $s$-minimal surfaces with bounded index occupies the rest (and main bulk) of our work. Most of the results are of a remarkably ``clean'' nature, providing---together with \cite{CFS}---a self-contained and robust (as $s\to 1$) regularity theory for min-max $s$-minimal surfaces. Combined with the elementary existence theory in \cite{CFS}, this showcases the nonlocal approximation as an attractive alternative to the Almgren--Pitts and Allen--Cahn approaches. For comparison, the Allen--Cahn approach offers a simple existence theory too, but at the expense of a very hard regularity theory:
\begin{itemize}
    \item In general dimension, the only available results are via the deep and powerful regularity theory in \cite{Wick} for stable integral varifolds. This method presents moreover important difficulties, such as a potential loss of geometric information in the prescribed mean curvature \cite{BW} or manifold-with-boundary \cite{PSL} cases.
    \item In dimension $n=3$, a stronger, alternative regularity theory is present, developed in the (similarly deep) breakthrough articles \cite{WW2,WW,CM}. An extension to higher dimensions depends on the longstanding open problem of classification of stable solutions (with bounded energy density) to Allen--Cahn, which so far has only been achieved for $n=4$ in the very recent (independent) article \cite{FS}.
\end{itemize}
Now, despite the ``clean" nature of the nonlocal approach, the regularity and compactness part of the article will require a lengthier exposition. The main reason is that, in fact, we need to build most of the theory from scratch: As an example, our very first step is to obtain the (challenging, yet deceptively simple-looking) formula for the second variation of the fractional perimeter of a domain on a manifold, which will require considering a smooth extension to the entire manifold of the normal vector of the domain, and will then take the form of a singular integral which will be shown to converge due to second order cancellations.

A point we consider most interesting in our strategy is the eventual proof of global estimates for $s$-minimal surfaces with bounded index, including the uniform, quantitative $L^1$ decay of their {\it classical} mean curvatures as $s\to 1$ which holds even around points of singular convergence, and which will allows us to obtain the convergence in the varifold sense to a limit. Once the limit varifold is obtained, we will show its regularity outside of finitely many points; then, a careful argument using this partial regularity will allow us to show the convergence of the second (or any) variations of the approximating $s$-minimal surfaces to those of the limit, as well as the convergence of the fractional perimeters, and only then we will remove the singularities thanks to the stability properties inherited by the limit. We emphasise here that, to our knowledge, this is the first instance of a compactness theorem for $s$-minimal surfaces as $s\to 1$ without the presence of uniform curvature estimates; in particular, the convergence of critical points under exclusively an area bound assumption, as in \cite{All} or \cite{HutTon}, is an open and very interesting problem.

Section \ref{sec:asests}, which exhibits uniform $C^{2,\alpha}$ and sheet separation estimates for almost-stable $s$-minimal surfaces, builds on the recent breakthrough article \cite{CDSV}, which among other things studies the stable case (without any area bound assumption) on $\R^3$ and proves the classification of stable $s$-minimal cones in $\R^4$ as hyperplanes for $s$ close to $1$. Most of the proofs in our Section \ref{sec:asests} introduce alternative arguments, in order to deal with the challenging form of the second variation formula and the Riemannian setting. As a notable point, through our approach we obtain uniform $C^{2,\alpha}$ estimates for every $\alpha\in(0,1)$, an improvement of the result in \cite{CDSV} which is for $\alpha$ sufficiently small.\\

We expect our methods and results to be extended soon to dimensions $4\leq n \leq 7$; the missing points are a classification of stable $s$-minimal cones (under the assumption of area bounds as $s\to 1$) in dimensions $5\leq n\leq 7$, a problem which should be much simpler than an unconditional classification of stable $s$-minimal cones\footnote{This should be compared, for example, with the stable Bernstein conjecture on the classification of stable minimal hypersurfaces in $\R^n$ with and without assuming area bounds.}, and an adaptation of Section \ref{sec:asests} to dimensions $4\leq n \leq 7$. For $n=3$, one can decouple the problem of quantifying the separation between different layers of an $s$-minimal surface from the regularity of the individual layers. Attacking both problems at the same time using an approach as in \cite{WW} (see also \cite{CM}) would seem like a promising strategy for $4\le n \le 7$, and we believe that the methods introduced in the present article will be useful for this approach.

Regarding the nonlocal Weyl Law in Theorem \ref{thmweyl}, it would be interesting to study which properties it implies for the $s$-minimal hypersurfaces $\partial E_\p^s$. The proof of Theorem \ref{Density} would seem hard to adapt to the nonlocal case, given the nonlocal interactions inherent to the objects under consideration. The proof of Theorem \ref{Equi}, a result which is stronger anyway, could however potentially have an adaptation to the nonlocal setting. As a matter of fact, we would expect even stronger equidistribution properties to hold in the nonlocal case, such as the possibility of obtaining results for arbitrary (and not only generic) metrics.

\section{Setting}\label{Setting}
\subsection{The fractional setting on a manifold}
\subsubsection{Definitions}\label{sec:defssetting}
Let $(M,g)$ be a closed (i.e. compact and without boundary) Riemannian manifold of dimension $n$. Let $H_M(t,p,q)$ be the heat kernel of $M$, and let $s\in(0,2)$. Set
 \begin{equation}\label{wethiowhoihw2}
K_{M,s}(p,q) := \frac{s/2}{\Gamma(1-s/2)} \int_{0}^{\infty} H_M(p,q,t) \frac{dt}{t^{1+s/2}}\,.
 \end{equation}
We will often omit the dependence in $M$ and $s$ of $K$ when there is no risk of confusion.\\
Define
\begin{equation}\label{alphadef}
    \alpha_{n,s}= \frac{2^s \Gamma\Big(\tfrac{n+s}{2}\Big)}{\pi^{ n/2} |\Gamma(-s/2)| } = \frac{s 2^{s-1} \Gamma\Big(\tfrac{n+s}{2}\Big)}{\pi^{n/2}\Gamma(1-s/2)}.
\end{equation}
Then, the kernel $K_{M,s}(p,q)$ differs from $\frac{\alpha_{n,s}}{d_M(p,q)^{n+s}}$ only by a lower order term, see Proposition \ref{kermfdconv}, and they coincide exactly in the case $M=\R^n$.\\
We then define the fractional Sobolev space $H^{s/2}(M)$ as the space of functions $u\in L^2(M)$ such that
 \begin{equation}\label{wethiowhoihw}
 [u]^2_{H^{s/2}(M)} := \iint_{M\times M}(u(p)-u(q))^2 K_{M,s}(p,q) \,dV_p \,d V_q
 \end{equation}
is finite. Several other definitions for the fractional Sobolev seminorm, identical (and not just equivalent) to \eqref{wethiowhoihw}, are given in \cite{CFS}, showing that it is canonically defined. We recall here in particular the spectral definition
\begin{equation}\label{ghfghfg}
    [u]^2_{H^{s/2}(M)} = 2\sum_{k\ge 1} \lambda_k^{s/2} \langle u,\phi_k \rangle^2_{L^2(M)}\,,
\end{equation}
 where $\{\phi_k\}_k$ is an orthonormal basis of eigenfunctions of the Laplace-Beltrami operator $(-\Delta_g)$ and $\{\lambda_k\}_k$ are the corresponding eigenvalues. This shows that in the limit $s\to 2$ we recover the usual $[u]^2_{H^{1}(M)}$ seminorm.\\
 
We can now define the fractional perimeter $\text{Per}_s$. Given $s \in (0, 1)$
and a (measurable) set $E \subset M$, we define
\begin{equation}
\textnormal{Per}_s(E): = [\chi_E]^2_{H^{{s/2}}(M)} = \frac{1}{4}[\chi_E-\chi_{E^c}]^2_{H^{{s/2}}(M)}  = 2\int_{E}\int_{E^c} K_{M,s}(p,q)dV_pdV_q\,, 
\end{equation}
where $\chi_E$ is the characteristic function of $E$, and $E^c: = M\setminus E$. In practice, $E$ will always be a domain with some regularity, and the notation $\textnormal{Per}_s(\partial E)$ will sometimes be used to denote $\textnormal{Per}_s(E)$ instead in analogy with the classical case. More generally, if $\Omega\subset M$, we set
\begin{equation}
\textnormal{Per}_s(E;\Omega): = 2\int_{E\cap\Omega}\int_{E^c\cap\Omega} K_{M,s}(p,q)dV_pdV_q\,.
\end{equation}
We have omitted the interactions between $\Omega$ and $\Omega^c$, which will be convenient in the present paper, but which the reader should keep in mind when consulting related literature in which different notation might be employed.

One can see that, for every subset $E\subset M$  with smooth boundary,
$(1-s)\textnormal{Per}_s(\partial E) \to \gamma_n\textnormal{Per}(\partial E)$  as $s\uparrow 1$, where
\begin{equation}\label{gammadef}
    \gamma_n= 2\alpha_{n,s}\omega_{n-1}
\end{equation}
and $\omega_{n-1}=|\B^{n-1}|$ indicates the volume of the unit ball in $\R^{n-1}$. See \cite{ADPM11} for further details in the case of $\R^n$, or the proof of Proposition \ref{globperconv} in the present paper for a significantly stronger result.

\begin{definition}[\textbf{$s$-minimal surface}]\label{def-fracminsurface}
Let $M$ be a closed Riemannian manifold.  Given $s\in (0,1)$, the boundary $\partial E$ of a set $E\subset M$ (or the set $E$ itself) is said to be an $s$-\textit{minimal surface} if  $\textnormal{Per}_s(E)<\infty$, and
\begin{equation}\label{innercrit}
     \frac{d}{dt}\Big|_{t=0} \textnormal{Per}_s(\psi^t_X(E)) = 0\qquad\mbox{for every vector field}\quad X\in C^2(M) \, ,
\end{equation}
where $\psi^t_X : M \times \R \to M$ denotes the flow of $X$ at time $t$.
\end{definition}

The previous definition admits a natural local version.
\begin{definition}\label{def-fracminsurface2}
Let $(M,g)$ be a closed Riemannian manifold. Given $\mathcal U\subset M$ open, the boundary $\partial E$ of a set $E\subset M$ is said to be an $s$-\textit{minimal surface} in $\mathcal U$ if $\textnormal{Per}_s(E)<\infty$, and for every smooth and compactly supported vector field $X$ in $U$ (denoted $X\in \mathfrak X_c(\mathcal U)$) we have
\begin{equation*}
     \frac{d}{dt}\Big|_{t=0} \textnormal{Per}_s(\psi^t_X(E)) = 0 \, .
\end{equation*}
\end{definition}

\begin{definition}[\textbf{Morse index and stability}]\label{WeakMorseDef} Let $(M,g)$ be a closed Riemannian manifold and $\partial E$ be an $s$-minimal surface in $\mathcal U\subset M$ open (as in Definition \ref{def-fracminsurface2}). Then,  $\partial E$ is said to have {\em Morse index} at most $m$ in $\mathcal U$ if for every $(m+1)$ vector fields $X_0,\dotsc ,X_m \in \mathfrak X_c(\mathcal U)$ there exists some linear combination $X=a_0X_0+\dots +a_mX_m$ with $a_0^2 + a_1^2+\dots +a_m^2=1$ such that 
\begin{equation*}
    \frac{d^2}{dt^2}\Big|_{t=0}  \textnormal{Per}_s(\psi^t_X(E)) \ge 0.
\end{equation*}
In the particular case $m=0$, we say that $\partial E$ is \textit{stable} in $\mathcal U$.
\end{definition}

\begin{remark}
As we proved in Lemma 3.10 in \cite{FracSobPaper}, if $\textnormal{Per}_s(E)<\infty$ and $X\in \mathfrak X_c(\mathcal U)$ then the map $t\mapsto \textnormal{Per}_s(\psi^t_X(E))$ is well-defined for all $t$ and of class $C^{\infty}$. Thus, the previous definitions are meaningful. 
\end{remark}

The $s$-minimal surfaces in Theorem \ref{FracYau3} are obtained as limits as $\ep\to 0$ of solutions to the fractional Allen-Cahn equation on $M$. Given $v:M\to\R$, the fractional Allen-Cahn energy of $v$ is defined as
\begin{equation}\label{ACenergy}
 \mathcal E_{\ep,s} (v,M) :=  \mathcal E_{\ep,s}^{\rm Sob}(v,M)   + \mathcal E_{\ep,s}^{\rm Pot}(v,M) ,
 \end{equation}
 where
 \[
  \mathcal E_{\ep,s}^{\rm Sob}(v,M) := \frac 1 4 \iint_{M\times M} (v(p)-v(q))^2 K_{M,s}(p,q)dV_p dV_q,   \quad \mathcal E_{\ep,s}^{\rm Pot}(v,M) :=\ep^{-s}\int_{M} W(v)\,dx\, ,
  \]
and $W(v)=\frac{1}{4}(1-v^2)^2$ is the standard quartic double-well potential with wells at $\pm 1$. The double-well potential penalizes functions which are not identical to $\pm 1$, thus why one expects to find nonlocal $s$-minimal surfaces as the limits of critical points of this energy when $\ep\to 0$.
A function $u:M\to \R $ is a critical point of $\mathcal E_{\ep,s}$ if and only if it satisfies the Allen-Cahn equation
\begin{equation}\label{restrictedeq}
 (-\Delta)_M^{s/2}u + \ep^{-s}W'(u)=0 \, .
\end{equation}
Here $(-\Delta)^{{s/2}}$ is the fractional Laplacian on $(M,g)$, and it can be represented as
\begin{equation}
    (-\Delta)_M^{{s/2}} u(p) = \int_{M} (u(p)-u(q))K_{M,s}(p,q) \, dV_q \, .
\end{equation}
\subsubsection{Fundamental properties}
Here, and in the rest of the paper, $\B_R(x)$ denotes the Euclidean ball of radius $R$ centered at $x\in\R^n$, and $B_R(p) $ denotes the metric ball on $M$ of radius $R$ and center $p$. We will often write $\B_R$ in place of $\B_R(0)$.

\begin{definition}[\textbf{Local flatness assumption}]\label{flatnessassup}
Let $(M^n,g)$ be an $n$-dimensional Riemannian manifold and $p\in M$. For $R>0$, we say that $ (M,g)$ satisfies the $\ell$-th order \textit{flatness assumption at scale $R$ around the point $p$, with parametrization $\varphi$}, abbreviated as ${\rm FA}_\ell(M,g, R,p,\varphi)$,  whenever there exists an open neighborhood $V$ of $p$ and a diffeomorphism
\begin{equation*}
\varphi:  \B_R(0) \to V, \quad \mbox{with }\varphi(0)=p_0 \,,  
\end{equation*}
such that, letting $g_{ij}=g\left(\varphi_*\left(\frac{\partial}{\partial x^i}\right), \varphi_* \left(\frac{\partial}{\partial x^j} \right) \right)$  
be the representation of the metric $g$ in the coordinates $\varphi^{-1}$, we have
\begin{equation}\label{hsohoh1}
\big( 1-\frac{1}{100}\big) |v|^2 \le  g_{ij}(x)v^i v^j \le \big( 1+\frac{1}{100}\big)|v|^2 \s\s \forall \, v\in \R^n  \mbox{ and } \forall \, x\in \B_R(0) \,,
\end{equation}
and
\begin{equation}\label{hsohoh2}
R^{|\alpha|}\bigg|\frac{\partial^{|\alpha|} g_{ij} (x)}{\partial x^{\alpha}}\bigg|\le \frac{1}{100} \s\s\forall \alpha \mbox{ multi-index  with }1\le|\alpha|\le \ell, \mbox{ and }\forall   x\in \B_R(0).
\end{equation}
\end{definition}

\begin{remark}\label{fbsvdg}
For any smooth closed Riemannian manifold $(M,g)$, given $\ell \ge 0$, there exists $R_0>0$ for which ${\rm FA}_\ell(M,g,R_0,p,\varphi_p)$ is satisfied for all $p\in M$. For example, $\varphi_p$ can be chosen to be the restriction of the exponential map (of $M$) at $p$ to the (normal) ball $\B_{R_0}(0)\subset T_p M \cong \R^n$.
\end{remark}

\begin{remark}
    As in \cite{CFS}, this notion of local flatness is used in our results to stress the fact that, once the local geometry of the manifold is controlled in the sense of Definition \ref{flatnessassup}, then our estimates are independent of
    $M$. Interestingly, this makes our estimates of local nature even though the objects we consider are nonlocal. 
\end{remark}

\begin{remark}\label{flatscalingrmk}
Throughout the paper the following scaling properties will be used several times.
\begin{itemize}
\item[(a)] Given $M = (M,g)$ and $r>0$, we can consider the "rescaled manifold" $\widehat M = (M, \frac{1}{r^2}g)$. When performing this rescaling, the new heat kernel $H_{\widehat {M}}$ satisfies 
\begin{equation*}
    H_{\widehat{M}}(p,q,t) = r^{n} H_M(p,q,r^2 t)  \,.
\end{equation*}
As a consequence,  the "rescaled kernel" $\widehat  K_s$ defining the $s$-perimeter on $\widehat M$ satisfies 
\[
\widehat  K_s(p,q) = r^{n+s} K_s(p,q).
\]

\item[(b)] Concerning the flatness assumption, it is  easy to show that ${\rm FA}_\ell(M,g, R, p,\varphi) \Rightarrow {\rm FA}_\ell(M,g, R',p,\varphi)$ for all $R'<R$ and  ${\rm FA}_\ell(M,g, R, p,\varphi)\Leftrightarrow {\rm FA}_\ell(M, \frac{1}{r^2}g, R/r, p,\varphi(r\,\cdot\,))$.

\item[(c)] Similarly, if ${\rm FA}_\ell(M,g, R, p,\varphi)$ holds,  and $q\in \varphi(\B_R(0))$ is such that $\B_\varrho(\varphi^{-1}(q))\subset \B_R(0)$, then 
${\rm FA}_\ell(M,\frac{1}{r^2}g, \varrho/r, q,\varphi_{\varphi^{-1}(q),r})$ holds, where $\varphi_{x,\,r} := \varphi(x+r \, \cdot \,)$. 
\end{itemize}
\end{remark}

We now give precise estimates for the kernel $K_s(p,q)$ which we will need in the article, taken from \cite{FracSobPaper}.
\begin{proposition}[{\cite{FracSobPaper}}] \label{prop:kern1}
Let $(M,g)$ be a Riemannian $n$-manifold, not necessarily closed, $s\in(0,2)$ and let $p\in M$. Assume ${\rm FA_\ell}(M,g,R,p, \varphi)$ holds, where $l\geq 1$, and denote $K(x,y): = K_s(\varphi(x), \varphi(y))$. 

\vsp
Given $x\in \B_R(0)$, let
$A(x)$ denote the positive symmetric square root of the matrix $(g_{ij}(x))$ ---$g_{ij}$ being the metric in coordinates $\varphi^{-1}$--- and, for $x,z\in \B_{R/2}(0)$, define
\begin{equation*}
    k(x, z): = K(x,x+z) \quad \mbox{and}\quad  \widehat k(x,z) := k(x,z) - \alpha_{n,s} \frac{{\rm det}\big(A(x)\big) }{|A(x)z|^{n+s}}.
\end{equation*}
Then
\begin{equation}\label{remaining0}
\big|\widehat k(x,z)\big| \le  R^{-1}\frac{C(n,s)}{|z|^{n+s -1}}\quad \mbox{for all } x,z \in \B_{R/4}(0) \,,
\end{equation}
and, for every multi-indices $\alpha,\beta$  with $|\alpha|+|\beta| \le \ell$,
we have
\begin{equation}\label{remaining}
  \bigg|\frac{\partial^{|\alpha|}}{\partial x^{\alpha}} \frac{\partial^{|\beta|}}{\partial z^{\beta}}  k(x,z)\bigg| \le \frac{C(n,s, \ell)}{|z|^{n+s+ |\beta|}}\quad \mbox{for all } x,z \in \B_{R/4}(0) .  
\end{equation}

Moreover, for all $x\in \B_{R/4}(0)$ and for all $q\in M\setminus \varphi(\B_R(0))$ we have
\begin{equation}\label{remaining2}
\bigg|\frac{\partial^{|\alpha|}}{\partial x^{\alpha}}  K_s(\varphi(x),q)\bigg| \le \frac{C(n,\ell)}{R^{n+s}} \,,
\end{equation}
and 
\begin{equation}\label{remaining3}
\int_{M\setminus \varphi(\B_R(0))}\bigg|\frac{\partial^{|\alpha|}}{\partial x^{\alpha}}  K_s(\varphi(x),q)\bigg| dV_q \le \frac{C(n,\ell)}{R^{s}}, 
\end{equation}
for every multi-index  $\alpha$ with $|\alpha|\le \ell$.\\
Here, the constants $c(n,s)$ depending on $s$ are uniform for $s$ away from $0$ and $2$.
\end{proposition}

\begin{lemma}\label{loccomparability}
    Let $s_0 \in (0,2)$ and $s \in (s_0,2)$. Let $(M,g)$ be a Riemannian $n$-manifold and $ p \in M$. Assume that ${\rm FA}_1(M,g,p,R,\varphi)$ holds. Then 
\begin{equation*}
   c_7 \frac{\alpha_{n,s}}{|x-y|^{n+s}} \le  K_s(\varphi(x),\varphi(y)) \le  c_8 \frac{\alpha_{n,s}}{|x-y|^{n+s}} \,, 
\end{equation*}
for all $x,y \in \B_{R/2}(0)$, where $c_7,c_8>0$ depend on $n$ and $s_0$.
\end{lemma}

\begin{proposition}[\cite{FracSobPaper}]\label{sdfgsdrgs}
        Let $(M,g)$ be a closed $n$-dimensional Riemannian manifold and $s\in(0,2)$. Assume that the flatness assumption ${\rm FA}_\ell (M,g,R,p,\varphi)$ holds, and let $X\in \mathfrak{X}( M)$ be a smooth vector field supported on $\varphi(\B_{R/4})$. Writing $\psi_X^t$ for the flow of $X$ at time $t$, then for every $x,y \in \B_{R/4}(0)$ we have
        \begin{equation}\label{cdsacsdc1}
    \bigg| \frac{d^\ell}{dt^\ell}\bigg|_{t=0} K_s(\psi_X^t(\varphi(x)),\psi_X^t(\varphi(y))) \bigg|\leq C K_s(\varphi(x), \varphi(y))\,,
\end{equation}
for some constant $C=C(n,s,l, \|X \|_{C^\ell(\varphi(\B_R))})$ which stays bounded for $s$ away from $0$ and $2$.
    \end{proposition}
\begin{lemma}[\cite{FracSobPaper}]\label{enboundslemma2} 
Let $s\in(0,2)$. Let $M$ satisfy the flatness assumptions ${\rm FA}_\ell(M,g,R,p,\varphi)$, and let $ v=\chi_E-\chi_{E^c}\in H^{s/2}(M)$. Let $X \in \mathfrak{X} (M)$ be a smooth vector field supported on $\varphi(\B_{R/2})$, and put $v_t := v \circ \psi_X^{-t}$, where $\psi_X^t$ is the flow of $X$ at time $t$. Then, for all $T>0$ there holds
\begin{equation*}
    \sup_{0< t<T} \left| \frac{d^\ell}{d  t^\ell}\textnormal{Per}_s(\psi_X^{t}(E)) \right|\leq C(1+\textnormal{Per}_s(E;\varphi(\B_{R})))\, ,
\end{equation*}
for some constant $C=C(n,s,l,T, \|X \|_{C^\ell(\varphi(\B_R))})$ which stays bounded for $s$ away from $0$ and $2$.
\end{lemma}

\begin{proof}
This is Lemma 2.17 in \cite{FracSobPaper}. We give a sketch of its proof, since we will use similar ideas later. One starts by changing variables with the flow of $X$, to see that
\begin{align*}
\frac{d^\ell}{d t^\ell}\textnormal{Per}_s(\psi_X^{t}(E))&=\frac{d^\ell}{d t^\ell}\iint |v(\psi_X^{- t}(p))-v(\psi_X^{- t}(q))|^2K_s(p,q) \, dV_p\, dV_q\\
&=\frac{d^\ell}{d t^\ell}\iint |v(p)-v(q)|^2 K_s(\psi_X^{ t}(p),\psi_X^{ t}(q))\,J_{ t}(p)J_{ t}(q) \, dV_p\, dV_q\\
&=\iint |v(p)-v(q)|^2\frac{d^\ell}{d t^\ell}\Big[K_s(\psi_X^{ t}(p),\psi_X^{ t}(q)) \,J_{ t}(p)J_{ t}(q)\Big]  dV_p\, dV_q\,.
\end{align*}
The Jacobians $J_t(p)$ and their derivatives can simply be bounded by a constant with the right dependencies. The kernel part and its derivatives can be estimated appropriately using Proposition \ref{sdfgsdrgs} for $p$ and $q$ close and \eqref{remaining3} for $p$ and $q$ separated enough.
\end{proof}

We conclude this subsection with two simple interpolation results, which in particular imply that sets of finite (classical) perimeter have bounded fractional perimeter as well.
\begin{lemma} \label{lem:interpropRn} Let $M$ satisfy the flatness assumptions ${\rm FA}_1(N,g,p_0,4R,\varphi)$. Let $s_0\in(0,1)$ and $s\in(s_0,1)$, and let $u:\varphi(\B_R)\to\R$ be a function of bounded variation (BV). Then,
\begin{equation}\label{ipr11}
    \iint_{\varphi(\B_R)\times \varphi(\B_R)} |u(p)-u(q)|^2 K_s(p,q)\,dV_p\,dV_q \leq \frac{C(n,s_0)}{1-s}[u]_{BV(\varphi(\B_R))}^s\|u\|_{L^1(\varphi(\B_R))}^{1-s}\, .
\end{equation}
    In particular, if $E\subset M$ is a set of finite perimeter in $\varphi(\B_R)$, then
\begin{equation}\label{interpropRn}
    R^{s-n}{\rm Per}_s(E;\varphi(\B_R)) \leq \frac{C(n,s_0)}{1-s} [R^{1-n}{\rm Per}(E;\varphi(\B_R))]^s\, .
\end{equation}
\end{lemma}
\begin{proof}
If $v:\B_R\subset\R^n\to\R$ is a function of bounded variation, using the fundamental theorem of calculus it is simple to see that
\begin{equation*}
    \iint_{\B_R\times \B_R} \frac{|v(x)-v(y)|^2}{|x-y|^{n+s}}\,dx\,dy \leq \frac{C(n)}{(1-s)s}[v]_{BV(\B_R)}^s\|v\|_{L^1(\B_R)}^{1-s}\, ,
\end{equation*}
see for instance \cite[Proposition 4.2]{Brasc}. The bound \eqref{ipr11} then follows from applying this result to $v=u\circ\varphi$ and using Lemma \ref{loccomparability} to compare the corresponding kernels.
\end{proof}

We can obtain a global version as well:
\begin{proposition} Let $(M,g)$ be a closed Riemannian manifold. Let $s_0\in(0,1)$ and $s\in(s_0,1)$, and let $E\subset M$ be a set of finite perimeter. Then,
\begin{equation}\label{interpropM}
    {\rm Per}_s(E;M) \leq \frac{C(M,s_0)}{1-s} [{\rm Per}(E;M)]^s|E|^{1-s} .
\end{equation}
\end{proposition}
\begin{proof}
\textbf{Step 1.} Given $u:M\to\R$ in $BV(M)$, we show that
\begin{equation}\label{itrop1}
    \iint_{M\times M} |u(p)-u(q)|^2 K_s\,dV_p\,dV_q \leq \frac{C(M,s_0)}{1-s}[u]_{BV(M)}^s\|u\|_{L^1(M)}^{1-s}+C(M)\|u\|_{L^1(M)}\, .
\end{equation}
First, note that since $M$ is closed, there exists a number $\delta>0$ with the property that, given any $p\in M$, the flatness assumptions ${\rm FA}_1(M,g,4\delta,p,\varphi_p)$ are satisfied by some $\varphi_p$; see Remark \ref{fbsvdg}. Having fixed such a $\delta>0$, by compactness we can find a finite collection of points $p_1,...,p_N$ such that the balls $\varphi_{p_1}(\B_\delta(p_1)),...,\varphi_{p_N}(\B_\delta(p_N))$ cover $M$.\\
Using \eqref{remaining2} and \eqref{ipr11}, we can then bound
\begin{align*}
    \iint_{\varphi_{p_i}(\B_{\delta}(p_i))\times M} |u(p)-u(q)|^2 K_s\,dV_p\,dV_q&\leq \iint_{\varphi_{p_i}(\B_{2\delta}(p_i))\times \varphi_{p_i}(\B_{2\delta}(p_i))} |u(p)-u(q)|^2 K_s\,dV_p\,dV_q+C(M)\|u\|_{L^1(M)}\\
    &\leq \frac{C(M,s_0)}{1-s}[u]_{BV(\varphi_{p_i}(\B_{2\delta}(p_i)))}^s\|u\|_{L^1(\varphi_{p_i}(\B_{2\delta}(p_i)))}^{1-s}+C(M)\|u\|_{L^1(M)}\\
    &\leq \frac{C(M,s_0)}{1-s}[u]_{BV(M)}^s\|u\|_{L^1(M)}^{1-s}+C(M)\|u\|_{L^1(M)}\,,
\end{align*}
which adding for $i=1,...,N$ shows \eqref{itrop1}.\\

\textbf{Step 2.} We prove that
\begin{equation}\label{itrop2}
    \iint_{M\times M} |u(p)-u(q)|^2 K_s\,dV_p\,dV_q\leq \frac{C(M,s_0)}{1-s}[u]_{BV(M)}^s\|u\|_{L^1(M)}^{1-s}\,,
\end{equation}
an inequality which is interesting on its own. Considering $u=\chi_E-\chi_{E^c}$ then yields the Proposition.

The main observation is that the LHS of \eqref{itrop1} is invariant under addition of a constant, while the RHS is not. This suggests to consider the function $v:=u-\frac{1}{{\rm Vol}(M)}\int_M u$ instead of $u$ in the inequality obtained in \eqref{itrop1}. Combining this with the $L^1$--Poincar\'e inequality on $M$ allows us to estimate
\begin{align*}
    \iint_{M\times M} &|u(p)-u(q)|^2 K_s\,dV_p\,dV_q=\iint_{M\times M} |v(p)-v(q)|^2 K_s\,dV_p\,dV_q\leq \frac{C(M,s_0)}{1-s}[v]_{BV(M)}^s\|v\|_{L^1(M)}^{1-s}+C(M)\|v\|_{L^1(M)}\\
    &= \frac{C(M,s_0)}{1-s}[u]_{BV(M)}^s\|v\|_{L^1(M)}^{1-s}+C(M)\|v\|_{L^1(M)}^s\|v\|_{L^1(M)}^{1-s}\\
    &\leq \frac{C(M,s_0)}{1-s}[u]_{BV(M)}^s\|v\|_{L^1(M)}^{1-s}+C(M)[u]_{BV(M)}^s\|v\|_{L^1(M)}^{1-s}\leq \frac{C(M,s_0)}{1-s}[u]_{BV(M)}^s\|v\|_{L^1(M)}^{1-s}\,.
\end{align*}
Bounding $\|v\|_{L^1(M)}\leq 2\|u\|_{L^1(M)}$ leads then to the bound \eqref{itrop2}, concluding the proof.
\end{proof}

\subsubsection{Monotonicity formula}
The monotonicity formula for $s$-minimal surfaces requires using the Caffarelli-Silvestre extension in one extra dimension, which we now recall.
\begin{theorem}[{\cite[Theorem 2.25 and Proposition 3.2]{FracSobPaper}}]\label{extmfd}
Let $(M^n,g)$ be a closed Riemannian manifold, let $s\in (0,2)$ and $u:M\to\R$ be in $H^{s/2}(M)$. Consider the product manifold $\widetilde{M}=M\times (0,+\infty)$ endowed with the natural product metric\footnote{That is, the metric defined by $\widetilde g \big((\xi_1, z_1), (\xi_2,z_2)\big)= g(\xi_1, \xi_2) + z_1z_2$, and where $\widetilde{\rm div}$ and $\widetilde \nabla$ denote the divergence and Riemannian gradient with respect to this product metric respectively.}. Then, there is a unique solution $U:M \times (0, \infty) \to \R $ among functions with $\int_{\widetilde{M}} |\widetilde \nabla U|^2 z^{1-s} \, dVdz<\infty$ to
\begin{equation}\label{caffextMfd}
    \begin{cases} \widetilde{ {\rm div}}(z^{1-s} \widetilde \nabla U) = 0  &   \mbox{in} \s \widetilde{M} \,,  \\ U(p,0) = u(p) &  \mbox{for} \s p \in \partial \widetilde{M} = M\, , \end{cases}
\end{equation}
and it satisfies
\begin{align}
    [u]^2_{H^{{s/2}}(M)} &= 2\beta_s \int_{\widetilde{M}} |\widetilde \nabla U|^2 z^{1-s} \, dVdz \label{eneqal}\\
    &= \inf\left\{ 2\beta_s \int_{\widetilde{M}} |\widetilde \nabla V|^2 z^{1-s} \, dVdz \s {\mbox{s.t.}} \s V(x, 0)=u(x) \right\}\,, \label{eneqal2}
\end{align}
where 
\begin{equation}\label{betadef}
    \beta_s = \frac{2^{s-1} \Gamma(s/2)}{\Gamma(1-s/2)}\, .
\end{equation}
In \eqref{caffextMfd} and \eqref{eneqal2}, the boundary condition is to be interpreted in the trace sense. 
\end{theorem}

We introduce some notational conventions for balls. For $r>0$ and $p \in M $, we will denote
\begin{equation}\label{notationballs}
\begin{aligned}
     B_r(p) & = \big\{ q \in M \, : \, d_g(q, p) < r \big\} \,, \\[0.5mm]
     \ov{B}_{r}^+(p, 0) & = \big\{ (q,z) \in \ov{M} \, : \, d_{\ov{g}}((q,z), (p, 0) ) < r \big\} \,,  \\
    \partial \ov{B}_{r}^+(p, 0) & = \partial \left( \ov{B}_{r}^+(p, 0) \right) \,, \\ \partial^+ \ov{B}_{r}^+(p, 0) & = \partial \ov{B}_{r}^+(p, 0) \cap \{ z>0 \} \,.
\end{aligned} 
\end{equation}
In the next theorem, we use $\nabla$ instead of $\widetilde \nabla$ to denote the gradient in  $\widetilde M$ with respect to the product metric.

\begin{theorem}[\textbf{Monotonicity formula} {\cite[Theorem 3.4]{FracSobPaper}}]\label{monfor}
Let $(M^n, g)$ be an $n$-dimensional, closed Riemannian manifold. Let $s \in (0, 1)$ and let $\Sigma$ be an $s$-minimal surface. For $(p_\circ,0) \in \ov{M} $ define 
\begin{equation*}
    \Phi(R) := \frac{1}{R^{n-s}} \left( \frac{\beta_s}{2} \int_{\widetilde{B}^+_R(p_\circ,0)} z^{1-s}| \nabla  U(p,z)|^2 \, dV_p dz \right) , 
    \end{equation*}
where $U$ is the unique solution given by Theorem \ref{extmfd} with $u:M\to\R$ given by $u=\chi_\Sigma-\chi_{\Sigma^c}$. Then, there exist constants $C=C(n)$ and $R_{\rm max}=R_{\rm max}(M, p_\circ)>0$ with the following property: whenever $R_\circ \le \inj_{M}(p_\circ)/4$ and $K$ is an upper bound for all the sectional curvatures of $M$ in $B_{R_\circ}(p_\circ)$, then
\begin{equation*}
    R \mapsto \Phi(R)e^{C \sqrt{K} R } \s \textit{is nondecreasing for} \s R < R_\circ \,.
\end{equation*} 
\end{theorem}
From the proper monotonicity formula in Theorem \ref{monfor}, we now obtain a weaker form which involves the fractional perimeter on balls of $M$ (instead of $\widetilde M$). This will be repeatedly (and crucially) used in what follows, since it gives the decay of the fractional perimeter with respect to the radius with the right rate. We note here that the proof can be slightly modified so that it recovers the monotonicity formula for classical minimal surfaces in the limit $s\to 1$.
\begin{lemma}\label{PerMon}
Let $E\subset M$ be an $s$-minimal surface in $\varphi(\B'_1\times[-1,1])$, and assume that ${\rm FA}_1(M,g,4,p_0,\varphi)$ holds.
Then there exists a constant $C$ depending only on $n$ such that, for any $0<r_1\leq r_2\leq 1$, we have that
    \begin{equation*}
        \frac{(1-s)\textnormal{Per}_s(E;B_{r_1}(p_0))}{r_1^{n-s}}\leq (1-s)C+C\frac{(1-s)\textnormal{Per}_s(E;B_{r_2}(p_0))}{r_2^{n-s}}\, .
    \end{equation*}
    In particular, setting $r_2=1$, for all $0<r_1<\frac{1}{8}$ we have that
    \begin{align}
        \textnormal{Per}_s(E;B_{r_1}(p_0))&\leq Cr_1^{n-s}(1+\textnormal{Per}_s(E;B_{1}(p_0)))\label{fpergrowth1}\\
        &\leq Cr_1^{n-s}(1+\textnormal{Per}_s(E;M))\label{fpergrowth}\, .
    \end{align}
\end{lemma}
\begin{proof}
    Let $\widetilde\varphi(x,z)=(\varphi(x),z)$ be the natural local parametrisation of $\widetilde M$. If $8r_1>r_2$, letting $C=8^n$ the Lemma is satisfied, therefore we can assume that $8r_1\leq \frac{1}{8}r_2$.
    
    \noindent \textbf{Step 1.} We show that, given $u:M\to [-1,1]$, we have the inequality
    $$\iint_{(M\times M)\setminus (\varphi(\B_r)\times \varphi(\B_r))} (u(p)-u(q))^2K(p,q) \leq C\Big(r^{n-s}+\int_{\widetilde \varphi(\widetilde \B_{4r}^+)} z^a|\nabla U|^2\Big)\,.$$
    An easy application of \eqref{remaining3} from Proposition \ref{prop:kern1} with $\alpha=0$, considering the rescaled manifold $(M,\frac{1}{r^2}g)$ and (a) in Remark \ref{flatscalingrmk}, allows us to bound
    \begin{align*}
        &\iint_{(M\times M)\setminus (\varphi(\B_r^c)\times \varphi(\B_r)^c)} (u(p)-u(q))^2K(p,q)\leq\iint_{\varphi(\B_{2r})\times \varphi(\B_{2r})} (u(p)-u(q))^2K(p,q) +8\iint_{\varphi(\B_r)\times (M\setminus \varphi(\B_{2r})^c)} K(p,q)\\
    &\hspace{1cm}\leq \iint_{\varphi(\B_{2r})\times \varphi(\B_{2r})} (u(p)-u(q))^2K(p,q) +C\int_{\varphi(\B_r)} r^{-s}\leq \iint_{\varphi(\B_{2r})\times \varphi(\B_{2r})} (u(p)-u(q))^2K(p,q)+Cr^{n-s}\,.
    \end{align*}
    We will now continue by proceeding as in \cite[Proposition 7.1]{CRS}. Let $\xi$ be a standard cutoff on $\widetilde\B_4^+\subset\R^{n+1}$, verifying that $\chi_{\widetilde\B_2^+\cap \{z=0\}}\leq \xi\leq \chi_{\widetilde\B_4^+}$. Define then $\widetilde \eta_r =\xi(\cdot/r)\circ\widetilde\varphi^{-1}$, extended by zero outside of $\widetilde\varphi(\widetilde \B_{4r}^+)\subset \widetilde M$. Moreover, set $\eta_r(p)=\widetilde\eta_r(p,0):M\to \R$. Then, $\eta_r(p)$ acts as a cutoff function on $\varphi(\B_{4r})\subset M$ which verifies that $\chi_{\varphi(\B_{2r})}\leq \eta_r\leq \chi_{\varphi(\B_{4r})}$. We can then compute
    \begin{align*}
    &\iint_{(M\times M)\setminus (\varphi(\B_r)^c\times \varphi(\B_r)^c)} (u(p)-u(q))^2K(p,q)\leq \iint_{\varphi(\B_{2r})\times \varphi(\B_{2r})} (u(p)-u(q))^2K(p,q)+Cr^{n-s}\\
    &\qquad\leq \iint_{M\times M} ((u\eta_r)(p)-(u\eta_r)(q))^2K(p,q)+Cr^{n-s}=[u\eta_r]_{H^{s/2}(M)}^2+Cr^{n-s}\leq C\Big(\int_{\widetilde M} z^{1-s}|\nabla (U\widetilde\eta_r)|^2+r^{n-s}\Big)\,.
    \end{align*}
    We have applied \eqref{eneqal2} in the last line, using that $U\widetilde\eta_r(p,0)=u\eta_r(p)$ for all $p\in M$. We can now bound
    \begin{align*}
    &\iint_{(M\times M)\setminus (\varphi(\B_r)\times \varphi(\B_r))} (u(p)-u(q))^2K(p,q)\leq C\Big(\int_{\widetilde\varphi(\widetilde \B_{4r}^+)} z^{1-s}|\nabla U|^2\widetilde\eta_r^2+\int_{\widetilde\varphi(\widetilde \B_{4r}^+)} z^{1-s}U^2|\nabla\widetilde\eta_r|^2+r^{n-s}\Big)\\
    &\hspace{1cm}\leq C\Big(\int_{\widetilde\varphi(\widetilde \B_{4r}^+)} z^{1-s}|\nabla U|^2+r^{-2}\int_{\widetilde\varphi(\widetilde \B_{4r}^+)} z^{1-s}+r^{n-s}\Big)\leq C\Big(\int_{\widetilde\varphi(\widetilde \B_{4r}^+)} z^{1-s}|\nabla U|^2+r^{n-s}\Big) \, .
    \end{align*}
    \textbf{Step 2.} Applying Step 1 with $u=\chi_E-\chi_{E^c}$ and $r=r_1$, we find that
    \begin{equation*}
        \frac{\textnormal{Per}_s(E;\varphi(\B_{r_1}))}{r_1^{n-s}}\leq C+\frac{C}{r_1^{n-s}}\int_{\widetilde\varphi(\widetilde \B_{4r_1}^+)} z^{1-s}|\nabla U|^2\, .
    \end{equation*}
    Applying then the monotonicity formula of Theorem \ref{monfor}, recalling that we are assuming that $4r_1<\frac{1}{2}r_2$, we have that
    \begin{align*}
        \frac{\textnormal{Per}_s(E;\varphi(\B_{r_1}))}{r_1^{n-s}}&\leq C+\frac{C}{r_2^{n-s}}\int_{\widetilde\varphi(\widetilde \B_{\frac{1}{2}r_2}^+)} z^{1-s}|\nabla U|^2\, .
    \end{align*}
    We now need essentially the converse of Step 1 to bound the RHS. This is \cite[Lemma 3.15]{FracSobPaper}, or more precisely the chain of inequalities at the end of its proof, and it allows us to conclude that
    \begin{align*}
        \frac{\textnormal{Per}_s(E;\varphi(\B_{r_1}))}{r_1^{n-s}}&\leq C(1+\frac{1}{r_2^{n-s}}\iint_{\varphi( \B_{r_2})\times \varphi( \B_{r_2})} (u(p)-u(q))^2K(p,q))\leq C+C\frac{\textnormal{Per}_s(E;\varphi(\B_{r_2}))}{r_2^{n-s}}
    \end{align*}
    as desired.
\end{proof}

\subsection{Min-max procedure -- The nonlocal volume spectrum}\label{sec:minmax}
Critical points of \eqref{ACenergy} can be obtained using an equivariant min-max procedure, exploiting the $\mathbb{Z}_2$-symmetry of the functional $\mathcal{E}_{\ep,s}$. In \cite{CFS}, a ``nonlocal spectrum" is defined employing as a min-max family of sets with $\p$ parameters the one consisting of odd images of the sphere $\Sp^\p$ into $H^{s/2}(M)$. This family has a very simple definition and served for the purposes of that article, but others could have been employed in an essentially exchangeable way, including families which come from a topological index; see \cite[Remark 3.1]{CFS} where this is briefly discussed. To obtain a Weyl Law, one wants to restrict instead to families which arise from some topological index.\\
Begin by defining
    \begin{equation*}
        \text{ind}(A)=\sup\{k: \nexists\text{ any continuous odd map }g:A\to \mathbb{R}^{k-1}\backslash\{0\}\}
    \end{equation*}
    if $A\subset H^{s/2}(M)$ is compact, symmetric (i.e. $A=-A$), and does not contain zero. Here odd means that $g(-x)=-g(x)$. Set moreover $\text{ind}(A)=\infty$ if $A$ contains $0$, and set $\text{ind}(\emptyset)=0$.\\
    It is standard that this notion of index satisfies all the axioms of a topological index, including the one that we will need and which is its subadditivity:
    \begin{lemma}\label{indsub} $\textnormal{ind}(A_1\cup A_2)\leq \textnormal{ind}(A_1)+\textnormal{ind}(A_2)$
    \end{lemma}
    \begin{proof}
        We can assume that $\max\{\textnormal{ind}(A_1),\textnormal{ind}(A_2)\}<\infty$.\\
        By the definition of $\textnormal{ind}$, there exist then odd maps $g_1:A_1\to\mathbb{R}^{\textnormal{ind}(A_1)}\backslash\{0\}$ and $g_2:A_2\to\mathbb{R}^{\textnormal{ind}(A_2)}\backslash\{0\}$. Extend them to maps $\tilde{g}_1$ and $\tilde{g}_2$ defined on all of $H^{s/2}(M)\backslash\{0\}$ by the Tietze extension theorem. After antisymmetrizing, we can assume that they are also antisymmetric, and they still coincide with the original maps on $A_1$ and $A_2$.\\
        The map $(\tilde g_1, \tilde g_2)$ restricted to $A:=A_1\cup A_2$ is then an odd map from $A$ into $\mathbb{R}^{\textnormal{ind}(A_1)+\textnormal{ind}(A_2)}\backslash\{0\}$, since either $\tilde g_1$ or $\tilde g_2$ is always nonzero on $A$. This means that $\textnormal{ind}(A)\leq \textnormal{ind}(A_1)+\textnormal{ind}(A_2)$, which concludes the proof.
    \end{proof}
    
    The family of sets to which the min-max procedure will be applied is
    \begin{equation}\label{cohom}
        \mathcal{F}_\p=\{A\subset H^{s/2}(M):A\text{ is compact, symmetric, satisfies } \mathcal{H}^{\p}(A)<\infty \text{, and has ind}(A)\geq \p+1\}\, ,
    \end{equation}
    where $\mathcal{H}^\p$ denotes the Hausdorff measure of dimension $\p$.\\
For fixed $\ep$, the min-max value of the family $\mathcal{F}_\p$ is defined as
\begin{equation}\label{minmaxval}
     c_{\ep,s}(\p,M):= \inf_{A\in\mathcal{F}_\p} \sup_{u\in A}\mathcal E_{\ep,s} (u,M).
\end{equation}
Note that, defining $T(u):=\max\{-1,\min\{u,+1\}\}$ the truncation of $u$ between the values $\pm 1$, we have that $|T(u)|(x)\le 1$ for all $x\in M$ and $\mathcal E_{\ep,s} (T(u),M) \le \mathcal E_{\ep,s} (u,M)$. Moreover, since $T$ is odd, is it immediate to see ${\rm ind}(T(A))\geq {\rm ind}(A)$. Hence
\begin{equation*}
     c_{\ep,s}(\p,M) = \inf_{A\in\mathcal{F}_\p} \sup_{u\in A}\mathcal E_{\ep,s} (u,M) =  \inf_{A\in \widetilde{\mathcal{F}}_\p} \sup_{u\in A}\mathcal E_{\ep,s} (u,M) \,,
\end{equation*}
where 
\begin{equation*}
    \widetilde{\mathcal{F}}_\p = \{A \in \mathcal{F}_\p : |u|\le 1 \s \textnormal{for all} \s u \in A \} .
\end{equation*}
The min-max values just defined are attained by critical points of the Allen-Cahn equation, and satisfy lower and upper bounds depending on $\p$:

\begin{theorem}[\textbf{Existence of min-max solutions}]\label{mmbounds}
Let $(M^n,g)$ be a compact Riemannian manifold. Fix $s_0\in(0,1)$, and let $s\in(s_0,1)$. Then, for every $\p \in \N$ there exists $\ep_\p>0$ such that, for all $\ep\in(0,\ep_\p)$, the min-max values of \eqref{minmaxval} satisfy
\begin{equation}\label{critlevbound}
    C^{-1} \p^{s/n} \le (1-s)c_{\ep,s}(\p,M) \le C\p^{s/n} \,, \s \textnormal{for all} \,\, \ep \in (0,\ep_\p) \,,
\end{equation}
and moreover there exists $u_{\ep,\p} \in H^{s/2}(M)$ critical point of $\mathcal E_{\ep,s}$ with $\mathcal E_{\ep,s} (u_{\ep,\p},M) =c_{\ep,s}(\p,M)$ and Morse index $m(u_{\ep,\p}) \le \p $. Here $C$ depends only on $M$ and $s_0$.
\end{theorem}
\begin{proof}
    We explain the slight modifications needed with respect to \cite[Section 3.1]{CFS} in Appendix \ref{app:minmaxthm}.
\end{proof}
\begin{remark}
    The min-max families employed also give lower bounds for the extended Morse index of the $u_{\ep,\p}$, i.e. the Morse index plus the nullity of the second variation of the Allen-Cahn energy. We will not use this in this work.
\end{remark}
We now consider the limit as $\ep\to 0$ for the min-max quantities. Define the ``nonlocal volume widths"
\begin{equation}\label{deflps}
    l_{s}(\p,M):=\lim_{\ep\to 0}c_{\ep,s}(\p,M)\, .
\end{equation}
The existence of the limit in $\ep$ is obvious, as the Allen-Cahn energy of any function $u$ is monotone nondecreasing in $\ep$, and hence so is the value of $c_{\ep,s}(\p,M)$. Furthermore, by the lower and upper bounds from Theorem \ref{mmbounds} we know that 
\begin{equation}\label{lpsbounds}
    \frac{1}{C}\p^{s/n} \leq (1-s)l_{s}(\p,M) \leq C\p^{s/n}.
\end{equation}
In \cite{CFS} it is proved that critical points of $\mathcal{E}_{\ep,s}$ with bounded Morse index converge, as $\ep\to 0$, to $s$-minimal surfaces. A partial statement is the following:
\begin{theorem}[\textbf{Convergence as $\ep\to 0$,} \cite{CFS}]
\label{StrongConv}
Fix $s \in (0,1)$. Let $u_{\ep_j}$ be a sequence of solutions of (\ref{restrictedeq}) on $M$ with parameters $\ep_j\to 0$ and Morse index at most $m$. Then, there exist a subsequence, still denoted by $u_{\ep_j}$, and a nonlocal $s$-minimal surface $E\subset M$ with Morse index at most $m$ (in the sense of Definition \ref{WeakMorseDef}), such that
\begin{equation*}
    u_{\ep_j} \conv{H^{s/2}} u_0=\chi_E-\chi_{E^c}\, .
\end{equation*}
In particular $\mathcal E_{\ep,s}^{\rm Sob} (u_{\ep_j},M)\to \textnormal{Per}_s(E) = \mathcal E_{\ep,s}^{\rm Sob} (u_{0},M)$. Moreover, $\mathcal E_{\ep,s}^{\rm Pot} (u_{\ep_j},M)\to 0=\mathcal E_{\ep,s}^{\rm Pot} (u_{0},M)$.

In addition, up to changing $E$ in a set of measure zero, we have
\begin{eqnarray}\label{wngioewhtioeh1}
E &=& \Big\{p \in M \ : \liminf_{r\downarrow 0} \tfrac{|E \cap B_r(p) |}{|B_r(p)|} =1 \Big\} \,,\\
M\setminus E &=& \Big\{p \in M \ : \limsup_{r\downarrow 0} \tfrac{|E \cap B_r(p) |}{|B_r(p)|} =0 \Big\} \,, \\
\partial E &= & \Big\{p \in M \ :  \tfrac{|E \cap B_r(p) |}{|B_r(p)|} \in [c, 1-c]\quad \forall r\in (0,r_\circ)\Big\}. 
\end{eqnarray}
Finally, there is a universal constant $s_0\in(0,1)$ such that if $s\in(s_0,1)$, then $\partial E$ is a smooth hypersurface outside from a set of Hausdorff dimension at most $n-5$.
\end{theorem}
Combined with Theorem \ref{mmbounds}, this shows:
\begin{corollary}[\textbf{Attainability of nonlocal widths}]\label{AttNloc}
Let $(M^n,g)$ be a compact Riemannian manifold. For every $\p\in\N$, there exists an $s$-minimal surface $E_{\p}^s\subset M$ such that ${\rm Per}_s(E_{\p}^s)=l_{s}(\p,M)$. Moreover, there is a universal constant $s_0\in(0,1)$ such that if $s\in(s_0,1)$, then $\partial E_{\p}^s$ is a smooth hypersurface outside from a set of Hausdorff dimension at most $n-5$.    
\end{corollary}
This gives Theorem \ref{FracYau3}, with $\text{Per}_s(E_{\p}^s)=l_{s}(\p,M)$ instead of the different min-max values used in \cite{CFS}.

\section{The Weyl Law}

\subsection{The Weyl Law with the distance kernel}

\subsubsection{The distance kernel}
The Allen-Cahn energies and fractional perimeters that we have (canonically) defined on a manifold $M$ depend on a non-explicit kernel $K_{M,s}$. Since the proof of the Weyl Law will be based on comparison between a manifold and its subdomains, as well as between such subdomains and Euclidean subdomains,
we make the technical choice of first proving the Weyl Law for a different (non-canonical) definition of fractional perimeter and Allen-Cahn energy:\\

Let $M$ be a closed Riemannian manifold or $\R^n$, and let $\Omega$ be an open subset of $M$. Define the space $H^{s/2,d}(\Omega)$ by setting
\begin{equation}\label{wethiowhoihwd}
 [u]^2_{H^{s/2,d}(\Omega)} := \iint_{\Omega\times \Omega}(u(p)-u(q))^2 \frac{\alpha_{n,s}}{\text{dist}_M^{n+s}(p,q)} \,dV_p \,d V_q\, ,
 \end{equation}
 and define
 \begin{equation}\label{ACenergyd}
 \mathcal E_{\ep,s}^d (v,\Omega) := \frac 1 4 \iint_{\Omega\times\Omega} (v(p)-v(q))^2 \frac{\alpha_{n,s}}{\text{dist}_M^{n+s}(p,q)}dV_p dV_q +\ep^{-s}\int_{\Omega} W(v)\,dx\, ,
\end{equation} where $W(v)=\frac{1}{4}(1-v^2)^2$.\\

Likewise, define the min-max value \begin{equation}\label{minmaxvald}
     c_{\ep,s}^d(\p,\Omega):= \inf_{A\in\tilde{\mathcal{F}}_\p^d} \sup_{u\in A}\mathcal E_{\ep,s}^d (u,\Omega)\, ,
\end{equation}
where
\begin{equation}\label{Fpd}
    \tilde{\mathcal{F}}_\p^d=\{A\in\mathcal{F}_\p^d:|u|\leq 1 \ \forall u\in A\}
\end{equation}
and
$$
    \mathcal{F}_\p^d=\{A\subset H^{s/2,d}(\Omega):A\text{ is compact, symmetric, satisfies } \mathcal{H}^{\p}(A)<\infty \text{, and has ind}(A)\geq \p+1\}\, ,
$$
and put
\begin{equation}\label{deflpsd}
    l_{s}^d(\p,\Omega):=\lim_{\ep\to 0}c_{\ep,s}^d(\p,\Omega)\, .
\end{equation}
The lower and upper bounds \eqref{lpsbounds}, which we have proved in Theorem \ref{critlevbound}, are immediately seen to also hold for $l_{s}^d(\p,\Omega)$ with the same proof (since essentially the only property of the kernel $K_{M,s}$ that they use is its comparability with $\frac{\alpha_{n,s}}{\text{dist}_M^{n+s}(p,q)}$). Therefore, given $s_0\in (0,1)$, for $s\in(s_0,1)$ we have that
\begin{equation}\label{lpsboundsd}
    \frac{1}{C}\p^{s/n} \leq (1-s)l_{s}^d(\p,\Omega) \leq C\p^{s/n}
\end{equation}
with $C=C(s_0,\Omega)$.\\

We will first show a Weyl Law in this setting:
\begin{proposition}\label{prop:weyld}
Let $(M,g)$ be a closed manifold of dimension $n$. There exists a universal $\tau(n,s)>0$ such that
\begin{equation}\label{weyllawd}
    \lim_{\p\to\infty} \p^{-\frac{s}{n}}l_{s}^d(\p,M) = \tau(n,s)\textnormal{vol}(M,g)^{\frac{n-s}{n}}\, .
\end{equation}
\end{proposition}
We will afterwards show that this implies the Weyl Law with the canonical kernel $K_{M,s}$ of \eqref{weyllaw}, and in fact with the {\bf same constant} $\tau(n,s)$. This is reflective of the ``homogenisation" property which is behind the Weyl Law in the first place.

\subsubsection{Euclidean domain case}

The Weyl Law \eqref{weyllawd} will be proved first for Euclidean domains with piecewise smooth boundary, adapting the reasoning in \cite{LMN} to our case. As mentioned before, if $\Omega$ is such a domain, the Allen-Cahn energy will be defined using \eqref{ACenergyd} with the ambient Euclidean kernel $\frac{\alpha_{n,s}}{|x-y|^{n+s}}$.

The main tool is the so-called Lusternik-Schnirelman inequality, which compares the min-max quantities of domains with those for their subdomains:
\begin{lemma}[\textbf{Lusternik--Schnirelman inequality}]\label{LSineq}
Let $\Omega$, $\{\Omega_i\}_{i=1}^N$ and $\{\Omega_i^*\}_{i=1}^N$ be bounded open subsets in $\mathbb{R}^{n}$ with piecewise smooth boundaries such that
\begin{enumerate}
    \item $|\Omega|=|\Omega_i|=1$ for $i=1,...,N$
    \item $\Omega_i^*$ is similar to $\Omega_i$ for $i=1,...,N$
    \item $\{\Omega_i^*\}$ are pairwise disjoint subsets of $\Omega$
\end{enumerate}
Then 
\begin{equation}\label{linq}
    \p^{-\frac{s}{n}}l_{s}^d(\p,\Omega) \geq \sum_{i=1}^N |\Omega_i^*|\p_i^{-\frac{s}{n}}l_{s}^d(\p_i,\Omega_i) - \frac{c}{pV},
\end{equation}
where $\p_i=\lfloor \p|\Omega_i^*|\rfloor $, $V=\min\{|\Omega_i^*|\}$ and $c=\max_i \gamma(\Omega_i)$. Here $\gamma(\Omega_i)$ is a bound for $\mathfrak q^{-s/n}l_{s}^d(\mathfrak q,\Omega_i)$ independent of $\mathfrak q$ (which exists by \eqref{lpsboundsd}). \\

\end{lemma}
\begin{proof} For simplicity, we will write $c_\ep^d(\p,\Omega)$ and $l^d(\p,\Omega)$ instead of $c_{\ep,s}^d(\p,\Omega)$ and $l_{s}^d(\p,\Omega)$ during the proof, and likewise for other instances of this notation such as for $\mathcal E_{\ep}^d (u,\Omega)$ instead of $\mathcal E_{\ep,s}^d (u,\Omega)$. We shall use both notations interchangeably also in other proofs as long as there is no possible confusion. The proof consists of two steps.

\noindent \textbf{Step 1.} We prove that 
    \begin{equation}\label{cepcomp} c_{\ep}^d(\p,\Omega)\geq \sum_{i=1}^N c_{\ep}^d(\p_i,\Omega_i^*)\, ,
    \end{equation}
    i.e. that $c_{\ep}$ is superadditive in a certain sense.
    
    Fix $A\in \mathcal{F}_\p$. Choose some small room $\delta$ that will be made to go to zero, and define $A_i:=\{u\in A: \mathcal E_{\ep}^d (u,\Omega_i^*)\leq c_{\ep}^d(\p_i,\Omega_i^*)-\delta\}$. By definition of $c_{\ep}^d(\p_i,\Omega_i^*)$ as the min-max value among sets in $\mathcal{F}_{\p_i}(\Omega_i^*)$, the set $A_i$ does not belong to $\mathcal{F}_{\p_i}(\Omega_i^*)$, which shows that $\text{ind}(A_i)\leq \p_i$.
    
    What we need to see to show \eqref{cepcomp} is precisely that there exists some $u\in A\setminus\cup_{i=1}^N A_i$: Then,
    \begin{equation*}
        \sup_{v\in A}\mathcal E_{\ep}^d (v,\Omega)\geq \mathcal E_{\ep}^d (u,\Omega) \geq \sum_{i=1}^N\mathcal E_{\ep}^d (u,\Omega_i^*) > \sum_{i=1}^N c_\ep^d(\p_i,\Omega_i^*)-N\delta\, ,
    \end{equation*}
    which sending $\delta$ to $0$ and taking infima among all $A\in\mathcal{F}_\p$ gives $c_{\ep,s}^d(\p,\Omega)\geq \sum_{i=1}^N c_{\ep}^d(\p_i,\Omega_i^*)$.
    
    To see that $A\setminus\cup_{i=1}^N A_i \neq \emptyset$, observe that by Lemma \ref{indsub}
    $$
    \text{ind}(\cup_{i=1}^N A_i)\leq\sum_{i=1}^N \text{ind}(A_i)\leq\sum_{i=1}^N \p_i\leq \p\, ,
    $$
    but on the other hand $\text{ind}(A)\geq \p+1$ by definition since $A\in\mathcal{F}_\p$.

\noindent \textbf{Step 2.} Conclusion.

From \eqref{cepcomp}, we get
\begin{align*}
    \p^{-\frac{s}{n}}c_{\ep}^d(\p,\Omega)&\geq \sum_{i=1}^N \p^{-\frac{s}{n}}c_{\ep}^d(\p_i,\Omega_i^*) =\sum_{i=1}^N \p^{-\frac{s}{n}}|\Omega_i^*|^{\frac{n-s}{n}}c_{\ep/|\Omega_i^*|^{1/n}}^d(\p_i,\Omega_i).
\end{align*}
We have used that $\Omega_i$ is isometric to $|\Omega_i^*|^{-\frac{1}{n}}\Omega_i^*$ together with the scaling of the Allen-Cahn energy \eqref{ACenergyd} under dilations.

Next, we multiply and divide by $\p_i^{-\frac{s}{n}}$ inside each term in the sum, and then use $\frac{\p_i}{\p|\Omega_i^*|}\geq 1-\frac{1}{\p|\Omega_i^*|}\geq (1-\frac{1}{\p|\Omega_i^*|})^{\frac{n}{s}}$ which comes from $\p_i=\lfloor \p|\Omega_i^*|\rfloor \geq \p|\Omega_i^*|-1$:

\begin{align*}
    \p^{-\frac{s}{n}}c_{\ep}^d(\p,\Omega) &\geq \sum_{i=1}^N \p^{-\frac{s}{n}}|\Omega_i^*|^{\frac{n-s}{n}}c_{\ep/|\Omega_i^*|^{1/n}}^d(\p_i,\Omega_i= \sum_{i=1}^N |\Omega_i^*|(\frac{\p_i}{\p|\Omega_i^*|})^{\frac{s}{n}}\p_i^{-\frac{s}{n}}c_{\ep/|\Omega_i^*|^{1/n}}^d(\p_i,\Omega_i)\\
    &\geq \sum_{i=1}^N |\Omega_i^*|(1-\frac{1}{\p|\Omega_i^*|})\p_i^{-\frac{s}{n}}c_{\ep/|\Omega_i^*|^{1/n}}^d(\p_i,\Omega_i).
\end{align*}
Now, we let $\ep\to 0$ and use that $\sum_{i=1}^N |\Omega_i^*|\leq |\Omega|=1$ to find that 
\begin{align*}
    \p^{-\frac{s}{n}}l^d(\p,\Omega)&\geq \sum_{i=1}^N |\Omega_i^*|\p_i^{-\frac{s}{n}}l^d(\p_i,\Omega_i)-\frac{\max_j \gamma(\Omega_j)}{\p V}\, ,
\end{align*}
which concludes the proof.
\end{proof}

Set $\tilde{l}_{s}^d(\p,\Omega):=\p^{-\frac{s}{n}}l_{s}^d(\p,\Omega)$. We will first prove the Weyl Law for Euclidean cubes, which amounts to showing

\begin{proposition}
\label{cubelim} For the unit cube $C$ in $\mathbb{R}^n$, \ $\liminf_{\p\to\infty} \tilde{l}_{s}^d(\p,C) = \limsup_{\p\to\infty} \tilde{l}_{s}^d(\p,C)$.
\end{proposition}

Once we prove this proposition, we will be able to designate the universal constant $\tau(n,s)$ that will appear in the Weyl Law of an arbitrary manifold. Indeed, if Theorem \ref{thmweyl} holds then Proposition \ref{cubelim} implies that the constant needs to be $\tau(n,s):=\lim_\p \tilde{l}_{s}^d(\p,C)$.

\begin{proof}[Proof of Proposition \ref{cubelim}]
Choose sequences $\{\p_k\}$ and $\{\mathfrak q_j\}$ giving the values of the limsup and the liminf of $\tilde{l}_{s}^d(\p,C)$, respectively.

Fix $k$, and for all $j$ large enough define $\delta_j=\frac{\p_k}{\mathfrak q_j}<1$. We will use the Lusternik--Schnirelman inequality with $\mathfrak q_j$ on the left side and $\p_k$ on the right side, together with the similarity between a cube and its subcubes, to prove that liminf$\geq$limsup. This will conclude the proof.\\

Let $N_j$ be the maximal number of cubes $C_i^*$ of \textnormal{vol}ume $\delta_j$ contained in $C$ and with pairwise disjoint interiors. From the definition of $\delta_j$, $\lfloor \mathfrak q_j|C_i^*|\rfloor = \mathfrak q_j|C_i^*| = \p_k$. Lemma \ref{LSineq} then gives
\begin{align}\label{cubein}
    \tilde{l}_{s}^d(\mathfrak q_j,C)\geq\sum_{i=1}^{N_j}|C_i^*|\tilde{l}_{s}^d(\p_k,C)-\frac{\gamma(C)}{\mathfrak q_j\delta_j}=N_j\delta_j\tilde{l}_{s}^d(\p_k,C)-\frac{\gamma(C)}{\p_k}.
\end{align}
Since $N_j\delta_j$ tends to the volume of $C$ by definition of $N_j$, letting $j\to \infty$ we get that
\begin{align*}
    \liminf_\p \tilde{l}_{s}^d(\p,C)\geq \tilde{l}_{s}^d(\p_k,C)-\frac{\gamma(C)}{\p_k}.
\end{align*}
Sending now $k$ to infinity, we arrive at $\liminf_\p \tilde{l}_{s}^d(\p,C)\geq \limsup_\p \tilde{l}_{s}^d(\p,C)$.
\end{proof}

Next, we will consider an arbitrary Lipschitz domain $\Omega$ with piecewise smooth boundary (assumed to be of \textnormal{vol}ume 1, after scaling) and show that $\lim_\p\tilde{l}_{s}^d(\p,\Omega)$ exists and is equal to $\tau(n,s)=\lim_p \tilde{l}_{s}^d(\p,C)$:
\begin{proposition}[Weyl Law for Euclidean domains]\label{weyleucd} Let $\Omega\subset\R^n$ be a Lipschitz domain with piecewise smooth boundary. Then
\begin{equation}
    \lim_\p\tilde{l}_{s}^d(\p,\Omega)=\tau(n,s)\textnormal{vol}(\Omega)^{\frac{n-s}{n}}\, ,
\end{equation}
where $\tau(n,s)=\lim_p \tilde{l}_{s}^d(\p,C)>0$ is a universal constant depending only on $n$ and $s$.
\end{proposition}
\begin{proof}
By scaling, we can assume that ${\rm vol}(\Omega)=1$.

\noindent \textbf{Step 1.} We prove that $\liminf_{\p\to\infty}\tilde{l}_{s}^d(\p,\Omega)\geq\tau(n,s)$.

Fix $\delta>0$. We can find a family of $N$ similar cubes $C_i^*$ with volume $\delta_i$ inside $\Omega$ with pairwise disjoint interiors and with total volume $\sum_{i=1}^{N}|C_i^*|\geq (1-\delta)$. Applying Lemma \ref{LSineq},
    \begin{align}\label{i1in} \tilde{l}(\p,\Omega)\geq \sum_{i=1}^{N}|C_i^*|\tilde{l}(\lfloor \p\delta_i \rfloor,C)-\frac{\gamma(c)}{\p \min_i \delta_i}\, .
    \end{align}
Taking $\liminf$ on both sides, we find for every $\delta>0$ that
\begin{align*}
    \liminf_p\tilde{l}(\p,\Omega)\geq(1-\delta)\liminf_\p\tilde{l}(\p,C)=(1-\delta)\tau(n,s)\, ,
\end{align*}
which gives the conclusion after letting $\delta\to 0$.

\noindent \textbf{Step 2.} $\limsup_{\p\to\infty}\tilde{l}_{s}^d(\p,\Omega)\leq\tau(n,s)$.

We will need the following simple result, corresponding to \cite[Lemma 3.5]{LMN}:

\noindent \textbf{Claim.} We can find a family $\{\Omega_i\}_{i\in\N}$ of subsets of $C$ with disjoint interiors, all of them similar to $\Omega$, such that given $\ep>0$ there exists $N\in\N$ such that $\sum_{i=1}^N \textnormal{vol}(\Omega_i)\geq 1-\ep$.
\begin{proof}[Proof of the Claim:]

Let $\Omega_1$ be similar to $\Omega$ and contained in $C$, and write $v=\textnormal{vol}(\Omega_1)$. Using the notation and reasoning in \cite{LMN}, define $R_1$ to be the closure of $C\setminus\Omega_1$. We can find disjoint cubes $\{C_{i,1}\}_{i=1}^{Q_1}$ contained in $R_1$ and such that $\sum_{i=1}^{Q_1} \textnormal{vol}(C_{i,1})\geq \textnormal{vol}(R_1)/2$. Now, for each $i=1,...,Q_1$, we know that there exists a region $\Omega_{i,1}\subset C_{i,1}$ similar to $\Omega$ and with \textnormal{vol}ume $v\textnormal{vol}(C_{i,1})$, by correspondence with the case of $\Omega_1$ and $C$. Let $\Omega_2=\cup_{1\leq i\leq Q_1}\Omega_{i,1}$.

    The next step is to define $R_2=C\setminus(\Omega_1\cup\Omega_2)$. Once again, we can find disjoint cubes $\{C_{i,2}\}_{i=1}^{Q_2}$ contained in $R_2$ and such that $\sum_{i=1}^{Q_2} \textnormal{vol}(C_{i,2})\geq \textnormal{vol}(R_2)/2$, and regions $\Omega_{i,2}\subset C_{i,2}$ for each $i=1,...,Q_2$ with \textnormal{vol}ume $v\textnormal{vol}(C_{i,2})$. We can now define $\Omega_3=\cup_{1\leq i\leq Q_2}\Omega_{i,2}$, and repeat the procedure inductively.
    
    We claim that for every $\ep>0$ there exists an $N\in\N$ such that $\sum_{i=1}^N \textnormal{vol}(\Omega_i)\geq 1-\ep$. Indeed, by construction, $\sum_{i=1}^n \textnormal{vol}(\Omega_i)\geq \sum_{i=1}^{n-1} \textnormal{vol}(\Omega_i) + v(1-\sum_{i=1}^{n-1} \textnormal{vol}(\Omega_i))/2=\sum_{i=1}^{n-1} \textnormal{vol}(\Omega_i)(1-v/2)+v/2$. This shows that $\sum_{i=1}^n \textnormal{vol}(\Omega_i)\geq v/2\sum_{i=1}^n (1-v/2)^i$, and the right side goes to $1$ as $n\to \infty$.
\end{proof}

We proceed with the proof of Step 2. Consider the $\{\Omega_i\}_{i\in \N}$ from the claim above. Given $\ep>0$, select $N\in \N$ such that $\sum_{i=1}^N \textnormal{vol}(\Omega_i)\geq 1-\ep$. Moreover, given $\delta>0$, let $\{C_j\}_{j=1}^{N_\delta}$  be a maximal collection of $N_\delta$ cubes contained in $C$ with \textnormal{vol}ume $\delta$ and pairwise disjoint interiors; in particular, $\delta N_{\delta}\to 1$ as $\delta\to 0$. For each $C_j$, by correspondence with the case of $C$ we can find regions $\{\Omega_{i,j}\}_{i=1}^{N}$ inside $C_j$ with pairwise disjoint interiors, similar to $\Omega$ and such that $|\Omega_{i,j}|=\delta|\Omega_i|$.

    The Lusternik-Schinerlman inequality of Lemma \ref{LSineq} tells us then that, given $\p\in \N$ and $\p_i=\lfloor \p |\Omega_{i,j}| \rfloor=\lfloor \p \delta|\Omega_{i}| \rfloor$, we have
    \begin{align*}
        \tilde l^d(\p,C) &\geq \sum_{j=1}^{N_\delta}\sum_{i=1}^N |\Omega_{i,j}|\tilde l^d(\p_i,\Omega) - \frac{\gamma(\Omega)}{\p\delta v}= N_{\delta}\sum_{i=1}^N \delta|\Omega_{i}|\tilde l^d(\p_i,\Omega) - \frac{\gamma(\Omega)}{\p\delta v}\, .
    \end{align*}
    Let $\mathfrak q\in\N$; selecting $\delta(\p)=\frac{\mathfrak q}{\p|\Omega_1|}$, we have that $\p_i=\lfloor \mathfrak q \frac{|\Omega_{i}|}{|\Omega_{1}|} \rfloor$, and in particular $\p_1=\mathfrak q$. Therefore,
    \begin{align}\label{i2in}
        \tilde l^d(\p,C) &\geq \delta(\p) N_{\delta(\p)}\Big[|\Omega_{1}|\tilde l^d(\mathfrak q,\Omega)+\sum_{i=2}^N |\Omega_{i}|\tilde l^d(\Big\lfloor \mathfrak q \frac{|\Omega_{i}|}{|\Omega_{1}|} \Big\rfloor,\Omega)\Big] - \frac{|\Omega_1|\gamma(\Omega)}{\mathfrak q\delta v}\, .
    \end{align}
    Taking the limit in $\p$, since $\delta(\p) N_{\delta(\p)}\to 1$ we get that
    \begin{align*}
        \tau(n,s) \geq \Big[|\Omega_{1}|\tilde l^d(\mathfrak q,\Omega)+\sum_{i=2}^N |\Omega_{i}|\tilde l^d(\Big\lfloor \mathfrak q \frac{|\Omega_{i}|}{|\Omega_{1}|} \Big\rfloor,\Omega)\Big] - \frac{|\Omega_1|\gamma(\Omega)}{\mathfrak q\delta v}\, .
    \end{align*}
    Choosing now $\mathfrak q=\mathfrak q_k$ such that $\lim_{k\to\infty} \tilde l_{\mathfrak q_k}(\Omega)=\limsup_{\p\to\infty}\tilde l_\p(\Omega)$ and taking $\liminf_k$ in the inequality, we get that
    \begin{align*}
        \tau(n,s) &\geq \Big[|\Omega_{1}|\limsup_{\p\to\infty}\tilde l^d(\p,\Omega)+\liminf_{\p\to\infty}\tilde l^d(\p,\Omega)\sum_{i=2}^N |\Omega_{i}|\Big]\geq\Big[|\Omega_{1}|\limsup_{\p\to\infty}\tilde l^d(\p,\Omega)+\liminf_{\p\to\infty}\tilde l^d(\p,\Omega)(1-|\Omega_1|-\ep)\Big]\, .
    \end{align*}
    Finally, using Step 1 which stated that $\liminf_{\p\to\infty}\tilde l^d(\p,\Omega)\geq \tau(n,s)$ we find that
    \begin{align*}
        (|\Omega_1|+\ep)\tau(n,s) \geq |\Omega_{1}|\limsup_{\p\to\infty}\tilde l^d(\p,\Omega)\, ,
    \end{align*}
    which finishes the proof after letting $\ep\to 0$.
\end{proof}

\subsubsection{Closed manifold case -- Proof of Proposition \ref{prop:weyld}}

Let $(M,g)$ be a closed manifold. We will use the Weyl Law for Euclidean domains that we just proved to show the Weyl Law for closed manifolds, following as before the strategy in \cite{LMN}. For that, we will need to be able to compare min-max energies on a domain of a closed manifold with min-max energies on a Euclidean domain, for which the next simple lemma will prove useful.
\begin{lemma}\label{Lipineq}
Let $(M_1,g_1)$ and $(M_2,g_2)$ be Riemannian manifolds, and let $\Omega_1\subset M_1$ and $\Omega_2\subset M_2$ be subdomains with piecewise $\mathcal{C}^1$ boundary. Let $F:\Omega_1\to\Omega_2$ be a $(1+\delta)$-biLipschitz diffeomorphism, meaning that $\max\{{\rm Lip}(F),{\rm Lip}(F^{-1})\}\leq (1+\delta)$. Then, if $u\in H^{s/2,d}(\Omega_2)$, it holds that $u\circ F\in H^{s/2,d}(\Omega_1)$ and
$$
\mathcal E_{\ep,s}^d(u\circ F,\Omega_1)\leq (1+\delta)^{3n+s}\mathcal E_{\ep,s}^d(u,\Omega_2)\, .
$$
The Lipschitz constant ${\rm Lip}(F)$ of $F$ is computed using the ambient Riemannian distances of $M_1$ and $M_2$.
\end{lemma}
\begin{proof}
By definition $\text{dist}_{M_2}(F(p),F(q)) \leq \text{Lip}(F)\text{dist}_{M_1}(p,q)$. Moreover, letting $JF^{-1}$ denote the Jacobian of the map $F^{-1}$, we have the bound
$$
|JF^{-1}|(p)\leq\|DF^{-1}\|_{\infty}^n\leq\text{Lip}(F^{-1})^n\, .
$$
Then, changing variables with $F$ and using the facts above, we can compute
{\small
    \begin{align*}
    \mathcal E_{\ep,s}^d(u\circ F&,\Omega_1)=\frac{1}{4}\iint_{\Omega_1\times \Omega_1}dp\,dq\, (u\circ F(p)-u\circ F(q))^2{\rm dist}_{M_1}(p,q)^{-(n+s)}+ \frac{1}{\ep^s}\int_{\Omega_1}dp\, W(u\circ F(p))\\
    &\leq \frac{{\rm Lip}(F)^{n+s}}{4}\iint_{\Omega_1\times \Omega_1}dp\,dq\, (u\circ F(p)-u\circ F(q))^2{\rm dist}_{M_2}(F(p),F(q))^{-(n+s)}+ \frac{1}{\ep^s}\int_{\Omega_1}dp\, W(u\circ F)\\
    &= \frac{{\rm Lip}(F)^{n+s}}{4}\iint_{\Omega_2\times \Omega_2}dp\,dq\, (u(p)-u(q))^2{\rm dist}_{M_2}(p,q)^{-(n+s)}|JF^{-1}(p)||JF^{-1}(q)|+\frac{1}{\ep^s}\int_{\Omega_2}dp\, W(u(p))|JF^{-1}(p)|\\
    &\leq \{{\rm Lip}(F)^{n+s}\|DF^{-1}\|_{\infty}^{2n},\|DF^{-1}\|_{\infty}^{n}\}\mathcal E_\ep(u,\Omega_2)\leq (1+\delta)^{3n+s}\mathcal E_{\ep,s}^d(u,\Omega_2)\, ,
\end{align*}
}
which concludes the proof.
\end{proof}

We show Proposition \ref{prop:weyld} in two comparisons. We start with the simpler one:
\begin{proposition}\label{infgeq2}
$\liminf_p\tilde{l}_{s}^d(\p,M)\geq\tau(n,s)\text{\textnormal{vol}}(M,g)^{\frac{n-s}{n}}$.
\end{proposition}
\begin{proof}
    Without loss of generality we can assume (by scaling) that $\textnormal{vol}(M,g)=1$. Let $\delta>0$, and select disjoint geodesic balls $B_i$ inside $M$ of radius $r_i\leq \bar r$ very small such that $\sum \textnormal{vol}(B_i)\geq \textnormal{vol}(M)/(1+\delta)=1/(1+\delta)$. We know that $c_\ep^d(\p,M)\geq\sum c_\ep^d(\p_i,B_i)$ with $\p_i=\lfloor \p\textnormal{vol}(B_i) \rfloor$, by Step $1$ in the proof of Lemma \ref{LSineq}. Consider now normal coordinates centered at the centers of the $B_i$, with associated diffeomorphisms $F_i:\B_i\to B_i$ going from Euclidean balls $\B_i$ to the geodesic balls $B_i$. Then, thanks to the compactness of $M$ we have that
    $$
    1-\delta(\bar r)\leq{\rm Lip}(F_i)\leq 1+\delta(\bar r)\,,
    $$
    where $\delta(\bar r)\to 0$ as $\bar r\to 0$.
    Then Lemma \ref{Lipineq} immediately gives (after taking min-max values) that, denoting by $\B$ the Euclidean ball of \textnormal{vol}ume $1$,  
    $$
    c_\ep^d(\p_i,B_i)\geq (1+\delta(\bar r))^{-(3n+s)}c_{\ep}^d(\p_i,\B_i)=(1+\delta(\bar r))^{-(3n+s)}|\B_i|^{\frac{n-s}{n}}c_{\ep|\B_i|^{-\frac{1}{n}}}{1+\delta(\bar r)}^d(\p_i,\B)\, ,
    $$
    thus
    $$
    c_\ep^d(\p,M)\geq\sum_i c_\ep^d(\p_i,B_i)\geq(1+\delta(\bar r))^{-(3n+s)}\sum_i|\B_i|^{\frac{n-s}{n}}c_{\ep|\B_i|^{-\frac{1}{n}}}^d(\p_i,\B)\, .
    $$
    Multiplying by $\p^{-s/n}$ and manipulating a bit as in the proof of Lemma \ref{LSineq},
    \begin{align*}
        \p^{-\frac{s}{n}}c_\ep^d(\p,M)&\geq(1+\delta(\bar r))^{-(3n+s)}\sum_i \p^{-\frac{s}{n}}|\B_i|^{\frac{n-s}{n}}c_{\ep|\B_i|^{-\frac{1}{n}}}^d(\p_i,\B)\\
        &=(1+\delta(\bar r))^{-(3n+s)}\sum_i |\B_i|(\frac{\p_i}{\p|\B_i|})^{\frac{s}{n}}\p_i^{-s/n} c_{\ep|\B_i|^{-\frac{1}{n}}}^d(\p_i,\B)\\
        &\geq (1+\delta(\bar r))^{-(3n+s)}\sum_i |B_i|(\frac{|B_i|}{|\B_i|}-\frac{1}{\p|\B_i|})^{\frac{s}{n}}\p_i^{-s/n} c_{\ep|\B_i|^{-\frac{1}{n}}}^d(\p_i,\B)\, .
    \end{align*}
    Letting $\ep\to 0$, we arrive at
    \begin{align}\label{i3in}
        \tilde l_{s}^d(\p,M)
        &\geq (1+\delta(\bar r))^{-(3n+s)}\sum_i |B_i|(\frac{|B_i|}{|\B_i|}-\frac{1}{\p|\B_i|})^{\frac{s}{n}}\tilde l_{s}^d(\p_i,\B)\, .
    \end{align}
    Taking then $\liminf$ in $\p$ on both sides, and using the Weyl Law for the Euclidean domain $\B$ (which is a ball with volume $1$), we get that
    \begin{align*}
        \liminf_\p \tilde l_{s}^d(\p,M)
        &\geq (1+\delta(\bar r))^{-(3n+s)}\sum_i |B_i|(\frac{|B_i|}{|\B_i|})^{\frac{s}{n}}\tau(n,s)
    \end{align*}
    Using that ${\rm Lip}(F_i^{-1})\leq 1+\delta(\bar r)$ and $\sum_i |B_i|\geq 1/(1+\delta)$, we find that
    \begin{align*}
        \liminf_\p \tilde l_{s}^d(\p,M)&\geq(1+\delta(\bar r))^{-(3n+s)}(1+\delta(\bar r))^{-s}\sum_i |B_i|\tau(n,s)\geq (1+\delta(\bar r))^{-(3n+2s)}(1+\delta)^{-1}\tau(n,s)\, .
    \end{align*}
    Letting $\delta\to 0$ and $\bar r\to 0$, we conclude the result.
\end{proof}
We now prove the opposite inequality:
\begin{proposition}\label{supleq2}
$\limsup_\p\tilde{l}_{s}^d(\p,M)\leq\tau(n,s)\text{\textnormal{vol}}(M,g)^{\frac{n-s}{n}}$.
\end{proposition}
We will need the following auxiliary result.
\begin{lemma}\label{LipDecomp}
    Let $N$ be a compact manifold, possibly with boundary. There exists a decomposition $\{C_i\}_i$ of $N$ into connected Lipschitz sets such that the $C_i$ intersect only at the boundaries and for each $i$ there exists a $(1+\delta)$-biLipschitz equivalence $F_i: C_i \to \mathcal{C}_i$ which maps $C_i$ to a Euclidean domain $\mathcal{C}_i\subset\R^n$. More precisely, the $F_i$ are given by the restriction of geodesic normal coordinate maps to the $C_i$.
\end{lemma}
\begin{proof}[Proof of the lemma]
This simple result is proved in \cite{LMN}. One first covers $N$ with small metric balls (possibly centered at boundary points) $B_{r/2}(p_1),...,B_{r/2}(p_d)$. One then constructs the $C_i$ as follows: Put $ \widetilde C_1=B_r(p_1)$, and inductively set $\widetilde C_{k+1}=B_{r_{k+1}}(p_{k+1})\setminus\cup_{i=1}^k \widetilde C_k$, where $r_{k+1}\in[r/2,r]$ is selected so that $\partial B_{r_{k+1}}$ intersects the boundaries of $ \widetilde C_1,...,\widetilde C_k$ transversally. Finally, one defines the collection $\{C_i\}_i$ by considering the connected components of the $\widetilde C_i$. If $r$ is small enough, the normal coordinate maps of the balls $B_{r}(p_1),...,B_{r}(p_d)$ are biLipschitz equivalences with constants close to $1$, thus their restrictions to the $C_i$ act as the $F_i$ we are looking for.
\end{proof}
\begin{proof}[Proof of Proposition \ref{supleq2}]
The proof is split into four steps.

\noindent \textbf{Step 1.}
    We consider a decomposition $\{C_i\}_i$ of $M$ as in the Lemma, with associated $(1+\delta)$-biLipschitz equivalences $F_i:C_i\to\mathcal{C}_i\subset\R^n$. By translating the $\mathcal C_i$ if necessary, we can assume that they are pairwise disjoint; this gives rise to a Euclidean domain $\widetilde\Omega=\bigcup \mathcal C_i\subset\R^n$. Construct then the connected Euclidean domain $\Omega=\bigcup_i \mathcal C_i\cup\bigcup_j T_j$, where the $T_j$ are small tubes connecting the $\mathcal C_i$ and with total volume less than $\delta \textnormal{vol}(M)$, so that 
    \begin{align}
    {\rm vol}(\Omega)&\leq {\rm vol}(\widetilde\Omega)+\delta {\rm vol}(M)=\sum_i {\rm vol}(\mathcal C_i)+\delta {\rm vol}(M)\nonumber\\
    &\leq (1+\delta)^n\sum_i{\rm vol}(C_i)+\delta {\rm vol}(M)\leq(1+2\delta)^n{\rm vol}(M).\label{volglue}
    \end{align}
    To prove the Proposition, given an admissible set $A$ in $\tilde{\mathcal{F}}_\p^d(\Omega)$ we need to construct an admissible set $A'$ in $\tilde{\mathcal{F}}_\p^d(M)$ with energy not surpassing that of $A$ by much. We just need an equivariant assignation from functions on $\Omega$ to functions on $M$ which doesn't increase energy by more than a little.
    
    Let, then, $u\in H^{s/2,d}(\Omega)$ with $|u|\leq 1$; it suffices to consider functions with absolute value at most one by \eqref{minmaxvald}-\eqref{deflpsd}. Thanks to our functional setting, we can obtain our desired function $U$ on $H^{s/2,d}(M)$ simply by composition:
    Given $i$, define the function $U\big|_{C_i}$ on $C_i$ by $U\big|_{C_i}=u\circ F_i$ . Since $M=\cup_i \mathcal{C}_i$ and the $\mathcal{C}_i$ intersect only on a set of zero Lebesgue measure, we get a well-defined function $U\in L^2(M)$. Let $f:H^{s/2,d}(\Omega)\to H^{s/2,d}(M)$ be defined by $f(u):=U$. Then, given $A\in\tilde{\mathcal{F}}_\p^d(\Omega)$, we set $A'=f(A)$. All of the properties in \eqref{cohom} are easily seen to be satisfied by $A'$ if they are by $A$, with the exception of the Hausdorff measure condition in \eqref{cohom}. In fact, $A'$ might not be in $\tilde{\mathcal{F}}_\p^d(M)$, but we can find a set $A'_\delta\in \tilde{\mathcal{F}}_\p^d(M)$ which is arbitrarily close to $A'$ in terms of Allen--Cahn energies, which allows us to argue as if $A'\in\tilde{\mathcal{F}}_\p^d(M)$ without loss of generality. The set $A'_\delta$ is constructed by regularisation as in the proof of Theorem \ref{mmbounds}, using a cutoff depending on the distance to the boundaries of the $C_i$; we postpone this construction to the last step of the proof.
    
    \noindent \textbf{Step 2.}
    Our job now is to compare the energy of $U$ with the energy of $u$. First, we bound the contributions of the interactions between different pieces, using that $|U|\leq 1$:
    \begin{align*}
    \mathcal E_\ep^d(U,M)&=\sum_i \mathcal E_\ep^d(U,C_i) + \sum_{i\neq j} \iint_{C_i\times C_j}(U(p)-U(q))^2\text{dist}_M(p,q)^{-(n+s)}\,dp\,dq\\
    &\leq \sum_i \mathcal E_\ep^d(U,C_i) + \sum_{i\neq j} \iint_{C_i\times C_j}4\text{dist}_{M}(p,q)^{-(n+s)}\,dp\,dq\\
    &\leq \sum_i \mathcal E_\ep^d(U,C_i) + \sum_{i\neq j} \iint_{C_i\times(M\setminus C_i)}4\text{dist}_M(p,q)^{-(n+s)}\,dp\,dq\\
    &= \sum_i \mathcal E_\ep^d(U,C_i) + c\sum_{i,j} \text{Per}_s^d(C_i) \,.
    \end{align*}
     Here the fractional perimeter $\text{Per}_s^d$ is of course understood as being defined with $\frac{\alpha_{n,s}}{\text{dist}_M^{n+s}}$ as the associated kernel. Since the $C_i$ are Lipschitz domains, which in particular have bounded classical perimeter, we have that
     \begin{equation}\label{lipdper}
        {\rm Per}_s^d( C_i)\leq \frac{C}{1-s}{\rm Per}(C_i)^s<\infty\,,
    \end{equation}
    see \eqref{interpropM}.
    
    Now, by Lemma \ref{Lipineq} we can bound $
    \mathcal E_\ep^d(U,C_i)\leq (1+\delta)^{3n+s} \mathcal E_\ep^d(u,C_i)
    $. Substituting, we obtain that
    \begin{align}
    \mathcal E_\ep^d(U,M)&\leq (1+\delta)^{3n+s}\sum_i \mathcal E_\ep^d(u,C_i) + c\sum_{i,j} \text{Per}_s^d(C_i)\nonumber\\
    &\leq (1+\delta)^{3n+s}\mathcal E_\ep^d(u,\Omega)+ c\sum_{i,j} \text{Per}_s^d(C_i)\, .\label{Uucin}
    \end{align}
    In particular, this shows that $U\in H^{s/2,d}(M)$.
    
    \noindent \textbf{Step 3.}
    Let $\p\in\N$. Taking min-max values in \eqref{Uucin}, letting $\ep\to 0$, and dividing by $\p^{s/n}$, we arrive at
    \begin{equation}\label{i4in}
    \tilde{l}_{s}^d(\p,M)\leq (1+\delta)^{3n+s}\tilde{l}_{s}^d(\p,\Omega)+\frac{c}{\p^{s/n}}\sum_{i,j} \text{Per}_s^d(C_i)\, .
    \end{equation}
    Taking $\limsup$ in $\p$ on both sides, using the Weyl Law of Proposition \eqref{weyleucd} for the Euclidean domain $\Omega$, and using \eqref{volglue}, we get that
    $$
    \limsup_\p \tilde l_{s}^d(\p,M)\leq (1+\delta)^{3n+s}\tau(n,s)\text{\textnormal{vol}}(\Omega)^\frac{n-s}{n} \leq (1+\delta)^{3n+s}(1+2\delta)^{n-s}\tau(n,s)\text{\textnormal{vol}}(M,g)^\frac{n-s}{n}.
    $$
    We conclude by letting $\delta\to 0$.
    
\noindent \textbf{Step 4.} To conclude, it only remains to deal with the technical Hausdorff measure issue from Step 1. We give the necessary modifications in Appendix \ref{app:hausdorffmod}.
\end{proof}
We can now conclude:
\begin{proof}[Proof of Proposition \ref{prop:weyld}]
    It follows from the combination of Propositions \ref{infgeq2} and \ref{supleq2}.
\end{proof}

\subsection{Equivalence with the canonical Sobolev kernel -- Proof of Theorem \ref{thmweyl}}
In this subsection we show that using the canonical kernel $K_{M,s}(p,q)$ for the Sobolev energy $H^{s/2}(M)$ we obtain exactly the same nonlocal Weyl Law. We use the same notation as in the previous sections, so that $l_{s}(\p,M)$ denotes the limit in \eqref{deflps} and which uses the kernel $K_{M,s}(p,q)$ in the definition \eqref{ACenergy} of the Allen--Cahn energy, and likewise for $l_{s}^{d}(\p,M)$ and the associated kernel $\frac{\alpha_{n,s}}{\text{dist}_{M}^{n+s}(p,q)}$.

We will need two preliminary results.

\begin{proposition}\label{kermfdconv}
    Let $M$ be a closed manifold, and let $s_0\in(0,1)$. Then, there exists $C=C(M,s_0)$ such that if $s\in(s_0,1)$, then 
$$\Big|K_{M,s}(p,q)-\frac{\alpha_{n,s}}{{\rm dist}_M(p,q)^{n+s}}\Big|\leq \frac{C(M,s_0)}{{\rm dist}_M(p,q)^{n+s-1}}\, .$$
In other words, the difference between the kernels in the left hand side is of lower order.
\end{proposition}
\begin{proof}
    As stated in Remark \ref{fbsvdg}, there is $R_0>0$ depending on $M$ such that the assumptions ${\rm FA}_1(M,g,R_0,p,\varphi_p)$ are satisfied at every $p\in M$, with $\varphi_p$ given by the restriction of the exponential map centered at $p$ to $\B_{R_0}(0)$.\\
    Let then $p$ and $q$ be such that $\text{dist}_M(p,q)\leq R_0/4$. We will apply Proposition \ref{prop:kern1} with $x=0=\varphi_p^{-1}(p)$, so that $A(x)=Id$ since the metric in normal coordinates is Euclidean at the origin. Letting $y=\varphi^{-1}(q)$, \eqref{remaining0} then gives that
    \begin{align*}
        \Big|K(p,q)-\frac{\alpha_{n,s}}{|x-y|^{n+s}}\Big|\leq R_0^{-1}\frac{C(n,s)}{|x-y|^{n+s-1}}\, ,
    \end{align*}
    where $C(n,s)$ is uniform for $s\in(s_0,1)$.\\ Since $\text{dist}_M(p,q)=|x-y|$ thanks to having chosen normal coordinates centered at $p$, we deduce that
    \begin{align*}
        \Big|K(p,q)-\frac{\alpha_{n,s}}{\text{dist}_M(p,q)^{n+s}}\Big|&\leq \frac{C(M,s_0)}{\text{dist}_M(p,q)^{n+s-1}}\, .
    \end{align*}
    Consider now $p$ and $q$ such that $\text{dist}_M(p,q)\geq R_0/4$ instead. The bound
    \begin{align*}
        \frac{\alpha_{n,s}}{\text{dist}_M(p,q)^{n+s}}\leq C(n,s_0)\frac{R_0/4}{\text{dist}_M(p,q)^{n+s-1}}=\frac{C(M,s_0)}{\text{dist}_M(p,q)^{n+s-1}}
    \end{align*}
    is then immediate. Moreover, applying \eqref{remaining2} with $\alpha=0$ and $R=R_0/4$ we can bound
    \begin{align*}
        K(p,q)&\leq \frac{C(n,s_0)}{(R_0/4)^{n+s}}=C(M,s_0)\, .
    \end{align*}
    Since $\text{dist}_M(p,q)\leq {\rm diam}(M)<\infty$, we can then deduce that
    \begin{align*}
        \Big|K(p,q)-\frac{\alpha_{n,s}}{\text{dist}_M(p,q)^{n+s}}\Big|&\leq K(p,q)+\frac{\alpha_{n,s}}{\text{dist}_M(p,q)^{n+s}}\leq C(M,s_0)+\frac{C(M,s_0)}{\text{dist}_M(p,q)^{n+s-1}}\leq \frac{C(M,s_0)}{\text{dist}_M(p,q)^{n+s-1}}\, .
    \end{align*}
    Combining the results for $\text{dist}_M(p,q)\leq R_0/4$ and $\text{dist}_M(p,q)\geq R_0/4$, we reach the desired conclusion.
\end{proof}
\begin{lemma}\label{complem}
The function spaces $H^{s/2}(M)$ and $H^{s/2,d}(M)$ coincide, and the associated fractional Sobolev norms are equivalent. Moreover, if $s_0\in(0,1)$ and $s\in(s_0,1)$, given $\delta>0$ we can bound
$$
|\mathcal E_{\ep,s}(u,M)-\mathcal E_{\ep,s}^d(u,M)|\leq C(M,s_0)\big[C(\delta)\|u\|_{L^\infty(M)}^2+\delta \mathcal E_{\ep,s}^d(u,M)\big]\,
$$
for all $u\in H^{s/2}(M)\cap L^\infty(M)$.
\end{lemma}
\begin{proof}[Proof of the Lemma] We divide the proof in two steps. We will omit the dependence of the constants on $s_0$.\\
\textbf{Step 1.} We show that $u\in H^{s/2}(M)$ if and only if $u\in H^{s/2,d}(M)$.\\
Proposition \ref{kermfdconv} shows that there is some $R(M)>0$ such that, for $p$ and $q$ with ${\rm dist}(p,q)\leq R(M)$,
\begin{equation*}
    \frac{c(M)}{\text{dist}_M(p,q)^{n+s}}\leq K_{M,s}(p,q)\leq \frac{C(M)}{\text{dist}_M(p,q)^{n+s}}\, ,
\end{equation*}
where $0<c(M)\le C(M)<\infty$. Moreover, it gives that
\begin{equation*}
    K_{M,s}(p,q)\leq C(M) \frac{\alpha_{n,s}}{\text{dist}_M(p,q)^{n+s}}
\end{equation*}
for ${\rm dist}(p,q)\geq R(M)$ as well. To also get a lower bound in this inequality, observe that, by strict positivity of the heat kernel at a fixed time and continuity,
$$
H_M(p,q,t)\geq c(M)>0
$$
for $t\in [1,2]$ and all $p,q$ in $M$, so that if ${\rm dist}(p,q)\geq R(M)$ we can bound
\begin{equation*}
    K_{M,s}(p,q)=\frac{s/2}{\Gamma(1-s/2)} \int_{0}^{\infty} H_M(p,q,t) \frac{dt}{t^{1+s/2}}\geq c(M) \geq\frac{c(M)}{\text{dist}_M(p,q)^{n+s}}\,.
\end{equation*}
Putting everything together, we deduce that
\begin{equation}\label{kerkdtcmp}
    \frac{c(M)}{\text{dist}_M(p,q)^{n+s}}\leq K_{M,s}(p,q)\leq \frac{C(M)}{\text{dist}_M(p,q)^{n+s}}\, ,
\end{equation}
for any $p,q\in M$. This shows that $u\in H^{s/2}(M)$ if and only if $u\in H^{s/2,d}(M)$, and that the associated fractional Sobolev norms are equivalent.

\noindent \textbf{Step 2.} To obtain the inequality in the Lemma, we make our comparison precise.\\
Using Proposition \ref{kermfdconv} once again,
\begin{align*}
\Bigg|\mathcal E_{\ep,s}(u,M)-\mathcal E_{\ep,s}^d(u,M)\Bigg|&=\frac{1}{4}\Bigg|\iint_{M\times M} (u(p)-u(q))^2\Big[K_{M,s}(p,q)-\frac{\alpha_{n,s}}{\text{dist}_M(p,q)^{n+s}}\Big]\Bigg|\\
&\leq \frac{1}{4}\iint_{M\times M} (u(p)-u(q))^2\frac{C_1(M)}{\text{dist}_M(p,q)^{n+s-1}}\,.
\end{align*}
Let $\delta>0$. There is some $R(M,\delta)>0$ be such that if $\text{dist}_M(p,q)\leq R(M,\delta)$, then
$$
\frac{C_1(M)}{\text{dist}_M(p,q)^{n+s-1}}\leq \delta\frac{\alpha_{n,s}}{\text{dist}_M(p,q)^{n+s}}\,.$$
Moreover, if $\text{dist}_M(p,q)\geq R(M,\delta)$ then
$$
\frac{C(M)}{\text{dist}_M(p,q)^{n+s-1}}\leq C(M,\delta)\,.$$
Substituting these bounds, we deduce that
\begin{align*}
|\mathcal E_{\ep,s}(u,M)-\mathcal E_{\ep,s}^d(u,M)|&\leq \delta\frac{1}{4}\iint_{M\times M} (u(p)-u(q))^2\frac{\alpha_{n,s}}{\text{dist}_M(p,q)^{n+s}}+C(M,\delta)\|u\|_{L^\infty(M)}^2\, ,
\end{align*}
which is what we wanted to prove.
\end{proof}
We are ready to give:
\begin{proof}[Proof of Theorem \ref{thmweyl}]
By Proposition \ref{prop:weyld}, which shows the Weyl Law for the $l_{s}^{d}(\p,M)$, it suffices to see that
\begin{equation}\label{equivkernelsp}
    \lim_{\p\to\infty}\big|\p^{-s/n}l_{s}^{d}(\p,M)-\p^{-s/n}l_{s}(\p,M)\big|=0\, .
\end{equation}
For that, let $A\in\tilde{\mathcal{F}}_\p$. Lemma \ref{complem} tells us that this is equivalent to $A$ being in $\tilde{\mathcal{F}}_\p^d$, since the Hausdorff dimension of a set is preserved under equivalent metrics, and moreover that (omitting the dependence on $s_0$)
$$
|\mathcal E_{\ep,s}(u,M)-\mathcal E_{\ep,s}^d(u,M)|\leq C(M,\delta)+\delta \mathcal E_{\ep,s}^d(u,M)\,
$$
for every $u\in A$, since $|u|\leq 1$ by definition of $\tilde{\mathcal{F}}_\p$.\\
In particular,
$$
\mathcal E_{\ep,s}(u,M)\leq C(M,\delta)+(1+\delta) \mathcal E_{\ep,s}^d(u,M)\,,
$$
so that
\begin{align*}
c_{\ep,s}(u,M)&=\inf_{A\in\tilde{\mathcal{F}}_\p}\max_{u\in A} \mathcal E_{\ep,s}(u,M)\leq C(M,\delta)+\inf_{A\in\tilde{\mathcal{F}}_\p}\max_{u\in A} (1+\delta) \mathcal E_{\ep,s}^d(u,M)=C(M,\delta)+(1+\delta)c_{\ep,s}^d(u,M)\,.
\end{align*}
Similarly, we obtain that
$$
(1-\delta)c_{\ep,s}^d(u,M)\leq C(M,\delta)+c_{\ep,s}(u,M)\,.
$$
In other words, we have shown that
$$
|c_{\ep,s}(u,M)-c_{\ep,s}^d(u,M)|\leq C(M,\delta)+\delta c_{\ep,s}^d(u,M)\,;
$$
letting $\ep\to 0$, we get that
\begin{equation}\label{lscomd}
|l_{s}(\p,M)-l_{s}^d(\p,M)|\leq C(M,\delta)+\delta l_{s}^d(\p,M)\, .
\end{equation}
If we now multiply by $\p^{-s/n}$ and take $\limsup$ in $\p$ on both sides, we find that
$$
\limsup_\p|\p^{-s/n}l_{s}(\p,M)-\p^{-s/n}l_{s}^d(\p,M)|\leq \delta \limsup_\p [\p^{-s/n}l_{s}^d(\p,M)]\, .
$$
We know that $\limsup_\p [\p^{-s/n}l_{s}^d(\p,M)]$ is a bounded quantity by \eqref{lpsboundsd}, so letting $\delta\to 0$ the right-hand side goes to $0$ and we conclude the proof.

\end{proof}

\subsection{The Weyl Law in the classical limit $s\to 1$ -- Proof of Theorem \ref{clasweylthm}}
This section proves Theorem \ref{clasweylthm}, which we restate for the reader's convenience:
\begin{theorem}[\textbf{Weyl Law for classical minimal surfaces}] Let $(M,g)$ be a closed manifold of dimension $n$. There exists a universal constant $\tau(n,1)>0$, depending only on $n$, such that
\begin{equation}\label{clasweyl2t}
    \lim_{\p\to\infty} \p^{-\frac{1}{n}}l_{1}(\p,M) = \tau(n,1)\textnormal{vol}(M,g)^{\frac{n-1}{n}}\, .
\end{equation}
\end{theorem}

\begin{proof}
Recall that
\begin{equation*}
l_{1}(\p,M)=\liminf_{s\to 1}[(1-s)l_{s}(\p,M)]\, .
\end{equation*}
In particular, by \eqref{lpsbounds} we deduce the bounds
\begin{equation}\label{lp1bounds}
c(M)\p^{1/n}\leq l_{1}(\p,M)\leq C(M)\p^{1/n}\, .
\end{equation}
We then split the proof into two steps.\\

\textbf{Step 1}.
Define analogously
\begin{equation}\label{deflp11}
    l_{1}^d(\p,M):=\liminf_{s\to 1}[(1-s)l_{s}^d(\p,M)]
\end{equation}
and $\tilde l_{1}^d(\p,M)=\p^{-\frac{1}{n}}l_{1}^d(\p,M)$. We will first prove \eqref{clasweyl} for the quantities $l_{1}^d(\p,M)$, meaning that
\begin{align}\label{clasweyld}
\lim_{p\to\infty}\p^{-\frac{1}{n}} l_{1}^d(\p,M)=\tau(n,1)\textnormal{vol}(M)^{\frac{n-1}{n}}
\end{align}
for some dimensional constant $\tau(n,1)$. To see this, recall that in the proof of the Weyl Law for $s<1$ we found several inequalities for quantities of the type $\p^{-s/n}l_{s}^d(\p,M)$, and then we took $\p\to\infty$ on those inequalities to obtain the results leading to the Weyl Law. Roughly, if we first multiply those inequalities by $(1-s)$ and take $\liminf_{s\to 1}$ instead, and only then send $\p\to\infty$, we will find \eqref{clasweyld}. More precisely, we proceed as follows:
\begin{itemize}
    \item We first obtain the analogue of Lemma \ref{cubelim} for $l_{1}^d(\p,C)$, by multiplying \eqref{cubein} by $(1-s)$ and taking $\liminf_{s\to 1}$, and then finishing the proof with the argument that ensues, up to selecting the sequences $\p_k$ and $\mathfrak q_j$ so that they give the values of the limsup and the liminf of $l_{1}^d(\p,C)$ instead.
    
    This shows that
    \begin{equation}\label{wlawc1}
    \liminf_{\p\to\infty} \tilde l_{1}^d(\p,C) = \limsup_{\p\to\infty} \tilde l_{1}^d(\p,C)
    \end{equation}
    for the unit cube $C$ in $\mathbb{R}^n$, and it allows us to define $\tau(n,1)$ as the common limit.
    
    \item We then prove the analogue for $l_{1}^d(\p,\Omega)$ of Proposition \ref{weyleucd}, meaning that
    \begin{equation}\label{wed1}
    \lim_{\p\to\infty} \tilde l_{1}^d(\p,\Omega)= \tau(n,1)\textnormal{vol}(\Omega)^{\frac{n-1}{n}}
    \end{equation}
    for any Euclidean domain $\Omega$.
    
    To show that
    $$
    \liminf_{\p\to\infty} \tilde l_{1}^d(\p,\Omega)\geq \tau(n,1)\textnormal{vol}(\Omega)^{\frac{n-1}{n}}\, ,
    $$
    we use the first step of the proof of Proposition \ref{weyleucd}: We first multiply \eqref{i1in} by $(1-s)$ and take $\liminf_s$. We then finish the proof by taking $\liminf_p$ on both sides of the resulting inequality, applying \eqref{wlawc1}, and sending $\delta\to 0$.
    
    Likewise, to show that
    $$
    \limsup_{\p\to\infty} \tilde l_{1}^d(\p,\Omega)\leq \tau(n,1)\textnormal{vol}(\Omega)^{\frac{n-1}{n}}\, ,
    $$
    we use the second step of the proof of Proposition \ref{weyleucd}: We multiply \eqref{i2in} by $(1-s)$, take $\liminf_s$, and follow the same arguments of the proof but selecting the subsequence $\mathcal q_k$ so that $\lim_{k\to\infty}\tilde l_{1}^d(\mathcal q_k, \Omega)=\limsup_{p\to\infty}\tilde l_{1}^d(\p,\Omega)$.
    
    \item To prove the Weyl Law in \eqref{clasweyld} for a closed manifold $M$, we prove the analogues of Proposition \ref{infgeq2} and Proposition \ref{supleq2}. They follow from multiplying the inequalities \eqref{i3in} and \eqref{i4in} by $(1-s)$, taking $\liminf_{s\to 1}$, and replicating the few steps that follow them in the proofs (using in particular the Weyl Law \eqref{wed1} for Euclidean domains and $s=1$, which we have just argued to be true, instead of the one for $s<1$). We remark that by \eqref{interpropM} we have $\limsup_{s\to 1}(1-s)\text{Per}_s^d(\mathcal{C}_i)\leq C$, so that the second term on the right hand side of \eqref{i4in} remains bounded when arguing as indicated.
\end{itemize}
This concludes the proof of \eqref{clasweyld}.

\noindent \textbf{Step 2}. To prove the Weyl Law for $l_{1}(\p,M)$ from the one for $l_{1}^d(\p,M)$ in \eqref{clasweyld}, we will actually prove that $l_{1}(\p,M)=l_{1}^d(\p,M)$ for any $\p$, a result which is interesting on its own and which of course gives \eqref{clasweyl2t} from \eqref{clasweyld}. It suffices, then, to show that
\begin{equation}\label{widthlims}
\lim_{s\to 1} |(1-s)l_{s}(\p,M)-(1-s)l_{s}^d(\p,M)|=0\, .
\end{equation}
By \eqref{lscomd}, we already know that
$$
|l_{s}(\p,M)-l_{s}^d(\p,M)|\leq C(M,\delta)+\delta l_{s}^d(\p,M)
$$
for any $\delta>0$.

Multiplying by $(1-s)$ and taking $\limsup$ in $s$ on both sides, we find that
$$
\limsup_{s\to 1} |(1-s)l_{s}(\p,M)-(1-s)l_{s}^d(\p,M)|\leq \delta \limsup_{s\to 1} [(1-s)l_{s}^d(\p,M)]\, .
$$
Since $\limsup_{s\to 1} [(1-s)l_{s}^d(\p,M)]$ is a bounded quantity by \eqref{lpsboundsd}, letting $\delta\to 0$ we obtain the desired result.
\end{proof}

\section{Compactness of $s$-minimal surfaces with bounded index as $s\to 1$}
\subsection{First and second variation formulas}
We obtain for the first time a formula for the second variation of the fractional perimeter on a Riemannian manifold. The Euclidean formula, which was obtained in \cite[Theorem 6.1]{FFMMM}, contains a difference of normal vectors based at different points, which does not have a Riemannian analogue. Moreover, a nonlocal correspondent of the Ricci curvature should implicitly appear. For these reasons, both the formula and its proof need fundamental modifications.
\begin{proposition}\label{12varprop} Let $M$ be a closed Riemannian manifold. Let $E\subset M$ be an open domain with boundary of class $C^2$, and let $\nu_{\partial E}$ denote its outer normal vector. Let $X$ be a vector field of class $C^2$ on $M$, and define $\xi(p)=\langle X,\nu_{\partial E}\rangle_g(p)$ on $\partial E$. Then the following hold:
\begin{itemize}
    \item \textbf{First variation}. The first variation of the fractional perimeter of $E$ is given by
    \begin{align}\label{1stvar}
\frac{d}{dt}\big|_{t=0}\textnormal{Per}_s(\psi_X^t(E))&=\int_{\partial E}dp\, \widetilde H_s[\partial E](p)\xi(p)
\, ,
\end{align}
where the nonlocal mean curvature $\widetilde H_s[\partial E](p)$ is defined as a principal value integral by
\begin{align*}
    \widetilde H_s[\partial E](p)&=-\,{\rm p.v.}\int_M dq\, (\chi_E-\chi_{E^c})(q)K(p,q)\\
    &:=-\lim_{\ep\to 0} \int_M dq\, (\chi_E-\chi_{E^c})(q)K_\ep(p,q)\,.
\end{align*}
Here $K_{\ep}(p,q)$ denotes the regularised (smooth) kernel
\begin{equation}\label{app ker K}
    K_{\ep}(p,q)=\frac{s/2}{\Gamma(1-s/2)} \int_{0}^{\infty} H_M(p,q,t) e^{-\ep^2/4t} \frac{dt}{t^{1+s/2}}\, .
\end{equation}
Other regularisations are possible, for example one can see that
\begin{align*}
    \widetilde H_s[\partial E](p)&=-\lim_{\ep\to 0} \int_{M\setminus B_\ep(p)} dq\, (\chi_E-\chi_{E^c})(q)K(p,q)
\end{align*}
by arguing as in \cite[Proposition 2.5]{FracSobPaper}.\\
\item \textbf{Second variation}. Choose a $C^1$ extension of $\nu_{\partial E}$, the normal vector to $E$, to a vector field (still denoted by $\nu_{\partial E}$) on all of $M$. The second variation of fractional perimeter of $E$ can then be expressed as
\begin{align*}
\frac{d^2}{dt^2}\Big|_{t=0}\textnormal{Per}_s(\psi_X^t(E))=&\iint_{\partial E\times \partial E} (\xi(p)-\xi(q))^2K(p,q)\,dp\,dq\\
&-\int_{\partial E}dp\, \xi(p)^2 \,{\rm p.v.} \int_M dq\, (\chi_E-\chi_{E^c})(q)[\textnormal{div}_q (\nu_{\partial E}(q)K(p,q))+\textnormal{div}_p (\nu_{\partial E}(p)K(p,q))]\\
&-\int_{\partial E}dp\, \widetilde H_s[\partial E](p)\Big[\xi^2 {\rm div}_p(\nu_{\partial E})+\langle \nabla_p^{\rm tan}\xi, X\rangle_g +\xi^2H[\partial E] - \xi{\rm div}_p^\perp X\Big](p)\,.
\end{align*}
The integral over $M$ is to be taken in the principal value sense around $p$, meaning that it is defined as the limit as $\ep$ tends to $0$ of the same expression computed with the smooth kernel \eqref{app ker K} instead. This limit is convergent due to second-order cancellations, as will be seen in the proof.
\end{itemize}
\end{proposition}
\begin{proof}
The proof is divided into two steps.

\noindent \textbf{Step 1.} We compute formally, assuming our kernel (and its derivatives) to be smooth. We will explain how to reduce to this case in Step 2, where we will consider the smooth kernel defined in \eqref{app ker K}.

Define $I_{t}:=\textnormal{Per}_s(\psi_X^t(E))$.

We can write:
\begin{align*}
I_{t+r}&=\int_{\psi^t(\psi^r(E))}dp\,\int_{\psi^t(\psi^r(E^c))}\,dq\,K(p,q)=\int_{\psi^t(E)}dp\,\int_{\psi^t(E^c)}\,dq\,K(\psi^r(p),\psi^r(q))J_r(p)J_r(q)\, .
\end{align*}
Differentiating at $r=0$, using the divergence theorem, and using that $K(p,q)=K(q,p)$, we find
\begin{align*}
\frac{d}{ds}\big|_{r=0}I_{t+r}&=\int_{\psi^t(E)}dp\,\int_{\psi^t(E^c)}\,dq\,[\text{div}_p(K(p,q)X(p))+\text{div}_q(K(p,q)X(q))]\\
&=-\int_{\psi^t(\partial E)}dp\,\langle X,\nu_{\Sigma}\rangle_g(p)\int_M dq\, (\chi_{\psi^t(E)}-\chi_{\psi^t(E^c)})(q)K(p,q)\, .
\end{align*}
Substituting $t=0$, we get the formula for the first variation.\\

To compute the second variation, we want to first differentiate the expression with respect to $t$ instead, and then substitute $t=0$. By the product rule, separating according to which of the two integrals gets differentiated, we can write
\begin{align}\label{eqderts}
\frac{d}{dt}\big|_{t=0}\frac{d}{dr}\big|_{r=0}I_{t+r}=&-\frac{d}{dt}\big|_{t=0}\int_{\psi^t(\partial E)}dp\,\langle X,\nu_{\Sigma}\rangle_g(p)\int_M dq\, (\chi_{E}-\chi_{E^c})(q)K(p,q)\\
&-2\int_{\partial E}dp\,\langle X,\nu_{\Sigma}\rangle_g(p)\int_{\partial E}dq\,\langle X,\nu_{\Sigma}\rangle_g(q)K(p,q) \, . \nonumber
\end{align}
To compute the first term, we proceed as follows. Using the divergence theorem, we can write 
$$
\int_{\psi^t(\partial E)}dp\,\langle X,\nu_{\Sigma}\rangle_g(p)\int_M dq\, (\chi_{E}-\chi_{E^c})(q)K(p,q)=\int_{\psi^t(E)}dp\,\int_M dq\, \text{div}_p(K(p,q)X(p))(\chi_{E}-\chi_{E^c})(q)\, .
$$
The time dependence in the right hand side is only on the domain of integration $\psi^t(E)$, which allows us to compute
\begin{align*}
\frac{d}{dt}\big|_{t=0}\int_{\psi^t(\partial E)}dp\,\langle X,\nu_{\Sigma}\rangle_g(p)\int_M dq\, (\chi_{E}-\chi_{E^c})(q)K(p,q)=\int_{\partial E}dp\,\langle X,\nu_{\Sigma}\rangle_g(p)\int_M dq\, \text{div}_p(K(p,q)X(p))(\chi_{E}-\chi_{E^c})(q)\, .
\end{align*}
The divergence $\text{div}_p(K(p,q)X(p))$ can be split into its tangential and perpendicular parts to $\partial E$. Applying the divergence theorem (on $\partial E$) to the tangential part, we have that
\begin{align*}
\int_{\partial E}dp\,\langle X,\nu_{\Sigma}\rangle_g(p)\int_M dq\, \text{div}_p^{\rm tan}(K(p,q)X(p))(\chi_{E}-\chi_{E^c})(q)&=-\int_{\partial E}dp\,\langle \nabla_p^{\rm tan}\langle X,\nu_{\Sigma}\rangle_g, X\rangle_g(p)\int_M dq\, K(p,q)(\chi_{E}-\chi_{E^c})(q)\\
&-\int_{\partial E}dp\,\langle X,\nu_{\Sigma}\rangle_g^2(p)H[\partial E](p)\int_M dq\, K(p,q)(\chi_{E}-\chi_{E^c})(q)\\
&=\int_{\partial E}dp\,[\langle \nabla_p^{\rm tan}\xi, X\rangle_g+\xi^2H[\partial E]](p)\widetilde H_s[\partial E](p)\,.
\end{align*}
We now focus on the perpendicular part of the divergence. Using $\nabla$ to denote covariant differentiation (unless stated otherwise, $\nabla$ always denotes the Riemannian gradient operator), we can compute 
\begin{align*}
\text{div}_p^\perp (K(p,q)X(p))&=\langle \nu_{\partial E}(p), \nabla_{\nu_{\partial E}(p)} [K(p,q)X(p)]\rangle_g\\
&=\langle\nu_{\partial E}(p),X(p)\rangle_g \langle\nu_{\partial E}(p),\nabla_p K(p,q)\rangle_g+\langle\nu_{\partial E}(p),\nabla_{\nu_{\partial E}(p)} X(p)\rangle_g K(p,q)\\
&=\langle\nu_{\partial E}(p),\nabla_p K(p,q)\rangle_g\xi(p)+[\text{div}_p^\perp X(p)] K(p,q)\, .
\end{align*}
Combining these expressions, we find that
\begin{align*}
\frac{d}{dt}\big|_{t=0}\int_{\psi^t(\partial E)}dp\,\langle X,\nu_{\Sigma}\rangle_g(p)\int_M dq\, (\chi_{E}-\chi_{E^c})(q)K(p,q)&=\int_{\partial E}dp\,\xi^2(p)\int_M dq\,(\chi_{E}-\chi_{E^c})(q) [\nu_{\partial E}(p)\cdot\nabla_p K(p,q)]\\
&\ \ - \int_{\partial E} \Big[\langle \nabla_p^{\rm tan}\xi, X\rangle_g - \xi\text{div}_p^\perp X\Big](p)\int_M dq\, K(p,q)(\chi_{E}-\chi_{E^c})(q)\\
&=\int_{\partial E}dp\,\xi^2(p)\int_M dq\,(\chi_{E}-\chi_{E^c})(q) [\nu_{\partial E}(p)\cdot\nabla_p K(p,q)]\\
&\ \ + \int_{\partial E} \widetilde H_s[\partial E](p)\Big[\langle \nabla_p^{\rm tan}\xi, X\rangle_g +\xi^2H[\partial E] - \xi\text{div}_p^\perp X\Big](p)\, .
\end{align*}
Substituting everything into \eqref{eqderts}, we reach
\begin{align*}
\frac{d}{dt}\big|_{t=0}\frac{d}{dr}\big|_{r=0}I_{t+r}=&-\int_{\partial E}dp\,\xi^2(p)\int_M dq\,(\chi_{E}-\chi_{E^c})(q) (\nu_{\partial E}(p)\cdot\nabla_p K(p,q))\\
&-2\int_{\partial E}dp\,\int_{\partial E}dq\,\xi(p)\xi(q)K(p,q) - \int_{\partial E} \widetilde H_s[\partial E](p)\Big[\langle \nabla_p^{\rm tan}\xi, X\rangle_g +\xi^2H[\partial E] - \xi\text{div}_p^\perp X\Big](p) \, . \nonumber
\end{align*}
Adding and subtracting the quantity
$$
\int_{\partial E}dp\,\int_{\partial E}dq\,(\xi^2(p)+\xi^2(q))K(p,q)=2\int_{\partial E}dp\,\int_{\partial E}dq\,\xi^2(p)K(p,q)\, ,
$$
we get
\begin{align*}
\frac{d}{dt}\big|_{t=0}\frac{d}{dr}\big|_{r=0}I_{t+r}=&\int_{\partial E}dp\,\int_{\partial E}dq\,(\xi(p)-\xi(q))^2K(p,q)-2\int_{\partial E}dp\,\int_{\partial E}dq\,\xi^2(p)K(p,q)\\
&-\int_{\partial E}dp\,\xi^2(p)\int_M dq\,(\chi_{E}-\chi_{E^c})(q) \langle\nu_{\partial E}(p),\nabla_p K(p,q)\rangle_g\\
&- \int_{\partial E} \widetilde H_s[\partial E](p)\Big[\langle \nabla_p^{\rm tan}\xi, X\rangle_g +\xi^2H[\partial E] - \xi\text{div}_p^\perp X\Big](p) \, . \nonumber
\end{align*}
We recognise the first term as the fractional Sobolev term in the statement of the Proposition. On the other hand, we note that, were we not considering $K$ to be smooth for now, the second term in the first line and the second line would both involve divergent integrals.\\

We will now rewrite the second term in a way which hints at the possibility of cancellations. Recall that we are considering a fixed $C^1$ extension of $\nu_{\partial E}$ to all of $M$. Using that for $q\in \partial E$ we can write $K(p,q)=\langle\nu_{\partial E}(q) ,\nu_{\partial E}(q) K(p,q)\rangle_g$, and applying the divergence theorem on $E$ and on $E^c$, we can write
\begin{align*}
-2\int_{\partial E}dp\,\int_{\partial E}dq\,\xi^2(p)K(p,q)&=-2\int_{\partial E}dp\,\xi^2(p)\int_{\partial E}dq\,\langle\nu_{\partial E}(q) ,\nu_{\partial E}(q) K(p,q)\rangle_g\\
&=-\int_{\partial E}dp\,\xi^2(p)\int_{ E}dq\,\text{div}_q(\nu_{\partial E}(q)K(p,q))+\int_{\partial E}dp\,\xi^2(p)\int_{ E^c}dq\,\text{div}_q(\nu_{\partial E}(q)K(p,q))\\
&=\int_{\partial E}dp\,\xi^2(p)\int_{ M}dq\,(\chi_{E^c}-\chi_{E})(q)\text{div}_q(\nu_{\partial E}(q)K(p,q))\,.
\end{align*}
Moreover, the second line in the second variation can be rewritten as
\begin{align*}
&-\int_{\partial E}dp\,\xi^2(p)\int_M dq\,(\chi_{E}-\chi_{E^c})(q) \langle\nu_{\partial E}(p),\nabla_p K(p,q)\rangle_g=\\
&=\int_{\partial E}dp\,\xi^2(p)\text{div}_p(\nu_{\partial E}(p))\int_M dq\,(\chi_{E}-\chi_{E^c})(q) K(p,q)-\int_{\partial E}dp\,\xi^2(p)\int_M dq\,(\chi_{E}-\chi_{E^c})(q) \text{div}_p(\nu_{\partial E}(p) K(p,q))\\
&=-\int_{\partial E}dp\,\xi^2(p)\text{div}_p(\nu_{\partial E}(p))\widetilde H_s[\partial E](p)+\int_{\partial E}dp\,\xi^2(p)\int_M dq\,(\chi_{E^c}-\chi_{E})(q) \text{div}_p(\nu_{\partial E}(p) K(p,q))\, .
\end{align*}
Substituting all of this into the second variation, we finally get
\begin{align*}
\frac{d}{dt}\big|_{t=0}\frac{d}{dr}\big|_{r=0}I_{t+r}=&\iint_{\partial E\times \partial E} (\xi(p)-\xi(q))^2K(p,q)\,dp\,dq\\
&-\int_{\partial E}dp\, \xi(p)^2 \, \int_M dq\, (\chi_E-\chi_{E^c})(q)[\textnormal{div}_q (\nu_{\Sigma}(q)K(p,q))+\textnormal{div}_p (\nu_{\Sigma}(p)K(p,q))]\\
&-\int_{\partial E}dp\, \widetilde H_s[\partial E](p)\Big[\xi^2\text{div}_p(\nu_{\partial E})+\langle \nabla_p^{\rm tan}\xi, X\rangle_g +\xi^2H[\partial E] - \xi\text{div}_p^\perp X\Big](p)\,,
\end{align*}
which is the formula we wanted to prove.

\noindent \textbf{Step 2.} We remove the smoothness assumption on the kernel.

Let $K_\ep(p,q)=\frac{s/2}{\Gamma(1-s/2)} \int_{0}^{\infty} H_M(p,q,t) e^{-\ep^2/4t} \frac{dt}{t^{1+s/2}}$, which is smooth, and which as $\ep\to 0$ increases monotonically to $K(p,q)$ and converges uniformly to $K(p,q)$ away from the diagonal of $M\times M$.

Fix a vector field $X$ on $M$. Applying Step 1 with $K_\ep$, we see that, defining
\begin{equation*}
\textnormal{Per}_s^\ep(F):= 2\int_{F}\int_{F^c} K_{\ep}(p,q)dV_pdV_q\,,
\end{equation*}
and moreover $I_{t}^\ep:=\textnormal{Per}_s^\ep(\psi_X^t(E))$ and $\widetilde H_s^\ep[\partial E](p)=\int_{M}(\chi_E-\chi_{E^c})(q)K_{\ep}(p,q)$,
the formulas
\begin{equation}\label{1varfep}
    \frac{d}{dt}\big|_{t=0}I_{t}^\ep=\int_{\partial E}dp\, \widetilde H_s^\ep[\partial E](p)\xi(p)
\end{equation}
and
\begin{align}
\frac{d^2}{dt^2}\big|_{t=0}I_{t}^\ep=&\iint_{\partial E\times \partial E} (\xi(p)-\xi(q))^2K_\ep(p,q)\,dp\,dq\label{2varfep}\\
&-\int_{\partial E}dp\, \xi(p)^2 \int_M dq\, (\chi_E-\chi_{E^c})(q)[\textnormal{div}_q (\nu_{\Sigma}(q)K_\ep(p,q))+\textnormal{div}_p (\nu_{\Sigma}(p)K_\ep(p,q))] \nonumber\\
&-\int_{\partial E}dp\, \widetilde H_s^\ep[\partial E](p)\Big[\xi^2 {\rm div}_p(\nu_{\partial E})+\langle \nabla_p^{\rm tan}\xi, X\rangle_g +\xi^2H[\partial E] - \xi{\rm div}_p^\perp X\Big](p) \nonumber
\end{align}
hold. We want to argue now by approximation to show that the formulas are true for $K$ as well, as formulated in the statement of the present Proposition.\\

We start by showing that, fixed $t$, we have that
\begin{equation}\label{convItep}
\lim_{\ep\to 0}\frac{d}{dt}I_t^\ep=\frac{d}{dt}I_t \hspace{0.3 cm} \mbox{ and } \hspace{0.3 cm} \lim_{\ep\to 0}\frac{d^2}{dt^2}I_t^\ep=\frac{d^2}{dt^2}I_t\,.
\end{equation}
By monotone convergence, since $K_\ep\nearrow K$ it is already clear that
$\lim_{\ep\to 0}I_t^\ep=I_t
$
for every $t\in\R$.

To conclude \eqref{convItep}, it suffices to show that the $I_t^\ep$, viewed as functions of $t$, are locally uniformly bounded in $C^3$. Indeed, then by Ascoli--Arzel\`a we deduce that for every sequence $\ep_i\to 0$, there is a further subsequence $\ep_{i_j}$ such that $\frac{d^k}{dt^k}I_t^{\ep_{i_j}}$, $k\in\{1,2\}$, converge locally uniformly to some functions, but since $\lim_{\ep\to 0}I_t^\ep=I_t
$
it then easily follows that actually $\lim_{\ep\to 0}\frac{d^k}{dt^k}I_t^\ep=\frac{d^k}{dt^k}I_t$ as desired.\\
It remains then to see that $\frac{d^k}{dt^k}I_t^\ep$ is locally uniformly bounded for $1\leq k \leq 3$. This follows directly from Lemma \ref{enboundslemma2}, applied to a covering of $M$ with small enough balls as given in Remark \ref{flatscalingrmk}; to be precise, Lemma \ref{enboundslemma2} is stated for $I_t$ instead, using the bound in Proposition \ref{sdfgsdrgs} for $K$, but the same hold for $K_\ep$ instead of $K$ with exactly the same proof. With this we conclude the proof of \eqref{convItep}.\\

Now that we know that \eqref{convItep} holds, and in particular for $t=0$, we see that the LHS of \eqref{1varfep} and \eqref{2varfep} converge as $\ep\to 0$ to the desired limits. The arguments which follow will be devoted to showing that each of the terms in the RHS of \eqref{1varfep} and \eqref{2varfep} converge to the analogous ones with $K$ instead of $K_\ep$, and with this we will conclude the proof of Proposition \ref{12varprop}.

We start by showing that
\begin{equation}\label{1varfepconv}
    \int_{\partial E}dp\, \widetilde H_s^\ep[\partial E](p)\xi(p)\xrightarrow{\ep\to 0} \int_{\partial E}dp\, \widetilde H_s[\partial E](p)\xi(p)\,,
\end{equation}
for which it suffices to prove that
\begin{equation}\label{hsepconv}
 \widetilde H_s^\ep[\partial E](p)\xrightarrow{\ep\to 0} \widetilde H_s[\partial E](p) \ \mbox{ uniformly in } p
\end{equation}
to some limit $\widetilde H_s[\partial E](p)$.\\
Now, by Remark \ref{fbsvdg}, the flatness assumptions ${\rm FA}_\ell(M,g,R_0,p,\varphi)$ are satisfied at every $p\in M$, where $\varphi$ denotes the exponential map on $T_p M$. Fix one such $p\in M$, and let $F:=\varphi^{-1}(E)$. From the definition of $\widetilde H_s$, given $\delta>0$ by passing to coordinates we can write 
\begin{align*}
    -\widetilde H_s^\ep[\partial E](p)&=\int_{M}(\chi_E-\chi_{E^c})(q)K_{\ep}(p,q)\\
    &=\int_{\B_\delta}dy\, (\chi_F-\chi_{F^c})(y) K_\ep(0,\varphi(y))\sqrt{|g|(y)}+\int_{M\setminus \varphi(\B_\delta)}(\chi_E-\chi_{E^c})(q)K_{\ep}(p,q)\,,
\end{align*}
and analogously for $\widetilde H_s^\ep[\partial E](p)$.\\
It is clear that, fixed $\delta$,
\begin{align*}
    \int_{M\setminus \varphi(\B_\delta)}(\chi_E-\chi_{E^c})(q)K_{\ep}(p,q)\xrightarrow{\ep\to 0}\int_{M\setminus \varphi(\B_\delta)}(\chi_E-\chi_{E^c})(q)K_{\ep}(p,q)\,.
\end{align*}
Therefore, to conclude that $\widetilde H_s^\ep[\partial E](p)\to \widetilde H_s[\partial E](p)$ it suffices to show that $\int_{\B_\delta}dy\, (\chi_F-\chi_{F^c})(y)K_\ep(0,y)\sqrt{|g|(y)}$ converges to zero as $\delta\to 0$, uniformly in $\ep$. Thanks to having chosen normal coordinates we have $g=Id+O(|x|^2)$, which combined with \eqref{remaining0} (see also Lemma \ref{loccomparability}) shows that
\begin{align*}
   \int_{\B_\delta}dy\, (\chi_F-\chi_{F^c})(y)K_\ep(0,\varphi(y))\sqrt{|g|(y)}&=\int_{\B_\delta}dy\, (\chi_F-\chi_{F^c})(y)K_\ep(0,\varphi(y))+O\Big(\int_{\B_\delta}dy\,K_\ep(0,\varphi(y))|y|^2\Big)\\
   &=\int_{\B_\delta}dy\, (\chi_F-\chi_{F^c})(y)\frac{\alpha_{n,s}}{|y|^{n+s}}+O\Big(\int_{\B_\delta}dy\,\frac{1}{|y|^{n+s-1}}\Big)\\
   &=\int_{\B_\delta}dy\, (\chi_F-\chi_{F^c})(y)\frac{\alpha_{n,s}}{|y|^{n+s}}+O(\delta^{1-s})\,.
\end{align*}
Up to a rotation, we can assume that $\partial F$ is tangent to $\{x_n=0\}$ at $x=0$. Since $\partial E$ is a compact $C^2$ hypersurface, there is some $r>0$ (independent of $p$) such that we can write $\partial F$ around $0$ as the graph of a function $g:\B_\delta'\to (-\delta,\delta)$, satisfying moreover $g(x')=h(x')|x'|^2$ for some continuous function $h$. Let $M:=\max |h|$, which is again bounded independently of $p$. Letting $H=\chi_{\{y_n<0\}}$ we see that $F\Delta H\subset \{|y_n|\leq M|y'|^2\}$, and likewise for $F^c\Delta H^c$. Moreover, observe that (by symmetry)
$$\int_{\B_\delta}dy\, (\chi_H-\chi_{H^c})(y)\frac{\alpha_{n,s}}{|y|^{n+s}}=0\,.$$
We can then write
\begin{align*}
   \int_{\B_\delta}dy\, (\chi_F-\chi_{F^c})(y)\frac{\alpha_{n,s}}{|y|^{n+s}}&=\int_{\B_\delta}dy\, \big[(\chi_F-\chi_{F^c})-(\chi_H-\chi_{H^c})\big](y)\frac{\alpha_{n,s}}{|y|^{n+s}}\\
   &=O\Big(\int_{\B_\delta\cap (F\Delta H)}dy\,\frac{\alpha_{n,s}}{|y|^{n+s}}\Big)=O\Big(\int_{\B_\delta\cap \{|y_n|\leq M|y'|^2\}}dy\,\frac{\alpha_{n,s}}{|y|^{n+s}}\Big)\,.
\end{align*}
Passing to polar coordinates it is immediate to see that the last integral is bounded by $C\int_0^{\delta} dr\,\frac{1}{r^{s}}=C\delta^{1-s}$. With this, as mentioned above, we conclude our proof of \eqref{hsepconv}. In particular, we deduce \eqref{1varfepconv}.\\

We now focus on the terms in the RHS of \eqref{2varfep}. Thanks to \eqref{hsepconv}, we immediately see that
\begin{align*}
    &\int_{\partial E} dp\,\widetilde H_s^\ep[\partial E](p)\Big[\xi^2 {\rm div}_p(\nu_{\partial E})+\langle \nabla_p^{\rm tan}\xi, X\rangle_g +\xi^2H[\partial E] - \xi{\rm div}_p^\perp X\Big](p)\\
    &\hspace{1cm}\xrightarrow{\ep\to 0} \int_{\partial E} dp\,\widetilde H_s[\partial E](p)\Big[\xi^2 {\rm div}_p(\nu_{\partial E})+\langle \nabla_p^{\rm tan}\xi, X\rangle_g +\xi^2H[\partial E] - \xi{\rm div}_p^\perp X\Big](p)\,,
\end{align*}
Moreover, it is clear that
$$
\int_{\partial E}dp\,\int_{\partial E}dq\,(\xi(p)-\xi(q))^2K_\ep(p,q)\xrightarrow{\ep\to 0} \int_{\partial E}dp\,\int_{\partial E}dq\,(\xi(p)-\xi(q))^2K(p,q)\,,$$
simply by monotone convergence.\\
It remains to see that the second term in \eqref{2varfep} converges as $\ep\to 0$. Recall that $\nu_{\partial E}$ denotes a $C^1$ extension to $M$ of the outer normal vector to $E$. We start by applying the Leibniz (product) rule for the divergences appearing in the integral, by which we can expand:
\begin{align}
    &\int_M dq\,(\chi_{E}-\chi_{E^c})(q)[\text{div}_p(\nu_{\partial E}(p) K_\ep(p,q))+\text{div}_q(\nu_{\partial E}(q)K_\ep(p,q))]= \label{divin2vt}\\
    &\hspace{0.75cm}=\text{div}_p(\nu_{\partial E}(p))\int_M dq\,(\chi_{E}-\chi_{E^c})(q) K_\ep(p,q)+\int_M dq\,(\chi_{E}-\chi_{E^c})(q)\text{div}_q(\nu_{\partial E}(q))K_\ep(p,q) \nonumber\\
    &\hspace{1.25cm}+\int_M dq\,(\chi_{E}-\chi_{E^c})(q)\langle\nu_{\partial E}(p), \nabla_p K_\ep(p,q)\rangle_g+\int_M dq\,(\chi_{E}-\chi_{E^c})(q)\langle\nu_{\partial E}(q),\nabla_q K_\ep(p,q)\rangle_g\,. \nonumber
\end{align}
Thanks to the smoothness of $\nu_{\partial E}(p)$, we can Taylor--expand $\text{div}_q(\nu_{\partial E}(q))=\text{div}_p(\nu_{\partial E}(p))+h(p,q)$, with $h(p,q)= O({\rm dist}(p,q))$. We can then write the first two terms on the right hand side as
\begin{align*}
    2\text{div}_p(\nu_{\partial E}(p))\int_M dq\,(\chi_{E}-\chi_{E^c})(q) K_\ep(p,q)+\int_M dq\,(\chi_{E}-\chi_{E^c})(q)h(p,q)K_\ep(p,q)\,.
\end{align*}
The first integral (which corresponds to $\widetilde H_s^\ep(p)$) converges as $\ep\to 0$ by \eqref{hsepconv}. The second integral can be shown to converge simply by dominated convergence (without the need for cancellations), since combining Lemma \ref{loccomparability} (see also Proposition \ref{kermfdconv}) and the fact that $h(p,q)= O({\rm dist}(p,q))$ we can bound $h(p,q)K_\ep(p,q)\leq \frac{C}{{\rm dist}^{n+s-1}(p,q)}$, and the latter is Lebesgue integrable in $q$.\\

We now focus on the last two terms on the RHS of \eqref{divin2vt}. These are the most challenging ones, since it is not obvious at all at first sight how they can be made to converge as $\ep\to 0$ (observe that $|\nabla_p K(p,q)|$ is of order $\frac{1}{{\rm dist}^{n+s+1}(p,q)}$, which is $1+s$ orders away from being Lebesgue integrable), and necessarily one needs to exploit some cancellations. In fact, the two terms need to be considered together, i.e. we write them as $$\int_M dq\,(\chi_{E}-\chi_{E^c})(q)[\nu_{\partial E}(p)\nabla_p K_\ep(p,q)+\nu_{\partial E}(q)\nabla_q K_\ep(p,q)]\,.$$
Let $\psi_{\nu_{\partial E}}^t$ denote the flow of the extended vector field $\nu_{\partial E}$, which is defined on all of $M$, at time $t$. Note that this is not related in any way to $\psi_X^t$, the flow of the vector field $X$ that we fixed at the beginning, and which does not appear at all in the integral we are considering. The first crucial observation is that we can actually write
\begin{align*}
\nu_{\partial E}(p)\nabla_p K(p,q)+\nu_{\partial E}(q)\nabla_q K(p,q)&=\frac{d}{dt}\big|_{t=0} [K(\psi_{\nu_{\partial E}}^t(p),q)+K(p,\psi_{\nu_{\partial E}}^t(q))]\\
&=\frac{d}{dt}\big|_{t=0} K(\psi_{\nu_{\partial E}}^t(p),\psi_{\nu_{\partial E}}^t(q))\\
&=:K'(p,q)\, .
\end{align*}
This new kernel $K'(p,q)$ can be bounded by $K(p,q)$, thanks to \eqref{cdsacsdc1} of Proposition \ref{sdfgsdrgs}. Likewise, defining $K'_\ep(p,q)=\frac{d}{dt}\big|_{t=0} K_\ep(\psi_{\nu_{\partial E}}^t(p),\psi_{\nu_{\partial E}}^t(q))$, we get that it can also be bounded by $K(p,q)$: Proposition \ref{sdfgsdrgs} is stated for $K$, but its proof works exactly the same for $K_\ep$. We then need to make sense of 
$\int_M dq\,(\chi_{E}-\chi_{E^c})(q)K'(p,q)$ as a singular integral, and for that, the second crucial observation is that $K'_\ep$ is actually symmetric at first order:\\

\noindent \textbf{Claim.} Denoting $K'_\ep(x,y):=K'_\ep(\varphi(x),\varphi(y))$, we have that $|K'_\ep(x,x+z)-K'_\ep(x,x-z)|\leq \frac{C}{|z|^{n+s-1}}$.
\begin{proof}[Proof of the claim]
    Since $K'_\ep(p,q)=K'_\ep(q,p)$, we can rewrite this as $|K'_\ep(x,x+z)-K'_\ep(x-z,x)|\leq \frac{C}{|z|^{n+s-1}}$. Equivalently, defining the constant vector field $v=\frac{z}{|z|}$, this corresponds to $|K'_\ep(x-z+|z|v,x+|z|v)-K'_\ep(x-z,x)|\leq \frac{C}{|z|^{n+s-1}}$. Consider the vector field $V=v\eta$, where $\eta$ is a cutoff with small support around $x$, and let $\psi_V^t$ denote the flow of $V$ at time $t$; then, $(x-z+|z|v,x+|z|v)=(\psi_V^{|z|}(x-z),\psi_V^{|z|}(x))$, so that our desired inequality can be rewritten as $|K'_\ep(\psi_V^{|z|}(x-z),\psi_V^{|z|}(x))-K'_\ep(x-z,x)|\leq \frac{C}{|z|^{n+s-1}}$. By the fundamental theorem of calculus, the LHS is \begin{align*}
    \Big|K'_\ep(\psi_V^{|z|}(x-z),\psi_V^{|z|}(x))-K'_\ep(x-z,x)\Big|&=\Big|\int_0^{|z|}dt\, \frac{d}{dt}K'_\ep(\psi_V^{t}(x-z),\psi_V^{t}(x))\Big|\\
    &\leq \int_0^{|z|}dt\, \Big|\frac{d}{dt}K'_\ep(\psi_V^{t}(x-z),\psi_V^{t}(x))\Big|\\
    &\leq C|z|\sup_{0\leq t \leq |z|} \Big|\frac{d}{dt}K'_\ep(\psi_V^{t}(x-z),\psi_V^{t}(x))\Big|\,.
\end{align*}
Finally, arguing as in the proof of Proposition \ref{sdfgsdrgs}, which corresponds to \cite[Proposition 3.9]{FracSobPaper} (recall that $K'_\ep(p,q)=\frac{d}{dt}\big|_{t=0} K_\ep(\psi_{\nu_{\partial E}}^t(p),\psi_{\nu_{\partial E}}^t(q))$), we see that we can bound
$$
\sup_{0\leq t \leq |z|} \Big|\frac{d}{dt}K'_\ep(\psi_V^{t}(x-z),\psi_V^{t}(x))\Big|\leq CK(x-z,x)\,,$$
which after applying Lemma \ref{loccomparability} concludes the proof of the claim.
\end{proof}
With this claim at hand, the proof of the convergence of $$\int_M dq\,(\chi_{E}-\chi_{E^c})(q)[\nu_{\partial E}(p)\nabla_p K_\ep(p,q)+\nu_{\partial E}(q)\nabla_q K_\ep(p,q)]=\int_M dq\,(\chi_{E}-\chi_{E^c})(q)K'_\ep(p,q)$$ to
$$
\int_M dq\,(\chi_{E}-\chi_{E^c})(q)K'(p,q)
$$
follows like the proof of the convergence of $\widetilde H_s^\ep[\partial E](p)=\int_M dq\,(\chi_{E}-\chi_{E^c})(q)K_\ep(p,q)$ to $\widetilde H_s[\partial E](p)$. The only difference is that, in place of the full cancellation property
$$\int_{\B_\delta}dy\, (\chi_H-\chi_{H^c})(y)\frac{\alpha_{n,s}}{|y|^{n+s}}=0\,,$$
one needs to estimate
$$\int_{\B_\delta}dy\, (\chi_H-\chi_{H^c})(y)K'_\ep(0,y)$$
by some function of $\delta$ which goes to zero as $\delta\to 0$. To see that the latter holds, using the symmetry of the Lebesgue measure and the claim, we compute
\begin{align*}
    \int_{\B_\delta}dy\, \chi_H(y)K'_\ep(0,y)&=\int_{\B_\delta}dy\, \chi_H(-y)K'_\ep(0,-y)=\int_{\B_\delta}dy\, \chi_{H^c}(y)K'_\ep(0,-y)\\
    &=\int_{\B_\delta}dy\, \chi_{H^c}(y)K'_\ep(0,y)+O\Big(\int_{\B_\delta}dy\, |K'_\ep(0,-y)-K'_\ep(0,y)|\Big)\\
    &=\int_{\B_\delta}dy\, \chi_{H^c}(y)K'_\ep(0,y)+O\Big(\int_{\B_\delta}dy\, \frac{1}{|y|^{n+s-1}}\Big)\\
    &=\int_{\B_\delta}dy\, \chi_{H^c}(y)K'_\ep(0,y)+O\Big(\delta^{1-s}\Big)\,.
\end{align*}
Bringing the first term on the RHS to the LHS, we see that
$$\int_{\B_\delta}dy\, (\chi_H-\chi_{H^c})(y)K'_\ep(0,y)=O\Big(\delta^{1-s}\Big)$$
as desired.
\end{proof}

The formulas above motivate the following definitions.
\begin{definition}\label{1vardef} For a set $E$ as in Proposition \ref{12varprop}, as indicated there the fractional mean curvature $\widetilde H_s[\partial E](p)$ at a point $p\in\partial E$ is defined as
\begin{equation*}
    \widetilde H_s[\partial E](p)=-\,{\rm p.v.} \int_M dq\, (\chi_{E}-\chi_{E^c})(q)K_s(p,q)\, .
\end{equation*}
In particular, by arbitrariness of $\xi$ in \eqref{1stvar}, $E$ is an $s$-minimal surface (recall Definition \ref{def-fracminsurface}) iff $\widetilde H_s[\partial E](p)=0$ for every $p\in\partial E$.\\
Denoting $\sigma:=1-s$, we also set
\begin{equation*}
    H_s[\partial E](p) := \sigma \widetilde H_s[\partial E](p)\, ,
\end{equation*}
with which we will work in the rest of the article.
Then, up to a multiplicative dimensional constant, $H_s[\partial E]$ converges as $s\to 1$ to the classical mean curvature $H[\partial E]$, as we will see later.
\end{definition}

\begin{definition}\label{2quaddef} The second variation $\delta^2 \textnormal{Per}_s(E)$ of an $s$-minimal surface $E$ is defined as the quadratic form acting on test functions $\xi\in C^1(\partial E)$ by
    \begin{align}
        \delta^2 \textnormal{Per}_s(E)[\xi,\xi]=\iint_{\partial E\times\partial E}&|\xi(p)-\xi(q)|^2K(p,q)\\
        &-\int_{\partial E}\,dp\,\xi^2(p)\,{\rm p.v.}\int_M\,dq\,(\chi_E-\chi_{E^c})(q)\Big[\textnormal{div}_p(\nu_{\partial E}(p)K(p,q))+\textnormal{div}_q(\nu_{\partial E}(q)K(p,q))\Big]\, .
    \end{align}
    If $X$ is a $C^1$ vector field on $M$, set $\xi=X\cdot\nu$ on $\partial E$. By Proposition \ref{12varprop}, since $\widetilde H_s[\partial E](p)=0$ we then have that $\frac{d^2}{dt^2}\Big|_{t=0}\textnormal{Per}_s(\psi_X^t(E))=\delta^2 \textnormal{Per}_s(E)[\xi,\xi]$.
\end{definition}

\subsection{A first classical perimeter estimate for $s$-minimal surfaces with bounded index}
In this section, we will prove the following classical perimeter estimate in all dimensions:
\begin{theorem}[\textbf{Perimeter estimate}]\label{BVest}
Let $p_0 \in M$, $s_0 \in (0,1)$, $s \in (s_0,1)$ and assume that $M$ satisfies the flatness assumption ${\rm FA}_2(M,g,1,p_0,\varphi)$. Let $E$ be a $C^2$ $s$-minimal surface with index bounded by $m$ in $\varphi(\B_1)$. 

\vspace{1pt}
Then, there exists a constant $ C=C(n,s_0,m)$ such that
\begin{equation}\label{bvunofin}
\textnormal{Per}(E;\varphi(\B_{1/2})) \leq \frac{C}{1-s} \,.
\end{equation}
If moreover we add the assumption that
\begin{equation}\label{bvfin}
    \textnormal{Per}_s(E;\varphi(\B_1))\leq \frac{\kappa}{1-s}=\frac{\kappa}{\sigma}
\end{equation}
for some $\kappa>0$, then we get the improved estimate
\begin{equation}\label{bvufin}
\textnormal{Per}(E;\varphi(\B_{1/2})) \leq C\frac{1+\sqrt \kappa}{\sqrt{\sigma}} \,.
\end{equation} 
\end{theorem}
\begin{remark}
    The proof is similar to the one in \cite{CFS}, which deals with a BV estimate for solutions of Allen--Cahn with bounded index.
\end{remark}
\begin{remark}
    In the case where $\textnormal{Per}_s(E;\varphi(\B_1))\leq \frac{\kappa}{\sigma}$, for $n=3$ we will eventually show (see Proposition \ref{LocBoundFMI}) that $\textnormal{Per}(E;\varphi(\B_{1/2}))\leq C(\kappa)$, which is what one naturally expects: Recall that, at least for a fixed smooth set $E$, $\displaystyle \lim_{s\to 1} (1-s){\rm Per}_s(E)= \gamma_n{\rm Per}(E)$, and a similar convergence will be obtained in the present article for sequences $E_k$ with bounded index. Nevertheless, the non-optimal perimeter bound in Theorem \ref{BVest} will prove very helpful in the proof of the sharp result.
\end{remark}
Before giving the proof of Theorem \ref{BVest}, we need some preliminary results that will be repeatedly used during the article. The first one is a lemma which deals with the finite index property in the nonlocal case. For critical points of local functionals, it is a standard fact that having Morse index bounded by $m$ implies stability in one out of every $m+1$ disjoint open sets. For nonlocal functionals this is not true anymore, but in one of the sets we can obtain a weaker, quantitative lower bound on the second derivative which we will refer to as \textit{almost stability}.
\begin{definition}[\textbf{Almost stability}]\label{almoststab}
Let $\Omega\subset M$ open and $E\subset M$ be an $s$-minimal surface in $\Omega$. Given $\Lambda\in\R$, we say that $E$ is \textit{$\Lambda$-almost stable in $\Omega$} if 
\begin{equation}\label{lamasineq}
    \delta^2\textnormal{Per}_s(E)[\xi,\xi] \ge -\Lambda \Big(\int_{\partial E\cap \Omega} |\xi|\Big)^2  \,\quad \forall \, \xi \in C_c^{1}(\partial E\cap \Omega) \,.
\end{equation}
\end{definition}

\begin{lemma}[\textbf{Finite Morse index and almost stability}]\label{asineq}
Let $\Omega\subset M$ open and let $E\subset M$ be an $s$-minimal surface in $\Omega$ with Morse index at most $m$. Consider a collection $\mathcal{U}_1,\dotsc, \mathcal{U}_{m+1}$ of $(m+1)$ disjoint open subsets of $\Omega$, and set
\begin{equation*}
\Lambda := m\max_{i\neq j} \sup_{\mathcal{U}_i\times \mathcal{U}_j} K(p,q) \,.    
\end{equation*}
Then, there is (at least) one set $\mathcal{U}_k$ among $\mathcal{U}_1, \dotsc , \mathcal{U}_{m+1}$ such that $E$ is \textit{$\Lambda$-almost stable in $\mathcal{U}_k$}, that is 
\begin{equation*}
    \delta^2\textnormal{Per}_s(E)[\xi,\xi] \ge -\Lambda \Big(\int_{\partial E\cap \mathcal{U}_k} |\xi|\Big)^2  \,\quad \forall \, \xi \in C_c^{1}(\partial E\cap \mathcal{U}_k) \,.
\end{equation*}
\end{lemma}

\begin{proof}
We prove the Lemma just for $m=1$ for the sake of clarity, the proof goes on exactly the same for general $m$. Let $\xi_1\in C_c^{\infty}(\partial E\cap\mathcal{U}_1)$ and $\xi_2\in C_c^{\infty}(\partial E\cap\mathcal{U}_2)$. Testing the second variation of the fractional perimeter, the expression of which we gave in Definition \ref{2quaddef}, with linear combinations of $\xi_1$ and $\xi_2$ gives (by expanding the square in the Sobolev part of the expression)
\begin{align*}
   \delta^2\textnormal{Per}_s(E)[a\xi_1+b\xi_2,a\xi_1+b\xi_2] =\,
    & a^2\delta^2\textnormal{Per}_s(E)[\xi_1,\xi_1] +b^2\delta^2\textnormal{Per}_s(E)[\xi_{2},\xi_{2}]\\
    &\hspace{1cm}-2ab\iint_{(\partial E\cap\mathcal{U}_1) \times (\partial E\cap\mathcal{U}_2)} \xi_1(p)\xi_2(q) K(p,q) \,.
\end{align*}
Since by assumption $K(p,q) \le \Lambda$ for all $(p,q)\in \mathcal{U}_1\times \mathcal{U}_2$, the interaction term can be bounded as
\begin{equation*}
\begin{split}
   - 2ab\iint_{(\partial E\cap\mathcal{U}_1) \times (\partial E\cap\mathcal{U}_2)}\xi_1(p)\xi_2(q) K(p,q) & \leq 2ab \Lambda \|\xi_1\|_{L^1(\partial E\cap\mathcal{U}_1)}\|\xi_2\|_{L^1(\partial E\cap\mathcal{U}_2)} \\
     &\leq a^2 \Lambda\|\xi_1\|_{L^1(\partial E\cap\mathcal{U}_1)}^2+b^2 \Lambda\|\xi_2\|_{L^1(\partial E\cap\mathcal{U}_2)}^2 . 
\end{split}
\end{equation*}
Hence
\begin{align}
    \delta^2\textnormal{Per}_s(E)[a\xi_1+b\xi_2,a\xi_1+b\xi_2]  &\le a^2 \Big( \underbrace{\delta^2\textnormal{Per}_s(E)[\xi_1,\xi_1]+\Lambda\|\xi_1\|_{L^1(\partial E\cap\mathcal{U}_1)}^2}_{=:F_1(\xi_1)}\Big)\nonumber\\
    &\hspace{1cm}+b^2\Big( \underbrace{\delta^2\textnormal{Per}_s(E)[\xi_2,\xi_2]+\Lambda\|\xi_2\|_{L^1(\partial E\cap\mathcal{U}_2)}^2}_{=:F_2(\xi_2)}\Big) \label{fmias}.
\end{align}
We want to show that either $F_1(\xi_1) \ge 0$ for all $\xi_1 \in C_c^{\infty}(\partial E\cap\mathcal{U}_1) $ or $F_2(\xi_2) \ge 0$ for all $\xi_2 \in C_c^{\infty}(\partial E\cap\mathcal{U}_2) $. Suppose neither of these two held, then there would exist $\xi_1$ and $\xi_2$ such that $F_1(\xi_1)<0$ and $F_2(\xi_2)<0$. This would imply, however, that \eqref{fmias} is negative for all $(a,b)\neq (0,0)$, thus contradicting that the Morse index of $u$ is at most one.
\end{proof}
\begin{remark}\label{rmk:asresc}
    Almost-stability improves upon zooming in. More precisely, if $E$ is a $\Lambda$-almost stable $s$-minimal set in $\Omega\subset M$ for $(M,g)$, it is a $(r^{n+s}\Lambda)$-almost stable $s$-minimal set in $\Omega$ for $(M,\frac{1}{r^2} g)$. This follows from (a) in Remark \ref{flatscalingrmk}, since it immediately gives that the $s$-perimeter of any set just rescales by a constant factor $(\frac{1}{r})^{n-s}$ under the scaling of the metric, which shows that the scaling preserves the criticality/$s$-minimality condition and it transforms (as seen by rescaling both sides of equation \eqref{lamasineq}) $\Lambda$-almost stability into $(r^{n+s}\Lambda)$-almost stability.\\
\end{remark}
We first prove Theorem \ref{BVest} in the almost-stable case.
\begin{proposition}[\textbf{Almost-stable case}]\label{prop:BVestas}
Let $p_0 \in M$, $s_0 \in (0,1)$, $s \in (s_0,1)$ and assume that $M$ satisfies the flatness assumption ${\rm FA}_2(M,g,1,p_0,\varphi)$. Let $E$ be a $C^2$ $\Lambda$-almost stable $s$-minimal surface in $\varphi(\B_1)$. 

\vspace{1pt}
Then, there exists a constant $ C=C(n,s_0,\Lambda)$ such that
\begin{equation}\label{bvunoas}
\textnormal{Per}(E;\varphi(\B_{1/2})) \leq \frac{C}{1-s} \,.
\end{equation}
If moreover we add the assumption that
\begin{equation}\label{bvfas}
    \textnormal{Per}_s(E;\varphi(\B_1))\leq \frac{\kappa}{1-s}=\frac{\kappa}{\sigma}
\end{equation}
for some $\kappa>0$, then we get the improved estimate
\begin{equation}\label{bvufas}
\textnormal{Per}(E;\varphi(\B_{1/2})) \leq C\frac{1+\sqrt \kappa}{\sqrt{\sigma}} \,.
\end{equation} 
\end{proposition}
\begin{proof}
    \textbf{Step 1.} We first prove that there exists a $\Lambda_0>0$, depending only on $n$ and $s_0$, such that if $E$ is $\Lambda$-almost stable in $ \varphi(\B_1)\subset M$ (see Definition \ref{almoststab}) with $\Lambda \le \Lambda_0$, then the Proposition is true.\\

The proof goes as follows. 
 Let $X$ be a $C^1$ vector field, compactly supported on $\varphi(\B_1)$, and define $\xi=\langle X,\nu_{\partial E}\rangle_g$ on $\partial E$. Considering the test functions $\xi$ and $|\xi|$ for the second variation formula (recall Definition \ref{2quaddef}), and applying the $\Lambda$-almost stability assumption to bound $\delta^2\text{Per}_s(E)[|\xi|,|\xi|]$ from below appropriately (see Definition \ref{almoststab}, and observe that it can be extended to Lipschitz functions such as $|\xi|$  by standard approximation arguments), we get
    \begin{equation}\label{diff2var}
       \delta^2\text{Per}_s(E)[\xi,\xi]-\delta^2\text{Per}_s(E)[|\xi|,|\xi|]\leq \delta^2\text{Per}_s(E)[\xi,\xi]+\Lambda\Big(\int_{\partial E} |\xi|\Big)^2\, . 
    \end{equation}

    On the other hand, both second variations in the left hand side have the same ``local parts'', since obviously $\xi^2=|\xi|^2$. Also for this reason, looking at the expression in Definition \ref{2quaddef}, one also sees that the difference of the terms of the form $\iint_{\partial E\times\partial E}$ has a very simple form, since only the cross-terms appearing after expanding the squares do not cancel out. Precisely, we obtain that
    \begin{equation}\label{diff2var2}
    \delta^2\text{Per}_s(E)[\xi,\xi]-\delta^2\text{Per}_s(E)[|\xi|,|\xi|]=2\iint_{\partial E\times\partial E} \Big[|\xi(p)||\xi(q)|-\xi(p)\xi(q)\Big]K(p,q)\, .
    \end{equation}
    Using the identity $|a||b|-ab=2a_+b_-+2a_-b_+$, which holds for real numbers $a,b$ and their positive and negative parts, together with the symmetry of $K$ we find that
    \begin{equation*}
    \delta^2\text{Per}_s(E)[\xi,\xi]-\delta^2\text{Per}_s(E)[|\xi|,|\xi|]=8\iint_{\partial E\times\partial E} \xi_+(p)\xi_-(q)K(p,q)\, .
    \end{equation*}
    Thanks to Proposition \ref{loccomparability}, we can bound $K(p,q)\geq c>0$ for $(p,q)\in\varphi(\B_1)\times\varphi(\B_1)$. Applying this to \eqref{diff2var2}, we then find that
    $$
    c\Big(\int_{\partial E} dp\,\xi_+(p)\Big)\Big(\int_{\partial E}dq\,\xi_-(q)\Big)\leq\delta^2\text{Per}_s(E)[\xi,\xi]-\delta^2\text{Per}_s(E)[|\xi|,|\xi|]\, .
    $$
    Recall now that $\xi=X\cdot\nu$, where $X$ is a $C^1$ vector field supported on $\varphi(\B_1)$. By the divergence theorem, we then have that 
    $$\Big|\int_{\partial E} \xi(p)\,dp\Big|=\Big|\int_{\partial E} X\cdot \nu(p)\,dp\Big|=\Big|\int_E \textnormal{div}(X)\Big|\leq |B_2|C_X\, .$$
    Therefore, putting together all the above, we can bound
    \begin{align*}
        \Big(\int_{\partial E}\Big|\xi\Big|\Big)^2&=\Big(\int_{\partial E}(X\cdot\nu)_++\int_{\partial E}(X\cdot\nu)_-\Big)^2=\Big(\int_{\partial E}(X\cdot\nu)_+-\int_{\partial E}(X\cdot\nu)_-\Big)^2+4\Big(\int_{\partial E}(X\cdot\nu)_+\Big)\Big(\int_{\partial E}(X\cdot\nu)_-\Big)\\
        &=\Big(\int_{\partial E}X\cdot\nu\Big)^2+4\Big(\int_{\partial E}(X\cdot\nu)_+\Big)\Big(\int_{\partial E}(X\cdot\nu)_-\Big)\leq C+C\delta^2\text{Per}_s(E)[\xi,\xi]+C\Lambda\Big(\int_{\partial E} |\xi|\Big)^2\, .
    \end{align*}
    In case $\Lambda_0$ is small enough (say, so that $C\Lambda_0\leq 1/2$), we can absorb the last term of the inequality on the LHS, which after taking square roots in the expression gives
    \begin{align*}
        \int_{\partial E}\Big|X\cdot\nu\Big|=\int_{\partial E}\Big|\xi\Big|&\leq \Big(C+C\delta^2\text{Per}_s(E)[\xi,\xi]\Big)^{1/2}\, .
    \end{align*}
    We now make the precise choice of $X$. We consider $X_1=\varphi_*(e_1\eta)$, $X_2=\varphi_*(e_2\eta)$, ... , up to $X_n=\varphi_*(e_n\eta)$, where $e_i$ is the $i$-th basis coordinate vector in $\R^n$ and $\chi_{\B_{1/2}}\leq\eta\leq\chi_{\B_{3/4}}$ is a standard cutoff. We denote then $\xi_i=\langle X_i,\nu_{\partial E}\rangle_g$. Adding up the inequalities that we obtain for these choices of $X$, and since the $e_1,...,e_n$ form an orthonormal basis of vectors for $\R^n$, together with the flatness assumptions we obtain the perimeter in the LHS:
    \begin{align*}
        \text{Per}(E,\varphi(\B_{1/2}))&\leq \sum_i\Big(C+C\delta^2\text{Per}_s(E)[\xi_i,\xi_i]\Big)^{1/2}=\sum_i\Big(C+C\frac{d^2}{dt^2}\Big|_{t=0}\text{Per}_s(\psi_{X_i}^t(E))\Big)^{1/2}\, .
    \end{align*}
    On the other hand, applying Lemma \ref{enboundslemma2} we can get rid of the second derivatives on the RHS, obtaining
    \begin{align*}
        \text{Per}(E,\varphi(\B_{1/2}))&\leq \Big(C+C\text{Per}_s(E, \varphi(\B_1))\Big)^{1/2}\, .
    \end{align*}
    Under the assumption \eqref{bvfas}, i.e. assuming a constant bound for the fractional perimeter $\text{Per}_s(E)$, we conclude \eqref{bvufas} as desired.
    
    The former is the only case we will need in this article. If, on the other hand, we did not assume \eqref{bvfas} and wished to obtain \eqref{bvunoas} instead, we would need to continue the argument as in the proof of \cite[Proposition 3.14]{CFS}. First, we would use the interpolation result in \eqref{interpropRn} and the inequality $ab\leq \frac{a^2}{\delta}+\delta b^2$, valid for any $\delta>0$, to get
    \begin{align*}
        \text{Per}(E,\varphi(\B_{1/2}))&\leq \frac{C}{(1-s)\delta}+C\delta\text{Per}(E, \varphi(\B_1))\, .
    \end{align*}
    Rescaling and translating this inequality, see Step 2 next for full details in a similar argument, we would find that
    \begin{align*}
        r^{1-n}\text{Per}(E,\varphi(\B_{r/2}(x)))&\leq \frac{C}{(1-s)\delta}+C\delta r^{1-n}\text{Per}(E, \varphi(\B_{r}(x)))
    \end{align*}
    for all $x\in\B_{3/4}$ and $r\leq 1/8$. Using then a covering lemma from \cite{Simon}, see \cite[Lemma 3.18]{CFS} for the exact statement, we would then conclude that $\text{Per}(E,\varphi(\B_{1/2}))\leq \frac{C}{1-s}$ as desired.

    \noindent \textbf{Step 2.} We obtain the Proposition in the $\Lambda$-almost stable case, for any $\Lambda$.
    
    We need to get rid of the requirement in Step 1 that $\Lambda\leq \Lambda_0$; this will follow from a scaling + finite covering argument. Let $0<r<1/2$ be a small radius to be chosen later depending only on $\Lambda$. Given $x_k\in \B_{1/2}$, define $\varphi_{x_k,r}(z):=\varphi(x_k+rz)$, for $z\in\B_1$. By (c) in Remark \ref{flatscalingrmk} and our flatness assumptions, the rescaled manifold $(M,\frac{1}{r^2} g)$ satisfies the flatness assumptions ${\rm FA}_\ell(M,\frac{1}{r^2} g, 1, \varphi(x_k),\varphi_{x_k,r})$. Moreover, since by assumption $E$ is a $\Lambda$-almost stable $s$-minimal set in $\varphi_{x_k,r}(\B_1)$ for $(M,g)$, by Remark \ref{rmk:asresc} it is a $(r^{n+s}\Lambda)$-almost stable $s$-minimal set in $\varphi_{x_k,r}(\B_1)$ for $(M,\frac{1}{r^2} g)$.\\
    Assume for concreteness that we are assuming \eqref{bvfas} and we want to prove \eqref{bvufas}; the proof of \eqref{bvunoas} is even simpler. Letting $r$ small enough so that $r^{n+s}\Lambda\leq\Lambda_0$, the above discussion shows that $E$ satisfies the almost-stability requirement in Step 1 on $\varphi_{x_k,r}(\B_1)$ for $(M,\frac{1}{r^2} g)$. Moreover, applying \eqref{fpergrowth1} with $r_1=r$ and rescaling shows that (up to changing the value of $\kappa$), the condition \eqref{bvfas} also holds for $E$ in the rescaled setting. Therefore, by Step 1 we find that \eqref{bvufas} holds on $(M,\frac{1}{r^2} g)$. Scaling back this information, this gives that
    \begin{align}
        \textnormal{Per}(E;\varphi(\B_{r/2}(x_k))) &\leq C(n,s_0)\frac{1+\sqrt \kappa}{\sqrt{\sigma}}r^{n-1}\label{perbvpfresc}\\
        &= C(n,s_0,\Lambda)\frac{1+\sqrt \kappa}{\sqrt{\sigma}}\,,\nonumber
    \end{align}
where the perimeter is now taken on $(M,g)$. We can now conclude by covering $\B_{1/2}$ with a finite number $C(n,\Lambda)$ of balls $\B_{r/2}(x_k)$ as above, so that summing over $k$ the perimeter bounds on each of them we conclude that
$$\textnormal{Per}(E;\varphi(\B_{1/2})) \leq C(n,s_0,\Lambda)\frac{1+\sqrt \kappa}{\sqrt{\sigma}}$$
as desired.
\end{proof}

The proof of Theorem \ref{BVest} in the finite index case will need the following covering-type lemma, which was devised in \cite{CFS} to exhibit estimates in the finite Morse index case by iteratively reducing to the almost-stable case. It was inspired by the proof of \cite[Proposition 2.6]{FZ}. Essentially, it says that if every time we take $(m+1)$ cubes which are sufficiently far from each other we are able to prove an estimate of power-type on at least one of them, then the estimate actually holds everywhere. The proof consists in subdividing a cube at scale 1 into small subcubes, adding the estimate on the small subcubes (except for a ``bad" set, which will consist of at most $m$ subcubes
and their close neighbours), and then iterating the procedure dyadically by subdividing the ``bad" subcubes once again into smaller subcubes and starting the argument again.
\vsp

\begin{lemma}[{\cite[Lemma 3.19]{CFS}}]\label{morsecovering}
    Let $n\ge 1$, $m\ge 0$, $\theta \in (0,1)$, $D_0>0$ and $\beta>0$. Let $\mathcal{S} : \mathfrak{B} \to [0, +\infty)$ be a subadditive\footnote{Meaning subadditive for finite unions of (hyper)cubes.} function defined on the family $\mathfrak{B}$ of the (hyper)cubes contained in $\Qb_1(0) \subset \R^n$. Denote by $\Qb_r(x) \subset \R^n$ the closed (hyper)cube of center $x$ and side $r$, and assume that
    \begin{itemize} 
        \item[(i)] $\displaystyle\sup_{\{x\,:\,\Qb_r(x)\in\mathfrak{B}\}}\mathcal{S}(\Qb_r(x)) \to 0$ \hspace{0.2cm} as $r\to 0$. 
        \item[(ii)] Whenever $\Qb_{r}(x_0), \Qb_r(x_1), \dotsc , \Qb_r(x_m) \subset \Qb_1(0) $ are $(m+1)$ disjoint cubes of the same side at pairwise distance at least $D_0r $, then 
    \begin{equation*}
        \exists \,i \in \{0,1,\dotsc,m \} \s \textit{such that} \s \mathcal{S}(\Qb_{\theta r}(x_i)) \le r^\beta M_0 \,.
    \end{equation*}
    \end{itemize}
     Then 
    \begin{equation*}
        \mathcal{S}(\Qb_{1/2}(0)) \le CM_0 \,, 
    \end{equation*}
    for some $C=C(n,\theta, m , \beta, D_0) >0$. 

\end{lemma}

\vsp

With the lemma at hand, we can finally give the
\begin{proof}[Proof of Theorem \ref{BVest}]\label{proof:BVest}
As in previous arguments, we assume for concreteness that we are assuming \eqref{bvfin} and we want to prove \eqref{bvufin}; the proof of \eqref{bvunofin} is again even simpler.\\
Let $\Qb_{r}(x_0), \Qb_r(x_1), \dotsc , \Qb_r(x_m) \subset \Qb_1(0) $ be $(m+1)$ disjoint cubes of the same side $r$ and at pairwise distance at least $D_0r $. By Lemmas \ref{asineq} and \ref{loccomparability}, we immediately see that there exists some $\ell \in \{0, \dotsc m \}$ such that the inequality
\begin{equation}
    \delta^2{\rm Per}_s(E)[\xi,\xi] \ge - \frac{C m}{(D_0r/2)^{n+s}} \|\xi \|^2_{L^1(\partial E\cap \varphi(\B_{r/2}(x_\ell)))}=: - \frac{\Lambda}{(r/2)^{n+s}} \|\xi \|^2_{L^1(\partial E\cap \varphi(\B_{r/2}(x_\ell)))}
\end{equation}
holds for all $\xi \in C_c^\infty(\partial E\cap \varphi(\B_{r/2}(x_\ell)))$ and some $C=C(n,s_0)$; we have used that $\B_{r/2}(x_l) \subset \Qb_r(x_l)$. Considering the rescaled manifold $ \widehat M := (M, (2/r)^2 g ) $, by Remarks \ref{flatscalingrmk} and \ref{rmk:asresc} the flatness assumption ${\rm FA}_2(M,  (2/r)^2 g, 1 , q_\ell , \varphi_{x_\ell,r/2}) $ holds and $E$ is a $\Lambda$-almost stable $s$-minimal surface in $ \varphi_{x_\ell,r/2} (\B_{1}(x_\ell))\subset\widehat M$. Therefore, we can apply Proposition \ref{prop:BVestas} (selecting $D_0$ be large enough that $\Lambda = \frac{C m}{D_0^{n+s}} \le \Lambda_0$) and get that
\begin{equation*}
   \textnormal{Per}(E;\varphi_{x_\ell,1/2}(\B_{1/2})) \leq C\frac{1+\sqrt \kappa}{\sqrt{\sigma}} \,,
\end{equation*}
which scaling back to $M$ shows that
\begin{equation*}
    \mathcal{S}(\Qb_{r/C}(x_\ell)) \le \textnormal{Per}(E;\varphi(\B_{r/4}(x_l))) \le C\frac{1+\sqrt \kappa}{\sqrt{\sigma}} r^{n-1}
\end{equation*}
for some $C=C(n,s_0,m)$.
Defining the subadditive function $\mathcal{S}(\Qb) := {\rm Per}(E;\varphi(\frac{1}{2\sqrt{n}}\Qb))$,
where $\lambda \Qb := \{\lambda y \, : \, y \in \Qb \}\,,
$
this shows that (ii) in Lemma \ref{morsecovering} is satisfied, with $\beta=n-1$, $M_0= C\frac{1+\sqrt \kappa}{\sqrt{\sigma}}$, and $\theta$ and $D_0$ depending only on $n$, $s_0$ and $m$. Moreover, (i) in Lemma \ref{morsecovering} trivially holds, since $E$ is just a fixed set with finite perimeter and thus has small perimeter on small enough subsets. Hence, applying the Lemma we find that
\begin{equation*}
    \mathcal{S}(\Qb_{1/2}(0)) = {\rm Per}(E;\varphi(\frac{1}{4\sqrt{n}}\Qb))  \le C\frac{1+\sqrt \kappa}{\sqrt{\sigma}} \,,
\end{equation*}
which after a simple scaling+covering argument shows that
$$
{\rm Per}(E;\varphi(\B_{1/2}))  \le C\frac{1+\sqrt \kappa}{\sqrt{\sigma}}
$$
as well.
\end{proof}

\subsection{Estimates for almost-stable $s$-minimal surfaces}\label{sec:asests}
In all of Sections \ref{sec:asests} and \ref{sec:convbind} we will always assume that $n\ge 2$ and
\begin{equation}
    l:=n+2\,.
\end{equation}
\subsubsection{$L^{-2}$ separation estimate for almost-stable graphs}
We will prove our local estimates under the assumption that
\begin{equation}\label{flatcond}
    {\rm FA}_\ell(M,g,4,p_0,\varphi) \s\mbox{is satisfied} 
\end{equation}
and that there is $\kappa<\infty$ such that $E\subset M$ is an $s$-minimal set in $\varphi(\B_2)$
with a fractional perimeter bound
\begin{equation}\label{locfperbound}
    \textnormal{Per}_s(E;\varphi(\B_2))\leq \frac{\kappa}{1-s}=\frac{\kappa}{\sigma}\, .
\end{equation}
Recall that in \eqref{flatcond}, $\varphi:\B_4\to M$ denotes a chart parametrization. We will suppose, in addition, that defining\footnote{The notation $\B_r'$ denotes the $(n-1)$-th dimensional Euclidean ball of radius $r$ and centered at $0\in\R^{n-1}$.} $\Omega_0 := \B_1'\times(-1,1)\subset\B_2\subset\R^n$,
\begin{eqnarray}\label{whtoihwohw1}
&&  \varphi^{-1}(\partial E)  \mbox{ coincides with } \ \Gamma=\bigcup_{i=1}^N \,\Gamma_i = \bigcup_{i=1}^N \,\{x_n = g_i (x') \} \   \mbox{ inside }\Omega_0,\quad
{\mbox{for some }} g_i: \B_1'\to \R\\
&&{\mbox{satisfying}}\quad
\label{whtoihwohw2}
g_1 < g_2< \cdots < g_N \qquad \mbox{and} \qquad \|\nabla g_i\|_{L^\infty(\B_1')} + \|D^2 g_i\|_{L^\infty(\B_1')} < \delta\,.
\end{eqnarray}
Here the number $N$ of graphs is some arbitrarily large positive integer. We will shortly be able to deduce that, in a smaller cylinder contained in $\Omega_0$, the number of graphical layers is actually bounded by a constant depending only on the choice of constants above (but independent of $\sigma$, so that it is uniform as $s\to 1$).\\

We now record a simple comparison lemma for graphs verifying flatness assumptions as in our hypotheses.
\begin{lemma}[{\cite[Lemma 3.3]{CDSV}}]\label{lem:kernel-comp}
Let~$1-s =\sigma \in (0, 1/2)$.
Assume that $E$ is a set such that \eqref{flatcond} and \eqref{whtoihwohw1}--\eqref{whtoihwohw2} hold. Moreover, assume that for some $i< j$ we have   $\min_{\overline {B_1'}}(g_j-g_i)\le \widetilde \delta$ for some~$\widetilde \delta>0$,
and let\footnote{Note that, in terms of the standard Beta and Gamma functions,
\[
c_{n,s}
=\frac{\mathcal H^{n-2}(\bS^{n-2})}{2}{\mathrm B}\Bigl(\frac{n-1}{2},\frac{1+s}{2}\Bigr)
=\frac{\cH^{n-2}(\bS^{n-2})\Gamma\bigl(\frac{n-1}{2}\bigr)\Gamma\bigl(\frac{1+s}{2}\bigr)}{2\Gamma\bigl(\frac{n+s}{2}\bigr)}
=\frac{\cH^{n-2}(\bS^{n-2})}{n-1}+O(\sigma)
\]
is universally comparable to $1$.
}
$$c_{n,s} : =  \mathcal H^{n-2}(\bS^{n-2})\int_0^ \infty  (1 + t^2 )^{-\frac{n+s}{2}} \,t^{n-2}\,dt.$$

Then, it follows that
\[
\big(1-c_{\delta, \widetilde \delta}\big)  \frac{c_{n,s}\;\sigma}{|g_j(x')-g_i(x')|^{1+s}}   \le   \int_{\Gamma_j }  \frac{ \sigma }{|x-y|^{n+s} }  \,dy  \le \big(1+c_{\delta, \widetilde \delta}\big) \frac{c_{n,s}\;\sigma}{|g_j(x')-g_i(x')|^{1+s}}
\]
for all $x\in \Gamma_i \cap B_{3/4}'\times \R$, where $c_{\delta, \widetilde \delta} \leq C(\sqrt{\delta}+\widetilde{\delta})\downarrow 0$ as $\delta, \widetilde \delta\downarrow 0$.
\end{lemma}

\begin{proposition}\label{propL-1ps}
Let $E$ be $\Lambda$-almost stable $s$-minimal in $\varphi(\B_2)$, and assume that \eqref{flatcond}--\eqref{whtoihwohw2} hold true, with~$(1-s)=\sigma \in (0,1/2)$. There exists $C<\infty$, depending only on $\Lambda$, $n$, and the constants in \eqref{flatcond}--\eqref{whtoihwohw2}, such that for every $i$ such that \begin{equation*}\min_{\overline B_{1/2}'} |g_i| \le \frac12,\end{equation*} we have that
\begin{equation}\label{ghjgrgiu}
\sigma \int_{B'_{1/2}} \frac{dx'}{|g_{i+1}-g_i|^{1+s}} \le C.
\end{equation}
\end{proposition}
\begin{proof}
Observe that $B'_1\times[-1,1]\subset \B_2$.
Let $X$ be a $C^1$ vector field of the form $\varphi_*(e_n\eta)$, where $e_n$ is the $n$--th basis coordinate vector of $\R^n$ (i.e. pointing "upwards", in the direction in which the $\Gamma_i$ are graphical) and $\eta$ is a standard cutoff verifying $\chi_{\B_{1/2}'\times [-3/4,3/4]}\le \eta \le \chi_{\B_{2/3}'\times [-4/5,4/5]}$. Define the function $\xi=\langle X,\nu_{\partial E}\rangle_g$ on $\partial E$. As in \eqref{diff2var}--\eqref{diff2var2}, considering the test functions $\xi$ and $|\xi|$, the $\Lambda$-almost stability assumption gives 
    \begin{equation}\label{aggdhfgadeq}
    2\iint_{\partial E\times\partial E} \Big[|\xi(p)||\xi(q)|-\xi(p)\xi(q)\Big]K(p,q)\leq \delta^2\text{Per}_s(E)[\xi,\xi]+\Lambda\Big(\int_{\partial E} |\xi|\Big)^2\, .
    \end{equation}
    Now, we have defined 
    \begin{equation}\label{eqtestnor}
    \xi(p)=\langle X,\nu_{\partial E}\rangle_g(p)=(\eta\circ\varphi^{-1})(p)\langle\varphi_*(e_n),\nu_{\partial E}\rangle_g(p)\, .
    \end{equation}
    Recall that, by assumptions \eqref{whtoihwohw1} and \eqref{whtoihwohw2}, $\varphi^{-1}(\partial E)$ is a union of very flat graphs $\Gamma_j$ of the form $\Gamma_j=\{x_n = g_j (x') \}$. Together with the flatness assumptions on the metric, this implies that
    $$|\langle\varphi_*(e_n),\nu_{\partial E}\rangle_g|\geq c_1| e_n\cdot\varphi_*^{-1}(\nu_{\partial E})|\geq c_2\, ,$$
    where $c_2>0$ depends only on $\delta$ from \eqref{whtoihwohw2}.
    Furthermore, $\langle\varphi_*(e_n),\nu_{\partial E}\rangle_g$ changes sign on $\varphi(\Gamma_i)$ and $\varphi(\Gamma_{i+1})$, since the consecutive layers $\Gamma_i$ and $\Gamma_{i+1}$ have normal vectors pointing in (almost) opposite directions.\\
    These simple facts and \eqref{eqtestnor} imply that, for $p\in\varphi(\Gamma_i\cap \B_{1/2}'\times [-3/4,3/4])$ and $q\in\varphi(\Gamma_{i+1}\cap\B_{1/2}'\times [-3/4,3/4])$, then 
    $$\Big[|\xi(p)||\xi(q)|-\xi(p)\xi(q)\Big]=2|\xi(p)||\xi(q)|\geq 2c_2^2[(\eta\circ\varphi^{-1})(p)][(\eta\circ\varphi^{-1})(q)]\geq 2c_2^2\, .$$
Substituting into the original inequality \eqref{aggdhfgadeq}, we deduce that
$$
8c_2^2\iint_{\Sigma_i\times \Sigma_{i+1}}K(p,q)\leq \delta^2\text{Per}_s(E)[\xi,\xi]+\Lambda\Big(\int_{\partial E} |\xi|\Big)^2\, ,
$$
where $\Sigma_j:=\varphi(\Gamma_j\cap \B_{1/2}'\times [-3/4,3/4])$.
So, passing to coordinates with $\varphi$, using Lemma \ref{loccomparability} and Lemma \ref{lem:kernel-comp} we obtain
\begin{equation}\label{l-2rhsin}
c\sigma\int_{B'_{1/2}}\frac{1}{|g_{i+1}(x')-g_i(x')|^{1+s}}dx'\leq \sigma\delta^2\text{Per}_s(E)[\xi,\xi]+\Lambda\sigma\Big(\int_{\partial E} |\xi|\Big)^2\, .
\end{equation}

On the other hand, the right hand side in this inequality is bounded: To estimate the first term, recalling Definition \ref{2quaddef}, by Lemma \ref{enboundslemma2} and the assumption \eqref{locfperbound} we can bound 
\begin{equation}\label{hiujshgi}
    \sg\delta^2\text{Per}_s(E)[\xi,\xi]=\sg\frac{d^2}{dt^2}\Big|_{t=0}\textnormal{Per}_s(\phi_X^t(E))\leq \sg C\textnormal{Per}_s(E)\leq C\kappa\, .
\end{equation}
To estimate the second term (corresponding to the error introduced by almost-stability), we apply the classical perimeter bound \eqref{bvfas} given by Proposition \ref{prop:BVestas}, since then we can bound
\begin{equation}\label{aiugfouifg}
\sg\Lambda\Big(\int_{\partial E} |\xi|\Big)^2\leq C\sg \Lambda\textnormal{Per}^2(E,\B'_{2/3}\times[-4/5,\,4/5])\leq C(\Lambda,n,s_0)\, .
\end{equation}
To be precise, Proposition \ref{prop:BVestas} is stated on a ball of radius one. We can cover $\B'_{2/3}\times[-4/5,\,4/5]$ with a finite number $C_n$ of balls $\B_{1/10}(x_k)$ so that $\B_{1/5}(x_k)\subset \B'_1\times[-1,1]$ for all $1\leq k\leq C_n$. As in Step 2 of the proof of Proposition \ref{prop:BVestas}, applying Proposition \ref{prop:BVestas} on the manifold $(M,5^2g)$ and scaling back we see that
$$\textnormal{Per}(E;\B_{1/5}(x_k)) \leq C\frac{1+\sqrt \kappa}{\sqrt{\sigma}} \,,$$
which summing over $k$ shows that
$$\textnormal{Per}(E,\B'_{2/3}\times[-4/5,\,4/5])\leq C(\Lambda,n,s_0)\frac{1+\sqrt{\kappa}}{\sqrt{\sg}}$$
thus giving \eqref{aiugfouifg}.\\

Applying \eqref{hiujshgi} and \eqref{aiugfouifg} to the RHS of \eqref{l-2rhsin}, we conclude that
\begin{equation*}
\sigma\int_{B'_{1/2}}\frac{1}{|g_{i+1}(x')-g_i(x')|^{1+s}}dx'\leq C\, ,
\end{equation*}
where $C$ has the right dependencies.
\end{proof}

The next lemma gives a rough lower bound for the separation between layers. Any bound which is polynomial in $\sigma$ would serve our purposes here; later, we will upgrade it to the optimal rate $\sigma^{1/2}$ in the case $n=3$.
\begin{lemma}\label{whtiohwoiwh}
Let $E\subset M$ be $\Lambda$-almost stable $s$-minimal in $\varphi(\B_2)$, and assume that \eqref{flatcond}--\eqref{whtoihwohw2} hold true, with~$(1-s)=\sigma \in (0,1/2)$. There exists $C<\infty$, depending only on $\Lambda$, $n$, and the constants in \eqref{flatcond}--\eqref{whtoihwohw2}, such that for every $i$ such that \begin{equation*}\label{ASSNI} \min_{\overline B_{1/2}'} |g_i| \le \frac12,\end{equation*}
we have that
\[
\inf_{B_{1/2}'} (g_{i+1}-g_i) \ge \frac{\sigma^2}{C} .
\]
\end{lemma}

\begin{proof}
Let $i$ be as in the statement, and assume that there exists a point $z' \in \B'_{1/2}$ such that  $g_{i+1}(z')-g_i (z')= \tau^2 \sigma^2$  for some $\tau\in (0,1)$. Our goal will be to bound $\tau$ away from zero (hence we may assume without loss of generality that $\tau$ is  very small whenever needed).

By the assumption on $i$, we have that~$|g_i(z')|<3/4$. Let $z: = (z', g_i (z'))$ and $\varrho:= \tau \sigma$. Define $\varphi_{z,\varrho}:=\varphi(z+\varrho\, \cdot \,)$. In the coordinates given by $\varphi_{z,\varrho}$, $\Gamma$ is simply transforming into the rescaled set $\widetilde \Gamma : = \frac{\Gamma-z}{\varrho}\cap\B_1$, in the sense that $(\varphi_{z,\varrho}^{-1}\circ \varphi)(\Gamma\cap\B_{\varrho}(z))=\widetilde\Gamma$.\\

By (c) in Remark \ref{flatscalingrmk} and \eqref{flatcond}, if $\rho$ is sufficiently small the rescaled manifold $(M,\frac{1}{\varrho^2} g)$ satisfies the flatness assumptions ${\rm FA}_\ell(M,\frac{1}{\varrho^2} g, 1, \varphi(z),\varphi_{z,\varrho})$. Moreover, by the same argument we gave in Step 2 of the proof of Proposition \ref{prop:BVestas}, $E$ is a $\Lambda$-almost stable $s$-minimal set for $(M,\varrho^2 g)$ in $\varphi_{z,\varrho}(\B_1)$; in fact, it is even $\varrho^{n+s}\Lambda$-almost stable. Likewise, thanks to \eqref{fpergrowth1} the fractional perimeter bound \eqref{locfperbound} is satisfied in this rescaled setting, up to changing the value of $\kappa$. Finally, we also observe that the ``layers'' $\{x_n = g_{j}(x')\}$
become, in the coordinates given by $\varphi_{z,\varrho}$,
\[
\{x_n = \widetilde g_j(x')\} \qquad \mbox{where }  \| D^2 \widetilde g_j \|_{L^\infty} \le  \varrho \delta, \quad \mbox{for all } j=1,\dots, N.
\]
Therefore, we can apply Proposition~\ref{propL-1ps} in our rescaled situation and deduce that
\[
\sg\int_{B_{1/2}'} \frac{1}{|\widetilde g_{i+1}(x')-\widetilde g_i(x')|^{1+s}}\, dx' \le C.
\]
On the other hand, since the point $z$ is mapped to $0$, we have that
\[
\widetilde g_i(0) =0 \quad \mbox{and} \quad  \widetilde g_{i+1} (0) \le \varrho.
\]
Combining this with the bound for the $D^2 \widetilde g_j$, a Taylor expansion immediately gives that, for $\delta, \varrho$ sufficiently small,
\[
0\le \widetilde g_{i+1}- \widetilde g_i \le  2\varrho \quad\mbox{in } B_1'.
\]

Therefore, we also get a lower bound
\[
 \frac{1}{\varrho^{1+s}}\leq C\int_{B_{1/2}'} \frac{1}{|\widetilde g_{i+1}(x')-\widetilde g_i(x')|^{1+s}}\, dx' \,.
\]
Combining the above, we deduce that
\[
\frac{\sigma}{\varrho^{1+s}} \le  C\,,
\]
which since~$\varrho = \tau \sigma $ gives that~$\sigma^{-s} \le C\tau^{1+s}$. In particular, this shows
that~$\tau$ is bounded away from zero with a uniform bound for $\sigma \in (0,1/2)$, as desired.
\end{proof}

From this rough separation estimate, we can already deduce that the number of layers inside a smaller cylinder is bounded.
\begin{lemma}\label{classperball} Let the same assumptions as in Lemma \ref{whtiohwoiwh} hold. Then, defining $N_2$ as the number of layers of $\varphi^{-1}(\partial E)$ intersecting $\B_{1/2}'\times[-1/2,1/2]$, there exists $C<\infty$, depending only on $\Lambda$, $n$ and the constants in \eqref{flatcond}--\eqref{whtoihwohw2}, such that $$N_2\leq C\,.$$
Therefore, we deduce that
$$\textnormal{Per}(E;\varphi(\B_{1/2}'\times[-1/2,1/2]))\leq C\, .$$ 
\end{lemma}
\begin{proof}
$N_2$ corresponds to the number of indices $i$ such that \begin{equation*} \min_{\overline B_{1/2}'} |g_i| \le \frac12.\end{equation*}
Since for any such $i$ the associated $\Gamma_i$ is a graph with uniform estimates (by \eqref{whtoihwohw1}--\eqref{whtoihwohw2}), bounding the number $N_2$ of such $i$ immediately gives a bound for the total classical perimeter. We thus focus on bounding $N_2$, which will be possible thanks to every layer $\Gamma_i$ ``contributing a positive amount to the fractional perimeter" due to the rough separation estimate between layers.

\noindent \textbf{Claim.} Let $i$ be such that $\min_{\overline B_{1/2}'} |g_i| \le \frac12$,
and let $j\neq i$.
Then there is $c>0$ small enough such that, given points $x\in\Gamma_i$ and $y\in\Gamma_j$, the balls $\B_r(x)$ and $\B_r(y)$ are disjoint for $r=c\sg^{4}$.
\begin{proof}[Proof of the claim]
    Indeed, if $\B_r(x)\cap\B_r(y)\neq\emptyset$, then in particular $|y'-x'|\leq 2r$ and $|g_{j}(y')-g_{i}(x')|\leq 2r$. But then, by Lemma \ref{whtiohwoiwh} and \eqref{whtoihwohw2},
\begin{align*}
2c\sg^{4}&=2r\geq|g_{j}(y')-g_{i}(x')|\geq |g_{j}(x')-g_{i}(x')|-|g_{j}(y')-g_{j}(x')|\\
&\geq \frac{\sg^{4}}{C} - \delta|y'-x'|\geq \frac{\sg^{4}}{C} - 2\delta r=\frac{\sg^{4}}{C} - 2\delta c\sg^{4}\, ,
\end{align*}
or in other words $\frac{1}{C}\leq 2(1+\delta)c$,
which is a contradiction for $c=\frac{1}{2C(2+\delta)}$ or smaller.
\end{proof}
Now that the claim has been proved, denote $F=\varphi^{-1}(E)$, and let $r$ be as in the claim. Now, thanks to the uniform $C^2$ estimates on our graphs, for $s$ sufficiently close to $1$ (i.e. $\sigma$ close to $0$) we have the following: if $x\in\Gamma_i$, then both $|F\cap \B_r(x)|\geq 1/4 |\B_r(x)|$ and $|F^c\cap \B_r(x)|\geq 1/4|\B_r(x)|$ hold. This is true since satisfying uniform $C^2$ estimates means that $\Gamma_i$ separates from its tangent hyperplane at $x$ at most $Cr^2$ inside $\B_r(x)$, so that for $\sigma$ (and thus $r$) sufficiently small $\Gamma_i$ actually divides the ball $\B_r(x)$ in two pieces of almost equal volume. Denote 
$$\textnormal{Per}_s|_{\B_r(x)}(F):=\iint_{(F\cap \B_r(x))\times (F^c\cap\B_r(x))} \frac{1}{|x-y|^{n+s}}\,dx\,dy\,.
$$
Then, by the relative fractional isoperimetric inequality (see \cite{FR}) together with $\min\{|F\cap \B_r(x)|,|F^c\cap \B_r(x)|\}\geq 1/4|\B_r(x)|$, we have that
$$
\textnormal{Per}_s|_{\B_r(x)}(F)\geq \frac{c}{1-s}r^{n-s}=\frac{c}{\sg}r^{n-s}
$$
with $c(n)>0$. Moreover, since $\Gamma_i$ is a graph over the $(n-1)$-dimensional ball of radius $1/2$ in $\R^{n-1}$, we can find at least $k\geq c/r^{n-1}$ disjoint balls of radius $r$ centered at points $x_1^{i}\in\Gamma_i,...,x_k^{i}\in\Gamma_i$ (indeed, it is clear that one can arrange at least $k\geq c/r^{n-1}$ points in a square lattice inside the Euclidean ball of radius $1/2$ in $\R^{n-1}$ with distance $r$ between adjacent points, and then it suffices to take their images under the graph parametrisation $g_i$). Repeating this for each $i=1,...,N_2$ and observing that any pair of balls obtained in this way will be disjoint (by construction if they belong to the same $\Gamma_i$, and by the choice of $r$ at the beginning of the proof if they belong to $\Gamma_i$ and $\Gamma_j$ with $i\neq j$), we can bound
\begin{align*}
\textnormal{Per}_s(F,\B'_{1/2}\times(-3/4,\,3/4))&\geq \sum_{1\leq i\leq N_2}\sum_{1\leq l\leq k}\textnormal{Per}_s|_{\B_r(x_l^i)}(F)\geq N_2k \frac{c}{\sg}r^{n-s}\\
&\geq N_2\frac{c}{\sg}r^{1-s}=N_2\frac{c}{\sg}\sg^{-4\sg}\geq N_2 \frac{c}{2\sg}\, .
\end{align*}
In the last line we are taking $s$ sufficiently close to $1$ (equiv. $\sigma=1-s$ sufficiently close to $0$) so that 
$$\sg^{-4\sg}=1+O( \sg |\log(\sg)|)\geq \frac{1}{2}\, .$$

In other words, we have found that the number $N_2$ of layers is bounded by
$$N_2\leq C(1-s)\textnormal{Per}_s(F,\B'_{1/2}\times(-3/4,\,3/4))\,.$$ Moreover, by Lemma \ref{loccomparability} and the flatness of the metric it is immediate that
$$
\textnormal{Per}_s(F,\B'_{1/2}\times(-3/4,\,3/4))\leq C \textnormal{Per}_s(E,\varphi(\B'_{1/2}\times(-3/4,\,3/4)))\leq C\textnormal{Per}_s(E,\varphi(\Omega_0))\,.
$$
Combining these two things, by assumption \eqref{locfperbound} we then deduce that $N_2\leq C\kappa$, which concludes the proof.
\end{proof}

We now introduce the following observation, which will be very useful since it will permit us to relate volume and boundary integrals through integration by parts while allowing us to reuse our known estimates for the family of kernels of the type $K_s$ (defined through \eqref{wethiowhoihw2}):
\begin{lemma}\label{kerdivlem}
The identity
$$
K_s(p,q)=-\textnormal{div}_q(\nabla_q K_{s-2})(p,q)
$$
holds.
\end{lemma}
\begin{proof}
By definition, $K_s(p,q)=\frac{s/2}{\Gamma(1-s/2)}\int_t\frac{H_t}{t^{1+s/2}}\,dt$. We start by computing, using the identity $\Gamma(1+z)=z\Gamma(z)$:
\begin{align*}
    K_{s-2}(p,q)=\frac{(s-2)/2}{\Gamma(1-(s-2)/2)}\int_t H_t(p,q)/t^{1+(s-2)/2}&=-\frac{1-s/2}{\Gamma(1+[1-s/2])}\int_t H_t(p,q)/t^{s/2}=-\frac{1}{\Gamma(1-s/2)}\int_t H_t(p,q)/t^{s/2}\,.
\end{align*}
Therefore, by an application of the heat equation and integration by parts (in time), we can compute
\begin{align*}
    \textnormal{div}_q(\nabla_q K_{s-2})(p,q)=\Delta_q K_{s-2}(p,q)&=-\frac{1}{\Gamma(1-s/2)}\int_t \Delta_p H_t(p,q)/t^{s/2}=-\frac{1}{\Gamma(1-s/2)}\int_t (\partial_t H_t(p,q))/t^{s/2}\\
    &=-\frac{s/2}{\Gamma(1-s/2)}\int_tH_t(p,q)/t^{1+s/2}=-K_s(p,q)\,,
\end{align*}
giving the claimed result.
\end{proof}
This allows us to make the following "generalised" definition, justified by the remark that follows it:
\begin{definition}\label{fracmeandef}
    The fractional mean curvature (or nonlocal mean curvature, or NMC) of a two-sided regular hypersurface $\Sigma\subset M$ with choice of normal vector $\nu_\Sigma$, where $\Sigma\subset M$ is not assumed to be part of the boundary of a set, at $p\in\Sigma$ is defined as
    $$
    H_s[\Sigma](p):=2\sigma\int_{\Sigma}\langle\nu_{\Sigma}(q),\nabla_q K_{s-2}(p,q)\rangle_g\,dq\, .
    $$
\end{definition}
\begin{remark}
    Definition \ref{fracmeandef} is compatible with Definition \ref{1vardef}, in the sense that given a regular open domain $E$, setting $\Sigma=\partial E$ and $\nu_\Sigma$ the outer normal vector gives the same value for $H_s[\Sigma]$ with both definitions. This is immediately seen by an application of the divergence theorem together with Lemma \ref{kerdivlem}.
\end{remark}
We can now deduce the desired $L^{-2}$ bound for the separation between layers, as well as the decay in $L^2$ of the nonlocal mean curvature of each of the individual layers as $s\to 1$. We remark once again that we are making no assumptions on the dimension $n$ in this entire section.
\begin{proposition}\label{L2NMC}
Let $E\subset M$ be $\Lambda$-almost stable $s$-minimal in $\varphi(\B_2)$, and assume that \eqref{flatcond}--\eqref{whtoihwohw2} hold true, with~$(1-s)=\sigma \in (0,1/2)$. There exists $C<\infty$, depending only on $\Lambda$, $n$ and the constants in \eqref{flatcond}--\eqref{whtoihwohw2}, such that for every $i$ such that \begin{equation*}\min_{\overline B_{1/2}'} |g_i| \le \frac12,\end{equation*}
we have that
\begin{equation}\label{l2sepeq}
\sigma \int_{B'_{1/2}} \frac{dx'}{|g_{i+1}-g_i|^2} \le C.
\end{equation}
Moreover, defining $f_i:=H_s[\varphi(\Gamma_i)]$, we have that
\begin{equation} \label{wiowhoih1}
\|f_i\|_{L^2(\varphi(\Gamma_i \cap [B'_{1/2}\times(-1,1)]))} \le C\sqrt \sigma.
\end{equation}
\end{proposition}
\begin{proof}
\textbf{Step 1.}
In Proposition \ref{propL-1ps} we have shown that
\[
\sigma \int_{B'_{1/2}} \frac{dx'}{|g_{i+1}-g_i|^{1+s}} \le C.
\]
To reach \eqref{l2sepeq}, we need to upgrade the exponent in the denominator from $(1+s)$ to $2$. For this, it suffices to notice that since by Lemma \ref{whtiohwoiwh} we have that $ (g_{i+1}-g_i) \in (\sigma^4/C,1)$, we can then estimate
\begin{equation}\label{eq:power-2}
	|g_j-g_i|^{s}=\bigl(1+O(\sigma|\log\sigma|)\bigr)	|g_j-g_i|.
\end{equation}

\noindent \textbf{Step 2.} We deduce \eqref{wiowhoih1} from \eqref{l2sepeq}.

    As before, we use the notation $\Omega_0 = \B_1'\times(-1,1)\subset\R^n$. We want to split the contributions of each of the individual layers of $E$ to the total nonlocal mean curvature. For that, define the set $A_i=\varphi(\{x_n<g_i(x')\}\cap\Omega_0)$, so that $\partial A_i=\varphi(\Gamma_i)\cup \varphi(\partial \Omega_0\cap \{x_n<g_i(x')\})$. In other words, in coordinates $A_i$ comprises all the space below the graph $\Gamma_i$, and the set $\partial A_i$ corresponds to a portion of the layer $\Gamma_i$ (plus some extra piece contained in $\partial \Omega_0$). Define also the sets $E_i=A_i\setminus A_{i-1}$, so that (up to possibly interchanging $E$ with its complement) in particular it follows that $E\cap \varphi(\Omega_0)=\bigcup_{i\in 2\Z} E_i$.
    
    Let $i$ be such that $\min_{\overline B_{1/2}'} |g_i| \le \frac12$, with $i$ even (the odd case is essentially identical), and let $p\in\varphi(\Gamma_i\cap\{x'\in\B'_{1/2}\})$. With the definitions above, it is immediate to see that we can split
    \begin{align*}
\int_{\varphi(\Omega_0)} dq\,(\chi_{E}-\chi_{E^c})(q)K_s(p,q)=\int_{\varphi(\Omega_0)} dq\,(\chi_{A_i}-\chi_{A_i^c})(q)K_s(p,q)&-2\sum_{j<i,\,j-i \,{\rm odd}}\int_{\varphi(\Omega_0)}dq\, \chi_{E_{j}}(q)K_s(p,q)\\
&+2\sum_{j>i,\,j-i\, {\rm even}}\int_{\varphi(\Omega_0)}dq\,\chi_{E_{j}}(q)K_s(p,q)\, .
\end{align*}
On the other hand, since $E$ is an $s$-minimal surface in particular it has zero NMC at $p$, so that recalling Definition \ref{1vardef} we have that (we omit $dq$ in the notation from now on)
$$
\int_{\varphi(\Omega_0)} (\chi_{E}-\chi_{E^c})K_s(p,q)+\int_{\varphi(\Omega_0)^c} (\chi_{E}-\chi_{E^c})K_s(p,q)=\int_{M} (\chi_{E}-\chi_{E^c})K_s(p,q)=0\,.
$$
Combining both expressions, we can then write
\begin{align*}
-\int_{\varphi(\Omega_0)} (\chi_{A_i}-\chi_{A_i^c})K_s(p,q)=-\int_{\varphi(\Omega_0)^c} (\chi_{E}-\chi_{E^c})K_s(p,q)&-2\sum_{j<i,\,j-i \,{\rm odd}}\int_{\varphi(\Omega_0)} \chi_{E_{j}}K_s(p,q)\\
&+2\sum_{j>i,\,j-i \,{\rm even}}\int_{\varphi(\Omega_0)} \chi_{E_{j}}K_s(p,q)\, .
\end{align*}
Applying the divergence theorem on the left hand side using Lemma \ref{kerdivlem}, we then find that
\begin{align}\label{nmcsplit}
H_s[\varphi(\Gamma_i)](p)=-2\sigma\sum_{j<i, j-i \,even}\int_{\varphi(\Omega_0)} \chi_{E_{j}}K_s(p,q)+2\sigma\sum_{j>i,j-i \,even}\int_{\varphi(\Omega_0)} \chi_{E_{j}}K_s(p,q)+ \sigma {\rm Ext}(p)\,,
\end{align}
where
\begin{equation*}
{\rm Ext}(p):=-\int_{\varphi(\Omega_0)^c} (\chi_{E}-\chi_{E^c})K_s(p,q)-\int_{\varphi(\partial \Omega_0)}\langle\nu_{\partial \Omega_0}(q),\nabla_q K_{s-2}(p,q)\rangle_g\,d\mathcal{H}^{n-1}(q)\,.
\end{equation*}
In other words, \eqref{nmcsplit} says that, since the NMC of $E$ is zero, the ``individual" NMC of $\varphi(\Gamma_i)$ in the sense of Definition \ref{fracmeandef} is equal to the contribution from the other layers plus an exterior error.\\

Now, it is clear that the $E_j$ with $j<i$ are contained in $A_{i-1}:=\varphi(\{x_n<g_{i-1}(x')\}\cap\Omega_0)$, and likewise the $E_j$ with $j>i$ are contained in $B_{i+1}:=\varphi(\{x_n>g_{i+1}(x')\}\cap\Omega_0)$, so that
$$
\Big|H_s[\varphi(\Gamma_i)]\Big|(p)\leq 2\sigma\Big|\int_{A_{i-1}}K_s(p,q) \Big|+2\sigma\Big|\int_{B_{i+1}}K_s(p,q) \Big|+\sigma \Big|{\rm Ext}\Big|(p)\, .
$$
Using this observation, passing to coordinates with $\varphi^{-1}$ (letting $x:=\varphi^{-1}(p)$, and using Lemma \ref{loccomparability} and the flatness assumptions) and then integrating by parts (once again thanks to Lemma \ref{kerdivlem}/the divergence theorem, applied also with the Euclidean kernel $\frac{1}{|x-y|^{n+s}}$), we can bound
\begin{align*}
\Big|H_s[\varphi(\Gamma_i)]\Big|(p)&\leq 2\sigma\Big|\int_{A_{i-1}}K_s(p,q) \Big|+2\sigma\Big|\int_{B_{i+1}}K_s(p,q) \Big|+\sigma\Big|{\rm Ext}\Big|(p)\\
&\leq C\sigma\int_{\varphi^{-1}(A_{i-1})}dy\,\frac{1}{|x-y|^{n+s}} +C\sigma\int_{\varphi^{-1}(B_{i+1})}dy\,\frac{1}{|x-y|^{n+s}} +\sg\Big|{\rm Ext}\Big|(p)\\
&\leq C\sigma\int_{\Gamma_{i-1}}d\mathcal H^{n-1}(y)\frac{1}{|x-y|^{n+s-1}}+C\sigma\int_{\Gamma_{i+1}}d\mathcal H^{n-1}(y)\frac{1}{|x-y|^{n+s-1}}+\sigma {\rm Ext}_2(p)\, .
\end{align*}
Here, the new exterior error term is
\begin{align*}
{\rm Ext}_2(p)&=\Big|\int_{\varphi(\Omega_0)^c} dq\,(\chi_{E}-\chi_{E^c})K_s(p,q)\Big|+\Big|\int_{\varphi(\partial \Omega_0)}\langle\nu_{\Sigma}(q),\nabla_q K_{s-2}(p,q)\rangle_g\,d\mathcal{H}^{n-1}(q)\Big|\\
&\ \ +\Big|\int_{\partial \Omega_0\cap \{x_n<g_{i-1}(x')\}}\frac{1}{|x-y|^{n+s-1}}\,d\mathcal{H}^{n-1}(q)\Big|+\Big|\int_{\partial \Omega_0\cap \{x_n>g_{i+1}(x')\}}\frac{1}{|x-y|^{n+s-1}}\,d\mathcal{H}^{n-1}(q)\Big|\,.
\end{align*}
We can bound the first term in ${\rm Ext}_2(p)$ by a uniform constant using \eqref{remaining3} (with $\alpha=0$). Moreover, by an application of \eqref{remaining2}, we can bound the (differentiated) kernel in the second term by a uniform constant (since $\textnormal{dist}(p,q)\geq \frac{1}{C}>0$ for all $q$ in $\varphi(\Omega_0)^c$), and likewise for the two last terms involving $\frac{1}{|x-y|^{n+s-1}}$ (since $|x-y|\geq \frac{1}{C}$ for all $y$ in $\partial \Omega_0$). Therefore, ${\rm Ext}_2(p)\leq C$.\\
Substituting this above and using Lemma \ref{lem:kernel-comp} and \eqref{eq:power-2}, we can estimate
\begin{align}
\Big|H_s[\varphi(\Gamma_i)]\Big|(p)
&\leq C\sigma\int_{\Gamma_{i-1}}\frac{1}{|x-y|^{n+s-1}}+C\sigma\int_{\Gamma_{i+1}}\frac{1}{|x-y|^{n+s-1}}+\sigma {\rm Ext_2}(p) \nonumber\\
&\leq C\frac{\sigma}{|g_{i}(x')-g_{i-1}(x')|^{s}}+C\frac{\sigma}{|g_{i+1}(x')-g_{i}(x')|^{s}}+C\sigma \nonumber\\
&\leq C\frac{\sigma}{|g_{i}(x')-g_{i-1}(x')|}+C\frac{\sigma}{|g_{i+1}(x')-g_{i}(x')|}+C\sigma \label{nmcsep1}\, .
\end{align}

Applying the $L^{-2}$ estimate in \eqref{l2sepeq}, we conclude the desired bound for the $L^2$ norm of $H_s[\varphi(\Gamma_i)]$.
\end{proof}

\subsubsection{Uniform $C^{2,\alpha}$ estimates for almost-stable graphs in dimension $n=3$}

In this section we will prove that (for $n=3$) almost-stability implies a uniform $C^{2,\alpha}$ bound for the layers of our surface. We remark that this is the only section of the article in which the restriction to dimension $3$ is used.\\
Interestingly, our strategy is somewhat different from the one in \cite{CDSV}, and it allows us to avoid using the nonlocal mean curvature equation for graphs, which could be complicated to implement in a Riemannian setting. We show instead that, even before upgrading the regularity of our layers  to a uniform $C^{2,\alpha}$ estimate, the local and nonlocal mean curvatures of our layers are much closer to each other than their natural size. This allows us to argue using the classical minimal graph equation instead of the nonlocal one in Proposition \ref{optsep}, the latter being the crucial step in dimension $n=3$ which allows to decouple the separation estimate from the regularity estimates. Moreover, the $C^{2,\alpha}$ regularity result in \cite{CDSV} is only obtained for $\alpha$ sufficiently small, whereas our strategy allows us to obtain it for any $\alpha\in(0,1)$.\\
With an eye towards future directions, using that one can show the proximity of the nonlocal mean curvatures to the local ones before upgrading the regularity of our layers to the uniform $C^{2,\alpha}$ estimate will probably be essential to extend our results to dimensions $4\leq n \leq 7$. Indeed, it hints at the possibility of using the Toda system (obtained in \cite{CDSV} under the assumption of uniform $C^{2,\alpha}$ convergence as $s\to 1$ to a hyperplane) prior to knowing this desired uniform regularity.\\

We begin the section by showing a rough $C^{2,\alpha}$ bound which explodes as $\sigma\to 0$, but which surprisingly will still allow us to compare the nonlocal and local mean curvatures.
\begin{proposition}\label{NegC2alpha}
Let $E\subset M$ be $\Lambda$-almost stable $s$-minimal in $\varphi(\B_2)$, and assume that \eqref{flatcond}--\eqref{whtoihwohw2} hold true, with~$(1-s)=\sigma \in (0,1/8)$. Let $\alpha\in(0,1/4)$. Then, for every $i$ such that \begin{equation*}\min_{\overline B_{1/2}'} |g_i| \le \frac12,\end{equation*}
we have that, for every $\alpha\in(0,1/4)$, there exists a constant $C< \infty$ depending only on $\Lambda$, $n$, $s_0$, $\alpha$ and the constants in \eqref{flatcond}--\eqref{whtoihwohw2}, and in particular independent of $s$, such that
\begin{equation} \label{wiowhoih567}
\|g_i\|_{C^{2,\alpha}(\B_{1/4}')} \le C\sigma^{-12\alpha}.
\end{equation}
\end{proposition}
\begin{proof}
\textbf{Step 1.} We first show the Hölder estimate $\|H_s[\varphi(\Gamma_i)](\cdot)\|_{C^{0,1/2}(\varphi(\Gamma_i \cap [B'_{1/2}\times(-1,1)]))}\leq C\sg^{-3}$.\\
We start by using Lemma \ref{whtiohwoiwh} in \eqref{nmcsep1}, which gives the rough pointwise estimate
\begin{equation}\label{H-3b}
\Big|H_s[\varphi(\Gamma_i)]\Big|(p)\leq C \sigma^{-1}\, .
\end{equation}
To find a higher-order bound, we differentiate both sides of equation \eqref{nmcsplit} over $\Gamma_i$. Doing this, and then reproducing the computations that led to \eqref{nmcsep1} from \eqref{nmcsplit}, we can compute
\begin{align}
|\nabla_p^{\varphi(\Gamma_i)}H_s[\varphi(\Gamma_i)]|(p)&\leq 2\sg\sum_{j<i, j-i\,{\rm odd}}\int_{\varphi(\Omega_0)} dq\,\chi_{E_{j}}(q)|\nabla_p^{\varphi(\Gamma_i)}K_s(p,q)|\nonumber\\
&\ \ \ \ +2\sg\sum_{j>i,j-i \,{\rm even}}\int_{\varphi(\Omega_0)} dq\,\chi_{E_{j}}(q)|\nabla_p^{\varphi(\Gamma_i)}K_s(p,q)|+ \sg|\nabla^{\varphi(\Gamma_i)} {\rm Ext}_1|(p)\nonumber\\
&\leq 2\sg\int_{A_{i-1}}|\nabla_p^{\varphi(\Gamma_i)}K_s(p,q)| +2\sg\int_{B_{i+1}}|\nabla_p^{\varphi(\Gamma_i)}K_s(p,q)| +C\sigma\nonumber\\
&\leq C\sg\int_{\varphi^{-1}(A_{i-1})}\frac{1}{|x-y|^{n+s+1}} +C\sg\int_{\varphi^{-1}(B_{i+1})}\frac{1}{|x-y|^{n+s+1}} +C\sg\nonumber\\
&\leq C\sg\int_{\Gamma_{i-1}}\frac{1}{|x-y|^{n+s}}+C\sg\int_{\Gamma_{i+1}}\frac{1}{|x-y|^{n+s}}+C\sigma\nonumber\\
&\leq C\frac{\sigma}{|g_{i}(0)-g_{i-1}(0)|^{2}}+C\frac{\sigma}{|g_{i+1}(0)-g_{i}(0)|^{2}}+C\sigma\label{derhgbound}\, .
\end{align}
Using Lemma \ref{whtiohwoiwh} once again, we get the pointwise bound
\begin{equation}\label{DH-7b}
|\nabla_p^{\varphi(\Gamma_i)}H_s[\varphi(\Gamma_i)]|(p)\leq C \sigma^{-3}\, .
\end{equation}
Combining equations \eqref{H-3b} and \eqref{DH-7b}, we find that $\|H_s[\varphi(\Gamma_i)](\cdot)\|_{C^{0,1/2}(\Gamma_i \cap B'_{1/2}\times(-1,1))}\leq C\sg^{-3}$. We could have argued by interpolation between Hölder spaces and obtained a bound of order $\sg^{-2}$, but we do not care about the precise negative exponent at this point.

\noindent \textbf{Step 2.} Schauder estimates and interpolation.

Define $A_i=\varphi(\{x_n<g_i(x')\}\cap\Omega_0)$, so that $\varphi^{-1}(\partial A_i)\cap [B'_{1/2}\times (-1/2,1/2)]=\Gamma_i\cap[B'_{1/2}\times (-1/2,1/2)]$ is a $C^2$ graph over $B'_{1/2}$ parametrised by $g_i$.  By Step 1, it moreover verifies that 
$$\|H_s[\varphi(\Gamma_i)](\cdot)\|_{C^{0,1/2}(\varphi(\Gamma_i \cap [B'_{1/2}\times(-3/4,\,3/4)]))}\leq C\sg^{-3}.$$
Then, by the Schauder estimate for the nonlocal mean curvature equation in \cite{FrancS} (in its version for graphs), we deduce that $\|g_i\|_{C^{2,1/4}(\B_{1/4}')} \le C\sigma^{-3}$. On the other hand, by assumption \eqref{whtoihwohw2} we have that $\|g_i\|_{C^{2}(\B_1')} \le C$. By interpolation, we deduce that
$$\|g_i\|_{C^{2,\alpha}(\B_{1/4}')} \le C(\alpha)\sigma^{-12\alpha}$$
for every $\alpha\in(0,1/4)$, as desired.
\end{proof}

The next proposition gives a precise representation for the gradient of the singular kernel $K_s$ up to a lower order error. The fact that we manage to get an error of two orders less will be crucial in the proof of Lemma \ref{LandNMC}.
\begin{proposition}\label{gradcomp}
    Assume that the flatness assumptions ${\rm FA}_\ell(M,g,4,p,\varphi)$ are satisfied, where $\varphi$ corresponds to the restriction of the exponential map ${\rm exp}_p$ to $\B_4\subset\R^n$. Then, the difference between the gradient of $K_s$ and the gradient of its Euclidean representative is of two orders less. More precisely, using the notation in Proposition \ref{prop:kern1}, there is $r=r(n)>0$ such that if $y \in \B_r(0)$ then
    \begin{align*}
        \nabla_y K_s(p,\varphi_p(y))&=\nabla_y  \frac{\alpha_{n,s}}{|y|^{n+s}} + O\Big(\frac{1}{|y|^{n+s-1}}\Big)\\
        &=-(n+s)\alpha_{n,s}\, \frac{y}{|y|^{n+s+2}} + O\Big(\frac{1}{|y|^{n+s-1}}\Big)\,.
    \end{align*}
\end{proposition}
\begin{proof}
    The proof is given in Appendix \ref{app:gradk}.
\end{proof}

The next proposition simply records the standard fact that having flatness assumptions centered at some point implies that, in a smaller neighbourhood, normal coordinates centered at any other point also induce flatness assumptions. See also Remark 1.6 in \cite{SS}.
\begin{proposition}\label{fanormal} Assume that the flatness assumptions ${\rm FA}_\ell(M,g,4,p_0,\varphi)$ are satisfied. Then, there exists a small constant $r_\circ=r_\circ(n)>0$ such that for every $x\in\B_2$ and $r\in(0,r_0)$, letting $\varphi_q:\B_r\to M$ denote the restriction of the exponential map $exp_q$ centered at $q=\varphi(x)$, then the flatness assumptions $FA_l(M,g,r,q,\varphi_q)$ are also satisfied. Moreover, the change of coordinates $F_q=\varphi^{-1}\circ\varphi_q$ is a well-defined diffeomorphism onto its image and both $F_q$ and $F_q^{-1}$ have bounded $C^3$ norm.
\end{proposition}

The next lemma shows that the local and nonlocal mean curvatures, of a uniformly $C^2$ graph whose $C^{2,\alpha}$ seminorm is allowed to degenerate with a small rate as $\sigma\to 0$, are nevertheless close to each other.
\begin{lemma}\label{LandNMC}
    Let~$1-s =\sigma \in (0, 1/2)$, and assume that ${\rm FA}_\ell(M,g,4,p_0,\varphi) \s\mbox{is satisfied}$. Let $\alpha>0$, and let $g:\B'_{1/2}\to(-1/2,1/2)$ be a $C^{2,\alpha}$ function with 
    $$\|\nabla g\|_{L^\infty(\B_{1/2}')} + \|D^2 g\|_{L^\infty(\B_{1/2}')} < \delta \quad\mbox{and}\quad [g]_{C^{2,\alpha}}\leq C\sigma^{-{1/4}}\, .
    $$  
    Define $\Gamma$ as the graph of $g$ in $\B'_{1/2}\times(-1/2,1/2)$, so that $\varphi(\Gamma)$ is a $C^{2,\alpha}$ hypersurface in $M$. Then, the nonlocal mean curvature (NMC) of $\varphi(\Gamma)$ is close to its classical mean curvature (or local mean curvature, LMC) with a small error. Precisely, if $p=\varphi(x)$ with $x\in\Gamma\cap [\B'_{1/4}\times(-1/2,1/2)]$, then
    $$H_s[\varphi(\Gamma)](p)=c_\circ H[\varphi(\Gamma)](p)+O(\sg^{3/4})\,,
    $$
    where $c_\circ: = \frac{\mathcal H^{n-2}(\partial \B_1')}{2(n-1)}$.
\end{lemma}
\begin{proof}
    \textbf{Step 1.} We begin by showing that the Riemannian NMC of the set and the Euclidean NMC of its representative in normal coordinates are close.\\
    Let $x\in\Gamma\cap [\B'_{1/4}\times(-1/2,1/2)]$ be as in the statement of the lemma. By Proposition \ref{fanormal}, we can consider normal coordinates centered at $p=\varphi(x)$ and then our setting still applies. More precisely, $FA_l(M,g,r,p,\varphi_p)$ is satisfied for some $r$ sufficiently small. Moreover, 
    $$\hat \Gamma:=\Big[\varphi_p^{-1}\circ\varphi(\Gamma)\Big]\cap \Big[\B'_{r/2}\times(-r/2,r/2)\Big]$$ is still a $C^{2,\alpha}$ surface, which (possibly after a rotation and making $r$ sufficiently smaller) is given in $\B'_{r/2}\times(-r/2,r/2)$ by the graph of some function $\hat g:\B'_{r/2}\to(-r/2,r/2)$ with
    \begin{equation}\label{78uih76t}
    \hat g(0)=|\nabla \hat g(0)|= 0\,,
\end{equation}
and satisfying the bounds
    $$\|\nabla \hat g\|_{L^\infty(\B_{r/2}')} + \|D^2 \hat g\|_{L^\infty(\B_{r/2}')} < C\delta \quad\mbox{and}\quad [\hat g]_{C^{2,\alpha}}\leq C\sigma^{-{1/4}}\, .
    $$
    Now, thanks to taking normal coordinates, we have that $g_{ij}(0)=\delta_{ij}$ and $\nabla g_{ij}(0)=0$. This will simplify the computations that will follow. We start by writing
    \begin{align}
        \frac{1}{2}H_s[\varphi(\Gamma)](p)&=\frac{1}{2}H_s[\varphi_p(\hat\Gamma)](p)+\sigma\int_{\varphi(\Gamma)\setminus\varphi_p(\hat\Gamma)}\langle\nu_{\Sigma}(q),\nabla_q^M K_{s-2}(p,q)\rangle_g\,d\mathcal{H}^{n-1}(q)\nonumber\\
        &=\frac{1}{2}H_s[\varphi_p(\hat\Gamma)](p)+O(\sigma)\,,\label{dgnnb}
    \end{align}
    where we have used \eqref{remaining2} with $s-2$ instead of $s$ to bound the second term. This shows that we can work with $\hat\Gamma$ and $\varphi_p$ instead of $\Gamma$ and $\varphi$ without any loss of generality. 
    Now, passing to coordinates we can write
    \begin{align}\frac{1}{2}H_s[\varphi_p(\hat\Gamma)](p)&=\sigma\int_{\varphi_p(\hat \Gamma)}\langle\nu_{\varphi_p(\hat\Gamma)}(q'),\nabla_{q'}^M K_{s-2}(p,q')\rangle_g\,d\mathcal{H}^{n-1}(q')\nonumber\\
    &=\sigma\int_{\hat\Gamma}\langle\nu_{\varphi_p(\hat\Gamma)}(\varphi_p(y)),\nabla_{q'}^M K_{s-2}(\varphi_p(0),\varphi_p(y))\rangle_{g(\varphi_p(y))}\sqrt{|g_{ij}^{\hat\Gamma}|(y)}\,d\mathcal H^{n-1}(y)\label{inprodcomp0}\,.
    \end{align}
    A simple Riemannian computation, which we give in Appendix \ref{sec:riemcomput}, allows us to compute the terms in the inner product, obtaining that
    \begin{align}\label{inprodcomp}
        \frac{1}{2}H_s[\varphi(\hat\Gamma)](p)&=\sigma\int_{\hat\Gamma}\frac{1}{\sqrt{\nu_{\hat\Gamma}^t g^{-1} \nu_{\hat\Gamma}}}[\nu_{\hat\Gamma}(y)]^t g^{-1}(y) \nabla_y^{\R^n}K_{s-2}(\varphi_p(0),\varphi_p(y))\,\sqrt{|g_{ij}^{\hat\Gamma}|(y)}\,d\mathcal H^{n-1}(y)\, .
    \end{align}
    Here all products involving $g$ denote products of matrices. Now, since $g(y)=Id + O(|y|^2)$ thanks to having taken normal coordinates, we can exchange all instances of the Riemannian metric $g$ by the identity metric up to a controlled quadratic error. More precisely, we can estimate
    \begin{align}
        \frac{1}{2}H_s[\varphi(\hat\Gamma)](p)&=\sg\int_{\hat \Gamma}\big[\nu_{\hat\Gamma}(y)\cdot\nabla_y^{\R^n}K_{s-2}(\varphi_p(0),\varphi_p(y))\big]\,d\mathcal H^{n-1}(y)\\
        &\hspace{0.6cm} +O\Big(\sg\int_{\hat\Gamma}|y|^2|\nabla_y^{\R^n}K_{s-2}(\varphi_p(0),\varphi_p(y))|\,d\mathcal H^{n-1}(y)\Big)\nonumber\,,
    \end{align}
    where $\cdot$ denotes the usual Euclidean scalar product.
    
    Proposition \ref{gradcomp} (applied with $s-2$ instead of $s$) allows us now to compare the gradient of the kernel $K_{s-2}$ with the one of the corresponding Euclidean kernel. We reach
    \begin{align}
        \frac{1}{2}H_s[\varphi(\hat\Gamma)](p)&=-(n+s-2)\alpha_{n,(s-2)}\sg\int_{\hat \Gamma}\frac{\nu_{\hat\Gamma}(y)\cdot y}{|y|^{n+s}}\,d\mathcal H^{n-1}(y)\nonumber\\
        &\hspace{0.6cm} +O\Big(\sg\int_{\hat\Gamma}|y|^2\frac{1}{|y|^{n+s-1}}\,d\mathcal H^{n-1}(y)\Big)\, .\nonumber
    \end{align}
    Setting $c(n,s)=-\frac{2(n+s-2)}{s}\alpha_{n,(s-2)}$ and bounding the second integral by a uniform constant, we can then write
    \begin{align}
        H_s[\varphi(\hat\Gamma)](p)&=c(n,s)\sg\int_{\hat \Gamma}\frac{\nu_{\hat\Gamma}(y)\cdot y}{|y|^{n+s}}\,d\mathcal H^{n-1}(y)\nonumber\\
        &\hspace{0.6cm} +O(\sg)\nonumber\,.
    \end{align}
    Now, the last expression is exactly the expression for the Euclidean nonlocal mean curvature of $\hat\Gamma$, plus an $O(\sg)$ error term. In other words, together with \eqref{dgnnb} we have found that
    $$H_s[\varphi(\Gamma)](p)=H_s[\hat\Gamma](0)+O(\sg)\,,$$
    where the NMC on the left and right hand sides are respectively on $M$ and on $\R^n$.
    
    \noindent \textbf{Step 2.} We now show that $H_s[\hat\Gamma](0)=c_\circ H[\hat\Gamma](0)+O(\sigma^{3/4})$, with $c_\circ: = \frac{\mathcal H^{n-2}(\partial \B_1')}{2(n-1)}$. This is an exclusively Euclidean computation and which we reproduce from \cite{CDSV}, with the caveat that we are not assuming uniform $C^{2,\alpha}$ bounds for our graphs but rather bounds which degenerate as $\sigma\to 0$.

Observe first that, from the assumption that
\begin{equation}
\hat g(0)=|\nabla \hat g(0)|= 0\,,
\end{equation}
we have that
\[
H[\hat\Gamma](0) = - \Delta \hat g(0).
\]
Thus our goal is to relate $H_s[\hat\Gamma](0)$ with $\Delta \hat g(0)$.

Recall that
\[
H_s[\hat\Gamma](0) = c(n,s)\sigma \int_{\hat\Gamma}  \frac{\nu_{\hat\Gamma}(y)\cdot y}{|y|^{n+s}}\, d\mathcal H^{n-1}(y)  =  c(n,s)\sigma \int_{\B_{r/2}'}  \frac{(- \nabla \hat g(y'),1)\cdot (y',\hat g(y'))}{(|y'|^2 + \hat g(y')^2)^{\frac{n+s}{2}}}  \,dy'.
\]

Now, on the one hand, we have
\[
I_1: =\sigma \int_{B_\varrho'}  \frac{(- \nabla\hat g(y'),1)\cdot (y',\hat g(y'))}{(|y'|^2 + \hat g(y')^2)^{\frac{n+s}{2}}} \,dy' = (1+ O(\varrho^2))\sigma \int_{B_\varrho'}  \frac{\hat g(y')- \nabla \hat g(y')\cdot y'}{|y' |^{n+s}}\,dy',
\]
where we have used that $|y'|^2 + \hat g(y')^2 = |y'|^2(1 + \hat g(y')^2/|y'|^2) = |y|^2(1 + O(\varrho^2))$ for $y'\in B_\varrho'$,
thanks to~\eqref{78uih76t}.

Note that our assumptions on $[\hat g]_{C^{2,\alpha}}$ give the Taylor expansions
\[ \big|\hat g(y')- \tfrac1 2 y'\cdot  D^2\hat g(0) y'\big|  \le C\sigma^{-1/4}|y'|^{2+\alpha} \quad \mbox{and}\quad   \big|\nabla \hat g(y') - D^2 \hat g(0)y'\big|\le C\sigma^{-1/4}|y'|^{1+\alpha}.
\]
Hence,
choosing  $\varrho =\sigma$, and using that~$\sigma^\sigma = \exp(\sigma \log \sigma) = 1 + O(\sigma |\log \sigma|)$ as $\sigma\downarrow 0$, we obtain
\[
\begin{split}
I_1 &= (1+O(\varrho^2))\sigma\left( \int_{B_\varrho'}  \frac{- \tfrac 1 2 y'\cdot D^2 \hat g(0)y'}{|y'|^{n+s}}\,dy'  +  O\Big(C\sigma^{-1/4}\int_{B_\varrho'}  \frac{2 |y'|^{2+\alpha}}{|y'|^{n+s}}\,dy' \Big)\right)\\
& = (1+O(\varrho^2))\left( \frac{{\rm trace}(D^2 \hat g(0))}{2(n-1)} \sigma\int_{B_\varrho'}  \frac{|y'|^2}{|y'|^{n+s}}\,dy'  + \sigma^{3/4} O(\varrho^{\alpha +1-s}) \right)\\
& = (1+O(\varrho^2))\left(-\Delta \hat g(0) \frac{\cH^{n-2}(\partial B_1')}{2(n-1)} \sigma\frac{\varrho^{1-s}}{1-s} + \sigma^{3/4} O(\varrho^{\alpha +1-s}) \right)\\
& =  -c_\circ\Delta \hat g(0)  + O( \sigma |\log\sigma| + \sigma^{3/4+\alpha}).
\end{split}
\]

On the other hand, using that $|\hat g(y')| + |y'\cdot\nabla \hat g(y')|  \le \frac 1 2 |y'|^2 + |y'|^2$ (thanks to our assumptions on $\hat g$),
we obtain that
\[
I_2: =\sigma \int_{B_{r/2}'\setminus B_\varrho'}    \frac{(- \nabla \hat g(y'),1)\cdot (y',\hat g(y'))}{(|y'|^2 + \hat g(y')^2)^{\frac{n+s}{2}}}  \,dy'
\]
can be bounded by
\[
\big| I_2\big| \le \sigma \int_{B_{r/2}'\setminus B_\varrho'}  \frac{\frac 3 2 |y'|^2}{|y'|^{n+s}}\, dy'
= C \sigma \int_\varrho^1 r^{-s}\, dr= C(1-\sigma^\sigma) = O(\sigma |\log \sigma|).
\]
Since $H_s[\hat\Gamma](0) = I_1 + I_2$, these considerations complete the proof of Step 2.

\noindent \textbf{Step 3.} Conclusion of the proof.
    
    Combining Steps 1 and 2, we have seen that $$H_s[\varphi(\Gamma)](p)=c_0H[\hat\Gamma](0)+O(\sg^{3/4})\,.$$ On the other hand, the identity
    \begin{equation}\label{mcrequl}
        H[\hat\Gamma](0)=H[\varphi_p(\hat\Gamma)](p)
    \end{equation}
    holds, where the left and right mean curvatures are computed on $\R^n$ and on $M$ respectively, thanks to having taken normal coordinates centered at $p$; a proof is given in Appendix \ref{sec:riemcomput}. Therefore, we conclude that
    $$H_s[\varphi(\Gamma)](p)=c_0H[\varphi_p(\hat\Gamma)](p)+O(\sg^{3/4})=c_0H[\varphi(\Gamma)](p)+O(\sg^{3/4})$$
    as desired.
\end{proof}
With the results up to now, we can show the optimal separation estimate between layers in dimension $3$.
\begin{proposition}\label{optsep}
    Let $n=3$. Assume that \eqref{flatcond}--\eqref{whtoihwohw2} are satisfied, and let $E\subset M$ be $\Lambda$-almost stable $s$-minimal in $\varphi(\B_2)$, with~$(1-s)=\sigma \in (0,1/8)$. Then, there exists $c>0$, depending only on $\Lambda$, $n$ and the constants in \eqref{flatcond}--\eqref{whtoihwohw2}, such that for every $i$ such that \begin{equation*}\min_{\overline B_{1/2}'} |g_i| \le \frac12,\end{equation*}
    we have the separation estimate
\begin{equation} \label{wiowhoih2}
\inf_{B'_{1/2}} (g_{i+1}-g_i) \ge c \sigma^{1/2} >0\, .
\end{equation}
Moreover, we have the pointwise mean curvature estimates
\begin{equation}\label{wiowhoih3}
    \|H_s[\varphi(\Gamma_i)]\|_{L^\infty(\varphi(\Gamma_i \cap [B'_{1/2}\times(-1,1)]))}  \le C\sqrt{\sigma} 
    \quad \mbox{and} \quad
    \|H[\varphi(\Gamma_i)]\|_{L^\infty(\varphi(\Gamma_i \cap [B'_{1/2}\times(-1,1)]))}  \le C\sqrt{\sigma}.
\end{equation}
\end{proposition}

\begin{proof}

\textbf{Step 1.} Let us show the pointwise separation estimate in \eqref{wiowhoih2}.\\
It is well known that the classical mean curvature can be written, for a graph in coordinates, as an operator depending on the graph and the coefficients of the metric. Moreover, the difference of the parametrising functions of two parallel graphs satisfies an elliptic equation. The precise statement is that, for $x'\in B'_{3/4}$ and $v:= g_{i+1}-g_i$, we have that (see for example \cite[page 237]{CM})
\begin{align}
h(x´):&=H[\varphi(\Gamma_{i+1})](\varphi(x', g_{i+1}(x')))  - H[\varphi(\Gamma_{i})](\varphi(x', g_{i}(x')))\label{LKMS:DA}\\
&= a_{ij}(x')v_{x_ix_j}(x")+b_i(x')v_{x_i}+cv,\nonumber
\end{align}
where the $a_{ij}$ form a uniformly elliptic matrix and $b_i$ and $c$ are uniformly bounded.\\
Moreover, we know from Proposition \ref{L2NMC} that the NMCs $H_s[\varphi(\Gamma_{i})]$ and $H_s[\varphi(\Gamma_{i+1})]$ are bounded in $L^2$ by $C\sqrt{\sigma}$, which together with the combination of \eqref{wiowhoih567} (with $\alpha=\frac{1}{48}$, so that $12\alpha=1/4$) and Lemma \ref{LandNMC} gives that the LMCs $H[\varphi(\Gamma_{i})]$ and $H[\varphi(\Gamma_{i+1})]$ are bounded in $L^2$ by $C\sqrt{\sigma}$ as well. Thus \eqref{LKMS:DA} gives that 
$$\|h\|_{L^2(B_{3/4}')}\leq C\sqrt{\sigma}\,.$$
On the other hand, by the Harnack inequality for uniformly elliptic operators, see \cite[Theorems 8.17 and 8.18]{GT},
we have that
\[
\sup_{B_{1/2}'} v  \le  C \inf_{B_{1/2}'} v+ C\|h\|_{L^q(B_{3/4}')}
\]
for every $q> n'/2$, with $n':=n-1$. For $n=3$ we may take $q=2$, thus obtaining that
\begin{equation}\label{wtihoiewhw}
\sup_{B_{1/2}'} (g_{i+1}-g_i) \le  C \inf_{B_{1/2}'}  (g_{i+1}-g_i)  + C\sqrt \sigma.
\end{equation}

We now argue as in \cite[Proposition 3.5]{CDSV}. Still in dimension $n=3$ (and thus $n'=2$), assume that $ (g_{i+1}-g_i) (x_\circ')  = \delta \sqrt \sigma$ for some $x'_\circ\in B'_{1/2}$ and let us prove a lower bound for $\delta$.

For $r\in (0,1/4)$ we now dilate around $x_\circ = (x_\circ',g_i(x_\circ')) \in \Gamma_i$ and
we obtain new surfaces $\Gamma_{i,r} : = \frac 1 r (\Gamma_i-x_\circ)$
which have graphical expressions~$x_n = g_{i,r}(x')$ in~$\B_1'\times (-1,1)$, where~$g_{i,r}(x') : = \frac 1 r g_i(x_\circ' + rx')$. As we have used several times, our assumptions \eqref{flatcond}--\eqref{whtoihwohw2} and the $\Lambda$-almost stability condition are all preserved in this rescaled setting, considering the manifold $(M,\frac{1}{r^2}g)$ and the new parametrisation $\varphi(rx+x_0)$; see for example Step 2 of the proof of Proposition \ref{prop:BVestas}.\\
Then, reapplying the argument above we deduce that \eqref{wtihoiewhw} also holds with $g_i$ replaced by $g_{i,r}$. Since $(g_{i+1,r}-g_{i,r}) (0)  = \frac 1 r (g_{i+1}-g_i) (x_\circ')  =\frac{\delta \sqrt \sigma}{r}$, this gives that
\[
\sup_{B_{r/2}'(x_\circ')} (g_{i+1}-g_i)   = r  \sup_{B_{1/2}'} (g_{i+1,r}-g_{i,r}) \le  Cr\left(\frac{\delta \sqrt \sigma}{r}  + \sqrt \sigma\right).
\]
In other words, for all $r\ge \delta$ we have that $
\sup_{B_{r/2}'(x_\circ')} (g_{i+1}-g_i) \le Cr\sqrt \sigma$.

But then using again Proposition~\ref{L2NMC} and the fact that~$n'=2$, we obtain that
\begin{equation}\label{nwoihtoiwh}
C\ge \int_{B_{1/4}'(x_\circ') }  \frac{\sigma}{(g_{i+1}-g_i)^2}  
\ge  \frac{\sigma }{C}
\int_\delta ^{1/4} \frac{r \,dr}{(Cr\sqrt \sigma)^2} =\frac{1}{C}
\int_\delta ^{1/4} \frac{dr}{r}\ge \frac{|\log \delta|}{C},
\end{equation}
which proves that~$\delta \ge c>0$.

\noindent \textbf{Step 2.} Now that we know the pointwise separation estimate in \eqref{wiowhoih2}, using it in equation \eqref{nmcsep1} gives the pointwise bound for the NMC,
\[
\|H_s[\varphi(\Gamma_i)]\|_{L^\infty(\varphi(\Gamma_i \cap B'_{1/2}\times(-1,1)))}  \le C\sqrt{\sigma}.
\]
Combining this once again with Lemma \ref{LandNMC}, which we can apply thanks to \eqref{wiowhoih567} in Proposition \ref{NegC2alpha}, we complete the proof of~\eqref{wiowhoih3}.
\end{proof}
We finally obtain the uniform $C^{2,\alpha}$ estimate under the assumptions in the section.
\begin{corollary}\label{whtu292112-2}
Let $n=3$. Assume that \eqref{flatcond}--\eqref{whtoihwohw2} are satisfied, and let $E\subset M$ be $\Lambda$-almost stable $s$-minimal in $\varphi(\B_2)$, with~$(1-s)=\sigma \in (0,1/8)$. Given $\alpha\in(0,1)$, there exists $C>0$, depending only on $\alpha$, $\Lambda$, $n$ and the constants in \eqref{flatcond}--\eqref{whtoihwohw2}, such that for every $i$ such that \begin{equation*}\min_{\overline \B_{1/2}'} |g_i| \le \frac12,\end{equation*}
we have that
$\|g_i\|_{C^{2,\alpha}(\B_{1/2}')} \le C$.
\end{corollary}
\begin{proof}
In \eqref{derhgbound} from Proposition \ref{NegC2alpha} we found the pointwise bound
$$|\nabla_p^{\varphi(\Gamma_i)}H_s[\varphi(\Gamma_i)](p)|\leq C\frac{\sigma}{|g_{i}(0)-g_{i-1}(0)|^{2}}+C\frac{\sigma}{|g_{i+1}(0)-g_{i}(0)|^{2}}+C\sigma\, ,$$
which combined with the separation estimate \eqref{wiowhoih2} shows that
$$|\nabla_p^{\varphi(\Gamma_i)}H_s[\varphi(\Gamma_i)](p)|\leq C\, .$$
On the other hand, by \eqref{wiowhoih3} in Proposition~\ref{optsep} we know that
\[
\|H_s[\varphi(\Gamma_i)]\|_{L^\infty(\varphi(\Gamma_i \cap [B'_{1/2}\times(-1,1)]))}  \le C\sqrt{\sigma}.
\]
Then, by interpolation between Hölder spaces, we deduce that
\begin{equation*}
\|H_s[\varphi(\Gamma_i)]\|_{C^{\beta}(\varphi(\Gamma_i \cap [B'_{1/2}\times(-1,1)]))}  \le C\sigma^{(1-\beta)/2}
\end{equation*}
for all $\beta\in(0,1)$.

By the Schauder estimate for the nonlocal mean curvature equation in \cite{FrancS} (in its version for graphs), since the NMC $H_s$ is an operator of order $1+s$, we deduce that $\|g_i\|_{C^{2,1+s+\beta-2}(\B_{1/2}')} \le C(\beta)$. Given $\alpha$ as in the statement, choose some $\beta>\alpha$ in $(0,1)$, say  $\beta=\alpha+(1-\alpha)/2$. Choosing then $s_0=s_0(\alpha)$ in $(0,1)$ large enough, so that $1+s_0+\beta-2\geq\alpha$, we conclude that 
$$\|g_i\|_{C^{2,\alpha}(\B_{1/2}')} \le C(\alpha)$$ for all $s\in[s_0,1)$.

The proof for the range $s\in(7/8,s_0)$ instead, in which $\sigma>(1-s_0)>0$, is essentially trivial at this point. First, by the lower separation bounds that we have found (here any rough bound, such as the one in Lemma \ref{whtiohwoiwh}, suffices), there is a uniform $r_{\circ}>0$ such that our surface is a single graph in balls of radius $r_{\circ}$ for this range of $s$. Then, differentiating \eqref{nmcsplit} two times and arguing exactly as in the proof of Proposition \ref{NegC2alpha}, we get a uniform pointwise bound
\[
\|\nabla^{2,\varphi(\Gamma_i)} H_s[\varphi(\Gamma_i)]\|_{L^\infty(\varphi(\Gamma_i \cap [B'_{1/2}\times(-1,1)]))}  \le C\,,
\] or for any chosen derivative of $H_s$ for that matter. By interpolation, in particular we get a uniform bound
\begin{equation*}
\|H_s[\varphi(\Gamma_i)]\|_{C^{1,1/4}(\varphi(\Gamma_i \cap [B'_{1/2}\times(-1,1)]))}  \le C\, ,
\end{equation*}
which again by the Schauder estimate in \cite{FrancS} gives a bound for the $C^{3,1+s+(1+1/4)-3}=C^{3,s-3/4}$ norm of $g_i$. Since $s\geq 7/8>3/4$, in particular we find a bound for the $C^3$ norm of $g_i$, thus also a bound for its $C^{2,\alpha}$ norm as desired. This concludes the proof.
\end{proof}

\begin{lemma}\label{StableLocPerComp}
Let $\gamma_n$ be as in \eqref{gammadef}, and let $\Omega_0 = \B_1'\times(-1,1)\subset\R^n$. Assume that \eqref{flatcond}--\eqref{whtoihwohw2} are satisfied, and let $E\subset M$ be $\Lambda$-almost stable $s$-minimal in $\varphi(\B_2)$, with~$(1-s)=\sigma \in (0,1/8)$. Moreover, assume that the metric $g$ (in coordinates) and the Euclidean metric are $\delta$-close in $\B_4$, in the sense that 
$$\sup_{x\in \B_4}\sup_{i,j} |g_{ij}(x)-\delta_{ij}(x)|\leq \delta\,.$$
Then, there exists a function $\ep(t)$, depending only on $\Lambda$ and $n$, such that $\lim_{t\to 0} \ep(t)=0$ and
\begin{equation*}
    |(1-s)\textnormal{Per}_s(E;\varphi(\Omega_0))-\gamma_n\textnormal{Per}(E;\varphi(\Omega_0))|\leq \ep(\sigma)+\ep(\delta)[\textnormal{Per}(E;\varphi(\Omega_0))+(1-s)\textnormal{Per}_s(E;\varphi(\Omega_0))]
\end{equation*}
holds.
\end{lemma}
\begin{proof}
    Let $A_i=\varphi(\{x_n<g_i(x')\}\cap\Omega_0)$, $B_i=\varphi(\{x_n>g_i(x')\}\cap\Omega_0)$, and $E_i=A_i\setminus A_{i-1}$. In particular (up to possibly interchanging $E$ with its complement) it follows that $E\cap \varphi(\Omega_0)=\bigcup_{i\in 2\Z} E_i$. We can write the fractional perimeter as the contribution of each subgraph $A_i$ plus an error:
    \begin{align*}
        {\rm Per}_s(E;\Omega_0)&=2\int_E\int_{E^c}K_s(p,q)dV_pdV_q=\sum_{i\in 2\Z} 2\int_{E_i}\int_{E^c}K_s(p,q)dV_pdV_q\\
        &=\sum_{i\in 2\Z} 2\int_{E_i}\int_{E_{i-1}}K_s(p,q)dV_pdV_q+\sum_{i\in 2\Z}2\int_{E_i}\int_{E_{i+1}}K_s(p,q)dV_pdV_q\\
        &\qquad +O\Big(\sum_{i\in 2\Z} 2\int_{E_i}\int_{A_{i-3}\cup B_{i+2}}K_s(p,q)dV_pdV_q\Big)\\
        &=\sum_{j\in \Z} 2\int_{E_j}\int_{E_{j-1}}K_s(p,q)dV_pdV_q+O\Big(\sum_{i\in 2\Z} \int_{E_i}\int_{A_{i-3}\cup B_{i+2}}K_s(p,q)dV_pdV_q\Big)\\
        &=\sum_{j\in \Z} 2\int_{A_{j-1}}\int_{A_{j-1}^c}K_s(p,q)dV_pdV_q+O\Big(\sum_{i\in \Z} \int_{E_i}\int_{A_{i-2}\cup B_{i+1}}K_s(p,q)dV_pdV_q\Big)\\
        &=\sum_{j\in \Z} {\rm Per}_s(A_j;\Omega_0)+O\Big(\sum_{i\in \Z} \int_{E_i}\int_{A_{i-2}\cup B_{i+1}}K_s(p,q)dV_pdV_q\Big)\\
        &=\sum_{j\in \Z} {\rm Per}_s(A_j;\Omega_0)+O\Big(\sum_{i\in \Z} \int_{A_i}\int_{ B_{i+1}}K_s(p,q)dV_pdV_q\Big)\,.
    \end{align*}
    Thanks to the separation estimates, ${\rm dist}(A_i,B_{i+1})\geq c \sqrt{\sg}$, so that passing to coordinates in $q$ and using Lemma \ref{loccomparability} we can bound
    \begin{align*}
        \int_{A_i}\int_{ B_{i+1}}K_s(p,q)dV_pdV_q &\leq \int_{A_i}\int_{\varphi(\B_2)\setminus B_{c\sqrt{\sg}}^c(p)}K_s(p,q)dV_pdV_q\\
        &\leq \int_{A_i}\int_{\{|\varphi^{-1}(p)-y|\geq c\sqrt{\sg}\}}\frac{C}{|\varphi^{-1}(p)-y|^{n+s}}dV_pdy\\
        &\leq \frac{C}{|c\sqrt{\sg}|^s}\leq \frac{C}{\sqrt{\sg}}\leq \frac{C}{\sg}\ep(\sg)\,.
    \end{align*}
    A more precise computation would lead to the much better bound $C\log{\sg}$, but this suffices for our purposes since it shows that
    \begin{align}\label{fcpc1}
        (1-s){\rm Per}_s(E;\varphi(\Omega_0))
        &=\sum_{j\in \Z} (1-s){\rm Per}_s(A_j;\varphi(\Omega_0))+\ep(\sg)\,.
    \end{align}
    Moreover, we obviously have
    \begin{align}\label{fcpc11}
        {\rm Per}(E;\varphi(\Omega_0))
        &=\sum_{j\in \Z} {\rm Per}(A_j;\varphi(\Omega_0))\,.
    \end{align}
    Now, thanks to the assumption on the metric,
    \begin{equation}\label{fcpc2}
        |{\rm Per}(A_j;\varphi(\Omega_0))-{\rm Per}(\varphi^{-1}(A_j);\Omega_0)|\leq \ep(\delta){\rm Per}(A_j;\varphi(\Omega_0))\leq \ep(\delta)\,,
    \end{equation}
    and likewise
    \begin{equation}\label{fcpc22}
        |(1-s){\rm Per}_s(A_j;\varphi(\Omega_0))-(1-s){\rm Per}_s(\varphi^{-1}(A_j);\Omega_0)|\leq \ep(\delta)\,.
    \end{equation}
    On the other hand, the $\varphi^{-1}(A_j)$ are sets with uniformly $C^2$ (graphical) boundaries and which are transversal to $\partial\Omega_0$ in a quantified way thanks to their graphical structure. We can then apply \cite[Lemma 11]{CV11} separately to each of the $\varphi^{-1}(A_j)$, obtaining that
    \begin{equation}\label{fcpc3}
        |(1-s){\rm Per}_s(A_j;\varphi(\Omega_0))-{\rm Per}(A_j;\varphi(\Omega_0))|\leq \ep(\sg)\,.
    \end{equation}
    Combining \eqref{fcpc1}--\eqref{fcpc3}, we conclude the desired result.
\end{proof}

\subsubsection{Uniform $C^{2,\alpha}$ regularity for almost-stable surfaces in dimension $n=3$}
In the previous section we obtained a $C^2$-implies-$C^{2,\alpha}$ result. We will now prove that uniform $C^2$ (and thus also $C^{2,\alpha}$) estimates actually hold unconditionally. We start with a simple general lemma about writing sets with bounded curvatures locally as graphs.
\begin{lemma}[\textbf{Curvature bounds imply graphicality}]\label{Curvimgraph}
    Assume that the flatness assumptions ${\rm FA}_\ell(M,g,4,p_0,\varphi)$ are satisfied. Let $E\subset M$ be an $s$-minimal surface such that $|{\rm I\negthinspace I}_{\partial E}|  \le C_1$ in~$\varphi(\B_{2})$ and
    \begin{equation}\label{fpass1}
    \textnormal{Per}_s(E;\varphi(\B_2))\leq \frac{\kappa}{1-s}=\frac{\kappa}{\sigma}\, .
\end{equation}
Then, up to scaling the metric by a constant factor, performing a rotation in coordinates and considering a different value of $\kappa$, the hypotheses \eqref{flatcond}--\eqref{whtoihwohw2} are satisfied. More precisely, given $\delta>0$, there exist $C\geq 1$ (depending only on $n, C_1$ and $\delta$) and some rotation $\mathcal R:\R^n\to\R^n$ such that
\begin{equation}\label{flatcond0}
    {\rm FA}_\ell(M,Cg,4,p_0,\psi) \s\mbox{is satisfied} 
\end{equation}
with $\psi(x)=\varphi(\frac{1}{C}\mathcal Rx)$,
$E$
has fractional perimeter (as a subset of $(M,Cg)$)
\begin{equation}\label{locfperbound0}
    \textnormal{Per}_s(E;\psi(\B_2))\leq \frac{C\kappa}{1-s}=\frac{C\kappa}{\sigma}\, ,
\end{equation}
and moreover, denoting $\Omega_0 = \B_1'\times(-1,1)\subset\R^n$ as per usual,
\begin{eqnarray}\label{whtoihwohw10}
&&  \psi^{-1}(\partial E)  \mbox{ coincides with } \ \Gamma=\bigcup_{i=1}^N \,\Gamma_i = \bigcup_{i=1}^N \,\{x_n = g_i (x') \} \   \mbox{ inside }\Omega_0,\quad
{\mbox{for some }} g_i: \B_1'\to \R\\
&&\label{whtoihwohw20}
{\mbox{satisfying}}\quad
g_1 < g_2< \cdots < g_{N} \qquad \mbox{and} \qquad \|\nabla g_i\|_{L^\infty(\B_1')} + \|D^2 g_i\|_{L^\infty(\B_1')} < \delta\,.
\end{eqnarray}
\end{lemma}
\begin{proof}
    Let $x$ be the closest point in $\varphi^{-1}(\partial E)\cap \B_{\ep}$ to $0$, with $\ep$ to be selected later depending only on $n, C_1$ and $\delta$. If there is no such point, the Proposition is trivially true with $C=\frac{4}{\ep}$. Otherwise, perform a rotation $\mathcal R$ in coordinates so that $\mathcal R x=(0,...,0,|x|)$. Then, defining $F=\mathcal R(\varphi^{-1}(E)-x)$, we have that $0\in F$ and moreover the tangent space $T_0E'$ is oriented orthogonally to the vertical direction $e_n=(0,...,0,1)$, i.e. $T_0F=\{y\,:\,y_n=0\}$. Then, (the boundary of) $F$ satisfies the hypotheses \eqref{whtoihwohw10}--\eqref{whtoihwohw20} in place of $\Gamma$, up to possibly replacing $\B_1'\times(-1,1)$ and $\delta$ by $\B_{1/C_2}'\times(-1/C_2,1/C_2)$ and $C_2$ in their statements for some $C_2=C_2(C_1,n)$ --- this is a standard fact, which follows from arguing by Taylor expansion as in the proof of Lemma \ref{lem:kernel-comp} using the second fundamental form bound (which passes to $F$) and the fact that $T_0F=\{y_n=0\}$. Then, we can choose $C_3=C_3(C_2,n,\delta)$ so that, considering the rescaled set $C_3F$, the hypotheses \eqref{whtoihwohw10}--\eqref{whtoihwohw20} will be satisfied with $\B_1'\times(-2,2)$ and $\delta$ instead. Since $\mathcal R x=(0,...,0,|x|)$, we deduce that the set $C_3\mathcal R(\varphi^{-1}(E))=C_3F+C_3\mathcal R x$ satisfies the same hypotheses with $\B_1'\times (-2+C_3|x|,2+C_3|x|)\subset \B_1'\times (-2+C_3\ep,2)$ instead as the domain, so that letting $\ep=\frac{1}{C_3}$ we have shown that $C_3\mathcal R(\varphi^{-1}(E))$ satisfies \eqref{whtoihwohw10}--\eqref{whtoihwohw20} with $\B_1'\times(-1,1)$ and $\delta$.
    
    Defining $C=C_3$, so that $\psi(x)=\varphi(\frac{1}{C_3}\mathcal{R}^{-1}x)$, since $\psi^{-1}(\partial E)=C_3\mathcal R(\varphi^{-1}(E))$ we conclude that \eqref{whtoihwohw10}--\eqref{whtoihwohw20} hold just as written. Moreover, as we have repeatedly used during the article, the condition \eqref{flatcond0} then holds thanks to (b) in Remark \ref{flatscalingrmk}, and \eqref{locfperbound0} holds thanks to \eqref{fpergrowth1}. This finishes the proof of the lemma.
\end{proof}
We are now ready to prove the main theorem of this section.
\begin{theorem}\label{C2alpha}
   Let $M$ be a closed manifold of dimension $3$. Let $s_0\in (s_0,1)$ and $s\in(0,1)$. Assume that the flatness assumptions ${\rm FA}_\ell(M,g,4,p_0,\varphi)$ are satisfied, and  let $E$ be a $\Lambda$-almost stable $s$-minimal surface in $\varphi(\B_2)$. Assume, moreover, that $\textnormal{Per}_s(E;\varphi(\Omega_0))\leq \frac{\kappa}{1-s}=\frac{\kappa}{\sigma}$. Then $E$ satisfies uniform $C^{2,\alpha}$ estimates and uniform separation estimates of order $\sqrt\sigma$ in $\varphi(\B_1)$, in the sense that there are some positive constants $r$, $C$ depending on $n,\Lambda, s_0$ and $\kappa$ such that for all $x\in\B_1$, after a rotation $\varphi^{-1}(\partial E) \cap \B_r(x)$ is a union of $C^{2,\alpha}$ graphs $g_i$ with uniform estimates $\|g_i\|_{C^{2,\alpha}}\leq C$ and at vertical separation $\inf (g_{i+1}-g_i)\geq\frac{1}{C}\sqrt{\sigma}$.
\end{theorem}
\begin{proof}
The combination of Lemma \ref{Curvimgraph} with Corollary~\ref{whtu292112-2} and Proposition \ref{optsep} readily shows that, in order to prove the Theorem, it suffices to see that $|{\rm I\negthinspace I}_{\partial E}|  \le C$ in~$\varphi(\B_{3/2})$, where~$C$ depends on the constants in the statement of the Theorem.
Letting $F:=\varphi^{-1}(\partial E)$, for convenience we will actually prove that
\begin{equation}\label{whioghohw657rt3}
|{\rm I\negthinspace I}_{\partial F} (x)|{\rm dist}(x, \partial \B_{7/4})  \le C \quad \mbox{for all }x\in \partial F\cap \B_{7/4},
\end{equation}
which considering $x\in\B_{3/2}$ implies that $|{\rm I\negthinspace I}_{\partial E}|  \le C$ in~$\varphi(\B_{3/2})$ for a different choice of $C$ (determined by the flatness assumptions).

To prove \eqref{whioghohw657rt3}, assume by
contradiction that
there exist sequences~$s_k$, $E_k\subset M_k$  (satisfying the assumptions of Theorem~\ref{C2alpha}), with associated $F_k=\varphi_k^{-1}(\partial E_k)$, such that
\[\max_{x \in\partial F_k\cap \varphi(B_{7/4}) }  |{\rm I\negthinspace I}_{\partial F_k} (x)|{\rm dist}(x, \partial \B_{7/4})  =: \mathcal M_k \uparrow \infty,\]
and let $x_k$ denote a sequence of points where the previous maxima are attained.

Let us consider
\[
R_k : =  \frac{\mathcal M_k}{{\rm dist}(x_k,\partial \B_{7/4})}.
\]
By construction, $F_k$ satisfies $|{\rm I\negthinspace I}_{\partial F_k}(x_k)|  =R_k$. Moreover, letting $r={\rm dist}(x_k, \partial \B_{7/4})$, we have that $|{\rm I\negthinspace I}_{\partial F_k}|  \le 2R_k$ in~$\B_{r/2}(x_k)$.

Now, by Proposition \ref{fanormal}, up to performing a change of coordinates we can (and do) assume w.l.o.g. that moreover $x_k=0$ and $g_{ij}(0)=\delta_{ij}$, which will simplify the proof. We should also modify the quantities above accordingly by uniformly controlled multiplicative constants, but we omit this for clarity since it introduces no change in the argument.

Considering the rescaled set $\widetilde F_k=R_k F_k$, we then have that
$|{\rm I\negthinspace I}_{\partial \widetilde F_k}(0)|  =1$ and $|{\rm I\negthinspace I}_{\partial \widetilde F_k}|  \le 2$ in~$\B_{\mathcal M_k/2}$. Moreover, defining $\widetilde \varphi_k(\cdot)=\varphi_k(\frac{1}{R_k}\cdot)$, we can write $\widetilde F_k=\widetilde\varphi^{-1}_k(E_k)$, and then by (b) in Remark \ref{flatscalingrmk} the flatness assumptions ${\rm FA}_\ell(M_k,R_k^2g_k,4R_k,p_k,\widetilde \varphi_k)$ are satisfied. Furthermore, by \eqref{fpergrowth}, up to changing the value of $\kappa$ the bound \eqref{locfperbound} is satisfied. We recall that $\mathcal{M}_k\uparrow\infty$ by assumption.\\

We divide the remainder of the proof into two steps.
\vspace{3pt}

\noindent{\bf Step~1}. Let us show first that necessarily~$s_k \to 1$. Indeed, suppose (up to extracting a subsequence)
that~$s_k\to s \in[s_*,1)$.
Then, thanks to (uniform, since $s_k$ is bounded away from~1)  layer separation estimates  ---here we can even use some
rough layer separation estimate as in Lemma~\ref{whtiohwoiwh}--- and the standard $C^3$ regularity of fractional minimal
surfaces (e.g. arguing as in the last part of the proof of Corollary~\ref{whtu292112-2}), by Ascoli--Arzel\`a we would find that the~$\partial \widetilde F_k$
converge locally in~$C^2$ (as submanifolds of $\R^3$)
to an embedded hypersurface $\partial \widetilde F\subset\R^3$ of class $C^2$. Moreover, by Remark \ref{rmk:asresc} the $\partial E_k$ are $\Lambda/R_k^{n+s_k-2}$-almost stable in $(M_k,R_k^2 g_k)$, which since $R_k\to\infty$ gives immediately that $\partial \widetilde F$ must be a stable $s$-minimal surface. See e.g.  Step~2 of the proof of \cite[Proposition 6.1]{CCS} for a similar argument on how stability passes to the limit in the present case.
Hence by the Bernstein--type result for $C^2$ stable $s$-minimal surfaces in \cite[Corollary 2.12]{CCS}, $\widetilde F$ must be a half-space.
On the other hand, passing to the limit the equation  $|{\rm I\negthinspace I}_{\partial \widetilde F_k}(0)|  = 1$ we obtain, thanks to the~$C^2$ convergence, that~$|{\rm I\negthinspace I}_{\partial \widetilde F}(0)| = 1$, which is a contradiction.

\vspace{3pt}

\noindent{\bf Step~2}. In light of Step~1, from now on we assume that~$s_k\to 1$ (or equivalently that~$\sigma_k :=(1-s_k) \to 0$).
Fix~$R \ge 1$, and let~$\Gamma_{k,\ell}$, with~$1\le \ell\le N_k$,  be the connected components
of the~$C^{2}$ hypersurfaces~$\partial \widetilde F_k  \cap \B_{2R}$.

On the one hand, since $\widetilde F_k$ satisfies $|{\rm I\negthinspace I}_{\partial \widetilde F_k}|  \le 2$, the combination of Lemma \ref{Curvimgraph} and Corollary~\ref{whtu292112-2} (plus a scaling+finite covering argument as usual) shows that the $\partial \widetilde F_k$ satisfy uniform $C^{2,\alpha}$ estimates in $\B_R$. Moreover, by Lemma \ref{classperball} we deduce that the number of layers $N_k$ of $\partial \widetilde F_k$ in $B_{R}$ is bounded. Therefore, up to passing to a subsequence we can assume that $N=N_k$ is constant.

Then, by Ascoli-Arzelà, up to passing to a further subsequence (which again we do not relabel) as $k\to\infty$ each of the $\{\Gamma_{k,l}\}_{l=1}^N$ converge simultaneously as $C^2$ submanifolds to limit surfaces $\{\widetilde \Gamma_l\}_{l=1}^N$. Moreover, since for a fixed $k$ the $\{\Gamma_{k,l}\}_{l=1}^N$ are disjoint, the $\widetilde \Gamma_l$ can only possibly touch tangentially, but they cannot traverse each other. Now, by \eqref{wiowhoih3} in Proposition \ref{optsep}, as $k\to\infty$ the local mean curvatures of the $\widetilde\varphi(\Gamma_{k,l})$ are converging to zero, which together with the fact that the metric $g_{ij}$ in coordinates is converging to the Euclidean metric (recall that we could assume that $g_{ij}(0)=\delta_{ij}$, and combine this with \eqref{hsohoh2} with $R=R_k$ large) shows that the local (i.e. classical) mean curvatures of the $\Gamma_{k,l}$ are converging to zero as well. Thus, the $\widetilde \Gamma_l$ are classical minimal surfaces. This implies, by the maximum principle for classical minimal surfaces, that for $i\neq j$ the surfaces $\widetilde \Gamma_i$ and $\widetilde \Gamma_j$ are actually either identical or disjoint.

Let the multiplicity $n_l$ denote the number of layers that have collapsed onto the same limit $\Gamma_l$; it is clear then that the $\partial \widetilde F_k$ converge in the varifold sense to the varifold $V=\sum_l n_l \Gamma_l$ in $B_R$ (perhaps after decreasing $R$ slightly to ensure that the $\Gamma_l$ are transversal to $\partial\B_R$). Moreover, from this and Lemma \ref{StableLocPerComp}, the mass $\|V\|$ of this varifold satisfies
\begin{align}
    \gamma_n \|V\|_{\B_{R}}&=\gamma_n\sum_l n_l \textnormal{Per}(\Gamma_l,\B_{R})\label{vmaspf}\\
    &=\gamma_n\lim_k \textnormal{Per}(\widetilde F_k,\B_R)=\lim_k (1-s_k)\textnormal{Per}_s(\widetilde F_k,\B_{R})\nonumber\\
    &=\lim_k (1-s_k)\textnormal{Per}_s(E_k,\widetilde\varphi_k(\B_{R}))\,,\nonumber
\end{align}
where the last fractional perimeter is computed on $(M_k,R_k^2g_k)$. The last equality follows, once again, from the fact that the metric coefficients are converging to the Euclidean ones in the coordinates given by $\widetilde\varphi_k$, as shown by a change of variables and our estimates for the kernel $K_s$ (see Proposition 4.3 and Lemma 4.4 in \cite{CFS} for full details).\\
Note that, in particular, from the expression above we deduce that
\begin{equation}\label{agrle}
\textnormal{Per}(\cup_l \Gamma_l;\B_{R})\leq CNR^{n-1}=CR^{n-1}\, .
\end{equation}
\textbf{Claim.} $V$ (or equivalently, each of the $\Gamma_l$) is a stable classical minimal surface.
\begin{proof}[Proof of the claim]
    More generally, we have the following. Given a vector field $X\in C_c^\infty(B_\rho;\R^n)$, define $X_k:=(\widetilde \varphi_j)_* X$ and extend it by zero to a vector field on $\widetilde M_k$. Fix $t\in\R$, and let $\psi_X^t$ and $ \psi_{X_k}^t$ denote their corresponding flows at time $t$. Then,
    \begin{equation}\label{q5etgqwa4e50}
        \frac{d^\ell}{d t^\ell}(1-s_k)\textnormal{Per}_{s_k}(\psi_{X_k}^t(E_k);\widetilde\varphi_k(\B_R))\to \gamma_n\frac{d^\ell}{d t^\ell}\|(\psi_X^t)_\#V\|_{\B_R}\, \quad \mbox{as } k\to\infty\, .
    \end{equation}
The proof is as the one for \eqref{convItep}, defining the single variable functions $f_k(t)=(1-s_k)\textnormal{Per}_{s_k}(\psi_{X_k}^t(E_k);\widetilde\varphi_k(\B_R))$ and $f(t)=\gamma_n\|(\psi_X^t)_\#V\|_{\B_R}$, showing that $f_k(t)\to f(t)$ as $k\to\infty$ thanks to Lemma \ref{StableLocPerComp} as in \eqref{vmaspf}, and applying Lemma \ref{enboundslemma2} to see that the functions $\frac{d^\ell}{dt^\ell}f_k(t)$ are locally uniformly bounded. See \cite[Lemma 4.4]{CFS} for full details.
\end{proof}
With the claim at hand, we can apply the curvature estimates for $C^2$ stable minimal hypersurfaces with area bounds of type \eqref{agrle}, see \cite{SS}, and deduce that $|{\rm I\negthinspace I}_{\cup_l \Gamma_l}|  \le \frac{C}{R}$ in $B_{R/2}$. On the other hand, let~$\Gamma_{k,0}$ denote the component of $\partial \widetilde F_k$ that passes through the origin (recall that $0\in\partial \widetilde F_k$), which verifies that $|{\rm I\negthinspace I}_{\Gamma_{k,0}}(0)| = 1$. Thanks to the~$C^2$ convergence, we deduce that~$|{\rm I\negthinspace I}_{\Gamma_0}(0)| = 1$ as well, where $\Gamma_0$ denotes the component of $\cup_l \Gamma_l$ obtained as a limit of the~$\Gamma_{k,0}$. Choosing $R$ large enough so that $\frac{C}{R}<1$, this clearly gives a contradiction.
\end{proof}

\subsection{Convergence of s-minimal surfaces with bounded index to classical minimal surfaces}\label{sec:convbind}
We continue to assume that $l=n+2$
in this section---recall \eqref{flatcond}.

\subsubsection{Estimates for $s$-minimal surfaces with bounded index}

Up to now, we have obtained estimates in the almost-stable case. We will now focus on the finite Morse index case. The pointwise $C^2$ estimates and pointwise classical mean curvature estimates do not necessarily hold anymore in the bounded index case, since we expect that neck-pinching may occur by analogy with classical minimal surfaces, as in the case of the rescalings of a catenoid converging to a hyperplane. On the other hand, as recorded in the next Proposition, the powerful covering argument in Lemma \ref{morsecovering} will remarkably allow us to obtain a bound for the classical perimeter and a decay for the $L^1$-norm of the mean curvature also in this case, since these are integral quantities with a positive power-type scaling. In particular, we will use this in the next subsection to obtain the convergence, in the sense of varifolds, to a stationary limit in a very simple manner.

\begin{proposition}\label{LocBoundFMI}
Let $M$ be a closed manifold of dimension $3$, and assume that the flatness assumptions ${\rm FA}_\ell(M,g,4,p,\varphi)$ are satisfied. Let $E$ be an $s$-minimal surface in $\varphi(\B_2)$, with Morse index bounded by $m$ and with

\begin{equation}\label{fpas1}
    \textnormal{Per}_s(E;\varphi(\B_2))\leq \frac{\kappa}{1-s}=\frac{\kappa}{\sigma}\, .
\end{equation}
Then, there exists $C>0$, depending only on $n$, $m$ and $\kappa$, such that
\begin{equation}\label{perindbound}
\textnormal{Per}(E;\varphi(\B_1))\leq C
\end{equation}
and
\begin{equation}\label{has1}
    \int_{\partial E\cap \varphi(\B_1)}|H_{\partial E}|\leq C\sqrt{1-s}\, .
\end{equation}
\end{proposition}
\begin{proof}
\textbf{Step 1.} We first prove \eqref{perindbound} in the case where $E$ is actually $\Lambda$-almost stable in $\varphi(\B_2)$, for some fixed $\Lambda$ (say $\Lambda=1$), with $C$ depending on $\Lambda$.

First, observe that by Theorem \ref{C2alpha}, we know a uniform bound on the second fundamental form of $\partial E$ in $\varphi(\B_1)$. Therefore, in the rescaled setting given by Lemma \ref{Curvimgraph} the hypotheses \eqref{flatcond0}--\eqref{whtoihwohw20} are satisfied, which shows that Lemma \ref{classperball} is applicable in this rescaled setting and thus \eqref{perindbound} holds. Scaling back this information, what we have shown is precisely that if $E$ is $\Lambda$-almost stable in $\varphi(\B_2)$, then it satisfies that
$$\textnormal{Per}(E;\varphi(B_{\frac{1}{2C}}))\leq C\,,$$
where $C$ has the right dependencies. As usual, a scaling + finite covering argument then gives that, up to changing the value of $C$, actually
$$\textnormal{Per}(E;\varphi(B_1))\leq C\,,$$
which concludes Step 1.

\noindent \textbf{Step 2.} We prove \eqref{perindbound} under the hypotheses of the present Proposition \ref{LocBoundFMI}.

By the same arguments above, the result in Step 1 rescales appropriately upon looking at smaller scales and centering at different points. Then, the argument using Lemma \ref{morsecovering} that we employed in the proof of Theorem \ref{BVest} (perimeter estimate in the finite index case) to deduce the latter from Proposition \ref{prop:BVestas} (perimeter estimate in the $\Lambda$-almost stable case) applies almost word-by-word to the present situation, and thus shows \eqref{perindbound} in our finite index setting.

\noindent \textbf{Step 3.} We also prove \eqref{has1}, which corresponds to a bound on the $L^1$ norm of the mean curvature, under the hypotheses of the present Proposition \ref{LocBoundFMI}.

To do so, we follow a similar strategy to Steps 1--2. First, assume that $E$ is $\Lambda$-almost stable in $\varphi(\B_2)$ for some fixed $\Lambda$, just as in Step 1. Then \eqref{has1} is indeed true: By the combination of Theorem \ref{C2alpha} and Lemma \ref{Curvimgraph}, after rescaling the metric appropriately the hypotheses of Proposition \ref{optsep} are satisfied. Applying \eqref{wiowhoih3}, scaling back the resulting information and performing a finite covering argument, we deduce that the mean curvature $H_{\partial E}$ satisfies a uniform pointwise bound of type $C\sqrt{\sg}$ in $\varphi(\B_1)$. Moreover, by Step 1 itself we know that $\partial E$ has bounded classical perimeter in $\varphi(\B_1)$. Therefore, integrating $H_{\partial E}$ over $\partial E\cap\varphi(\B_1)$ we see that if $E$ is $\Lambda$-almost stable in $\varphi(\B_2)$ for some fixed $\Lambda$, then
\begin{equation*}
    \int_{\partial E\cap \varphi(\B_1)}|H_{\partial E}|\leq C\sqrt{1-s}\, .
\end{equation*}
The finite index case follows from the $\Lambda$-almost stable case, which we have just proved, by arguing as in Step 2. Indeed, observe that the statement we have just shown can be appropriately rescaled, upon looking at smaller scales and centering at different points. More precisely, the $L^1$ norm of the mean curvature of a hypersurface scales, upon zooming in by a factor $r$, with power $r^{n-2}$, which is a positive power of $r$ for $n=3$ or higher. Then, the covering argument in Step 2 of the proof of Theorem \ref{BVest} can also be applied in our case: Indeed, Lemma \ref{morsecovering} applies to subadditive quantities which satisfy a positive power-type scaling $r^\beta$, which includes the case of the $L^1$ norm of the mean curvature of a hypersurface and $\beta=n-2$. This gives the desired result.
\end{proof}

With the results up to now, we can obtain our first global result.
\begin{proposition}
\label{GlobBoundFMI} Let $E\subset M$ be an $s$-minimal surface with index at most $m$ and a fractional perimeter bound $\textnormal{Per}_s(E;M)\leq\frac{\kappa}{1-s}$. Then, there exists a constant $C=C(M,m,\kappa)$ such that 
\begin{equation*}\textnormal{Per}(E;M)\leq C\quad\mbox{and}\quad \int_{\partial E\cap M}|H_{\partial E}|\leq C\sqrt{1-s}\, .
\end{equation*}
\end{proposition}

\begin{proof}
Since $M$ is closed, there exists a number $\delta>0$ with the property that, given any $p\in M$, the flatness assumptions ${\rm FA}_\ell(M,g,4\delta,p,\varphi_p)$ are satisfied for some $\varphi_p$; see Remark \ref{fbsvdg}. Having fixed such a $\delta>0$, by compactness we can find a finite collection of points $p_1,...,p_N$ such that the sets $\varphi_{p_1}(\B_{\delta/32}(p_1)),...,\varphi_{p_N}(\B_{\delta/32}(p_N))$ cover $M$. Moreover, by \eqref{fpergrowth} in Lemma \ref{PerMon} and the assumption on $\textnormal{Per}_s(E;M)$, up to having chosen $\delta$ small enough we deduce that $\textnormal{Per}_s(E;\varphi_{p_i}(\B_{\delta/16}(p_i)))\leq \frac{C}{1-s}\delta^{n-s}$ for each $i$. Applying then Proposition \ref{LocBoundFMI}, appropriately rescaled (i.e. considering the manifold $(M,(\delta/32)^{-2}g)$, by our usual argument), and scaling back, we deduce that
\begin{equation}\label{cpergrowth}
    \textnormal{Per}(E;\varphi_{p_i}(\B_{\delta/32}(p_i)))\leq C\delta^{n-1}
\end{equation}
and
$$
\int_{\partial E\cap \varphi_{p_i}(\B_{\delta/32}(p_i))}|H_{\partial E}|\leq C\sqrt{1-s}\,\delta^{n-2}\, ,$$
where $C$ has the dependencies indicated in the statement of the proposition.

Adding up these inequalities for $i=1,...,N$, since the $\varphi_{p_i}(\B_{\delta/32}(p_i))$ cover $M$ we reach the conclusion of the theorem.
\end{proof}

\subsubsection{Convergence to a smooth classical minimal surface -- Proof of Theorem \ref{FMIConv}}
This subsection proves, in several steps, the subsequential convergence (as $s\to 1$) of sequences of $s$-minimal surfaces with bounded Morse index and fractional perimeters to limiting smooth, classical minimal surfaces, culminating in the proof of Theorem \ref{FMIConv}.

In fact, using the results in the previous section, we can readily prove the convergence to a limit stationary varifold. The remainder of the section will then focus on proving its regularity.

\begin{theorem}\label{limvarcor}
Let $M$ be a closed manifold of dimension $3$. Let $s_i\to 1^-$ as $i\to\infty$, and let $E_i\subset M$ be an associated sequence of $s_i$-minimal surfaces with index $\leq m$ and $\limsup_i (1-s_i)\textnormal{Per}_{s_i}(E_i;M)\leq C_1$. Then, a subsequence of the $\partial E_i$ converges, in the varifold sense, to a limit integral stationary varifold $V$. Moreover, the $\partial E_i$ in the subsequence converge in the Hausdorff sense to $\Sigma:={\rm supp}(V)$.
\end{theorem}
\begin{proof}
\textbf{Step 1.} Varifold convergence.

The convergence in the varifold sense of a subsequence (not relabeled) of the $\partial E_i$ to a limit integral varifold $\Sigma$ follows directly from Proposition \ref{GlobBoundFMI}, by Allard's compactness theorem (\cite{All}). To check the stationarity, let $X$ be a smooth vector field on $M$, and denote its flow at time $t$ by $\psi_X^t$. By Proposition \ref{GlobBoundFMI}, there is a uniform constant $C$ such that we can bound
\begin{equation}\label{Hl1bound}
    \Bigg|\frac{d}{dt}\Big|_{t=0}\textnormal{Per}(\psi_X^t(E_i);M)\Bigg|=\Bigg|\int_{\partial E_i}X\cdot\vec H_{\partial E_i}\Bigg|\leq \|X\|_{C^0}\int_{\partial E_i}|H_{\partial E_i}|\leq C\sqrt{1-s_i}\|X\|_{C^0}\, .
\end{equation}
Letting $i\to\infty$, since $\sqrt{1-s_i}\to 0$ the convergence in the varifold sense shows that the limit varifold $V$ has vanishing first variation, or in other words it is stationary.

\noindent \textbf{Step 2.} Convergence in Hausdorff distance.

Let $\delta>0$ be such that the flatness assumptions ${\rm FA}_\ell(M,g,4\delta,p,\varphi_p)$ are satisfied for every $p\in M$, and such that $\varphi_p$ corresponds to the exponential map in normal coordinates; see Remark \ref{fbsvdg}. Up to arguing in what follows over the rescaled manifold $(M,\frac{1}{\delta^2}g)$ instead, we can also just assume that actually ${\rm FA}_\ell(M,g,4,p,\varphi_p)$ holds.\\
Now, by the varifold convergence proved in Step 1, to obtain the Hausdorff convergence it suffices to prove that the $E_i$ satisfy lower perimeter estimates, meaning that there exists a constant $c=c(M)>0$ such that (for $i$ large enough) if $p\in\partial{E_i}$ then
\begin{equation*}
    cr_2^{n-1}\leq \textnormal{Per}(E;\varphi_p(\B_{r_2})) 
    \end{equation*}
    for all $0<r_2<c$. To prove that this holds, fix $i$, and let $p\in\partial{E_i}$. Since the fixed set $E_i$ has a smooth boundary, performing a blow-up of $E_i$ around $p$ we see that the corresponding sequence of rescalings converges locally (in normal coordinates, say) in $H^{s_i/2}$ norm to a half-space in $\R^n$. Scaling back this information, this shows that 
$$\liminf_{r_1\to 0}\frac{\textnormal{Per}_{s_i}(E_i;\varphi_p(\B_{r_1}))}{r_1^{n-s}}= c(n,s_i)\, ,
$$
where the universal constant
$$
c(n,s_i):=\textnormal{Per}_{s_i}(\{x_n\leq 0\};\B_1)\geq \frac{c(n)}{1-s_i}>0
$$
comes from the contribution in $\B_{1}\subset\R^n$ to the fractional perimeter of a half-space. Multiplying by $(1-s_i)$ and applying Lemma \ref{PerMon}, we obtain that
\begin{equation*}
        c(n)\leq (1-s_i)\liminf_{r_1\to 0}\frac{\textnormal{Per}_{s_i}(E_i;\varphi_p(\B_{r_1}))}{r_1^{n-s}}\leq C(1-s_i)+C\frac{(1-s_i)\textnormal{Per}_{s_i}(E_i;\varphi_p(\B_{r_2}))}{r_2^{n-s}}
    \end{equation*}
    for all $r_2\leq 1$. Selecting $i$ large enough so that $C(1-s_i)\leq c(n)/4$, we deduce that
    \begin{equation*}
    cr_2^{n-s}\leq (1-s_i)\textnormal{Per}_{s_i}(E_i;\varphi_p(\B_{r_2})) 
    \end{equation*}
    for some $c=c(n)>0$. Combining this with the interpolation result of \eqref{interpropRn}, we conclude that
    \begin{equation*}
    cr_2^{n-1}\leq \textnormal{Per}(E_i;\varphi_p(\B_{r_2}))\, , 
    \end{equation*}
    which is precisely what we wanted to prove. 
\end{proof}
The rest of the section focuses on proving the smoothness of $\Sigma$. We begin with some preliminary results. The first one is the following lemma, which shows that the Morse index can concentrate at most at $m$ many points along a subsequence.

\begin{lemma}[\textbf{Almost stability outside of $m$ points}]\label{aspoints} Let $E_i$ be as in Theorem \ref{limvarcor}. Then, up to passing to a subsequence $E_{i_j}$, there exist points $p_1,...,p_l$ in $M$ with $l\leq m$ such that the following holds: Given any $p\neq p_1,...,p_l$, there exists a radius $r_p>0$ such that the $E_{i_j}$ are eventually $\Lambda_p$-almost stable on $B_{r_p}(p)$ for some $\Lambda_p<\infty$.
    
\end{lemma}
\begin{proof}
Define the concentration scale by 
\begin{equation*}
R(p,i,C)=\sup\{ r>0: \mbox{ a.s. ineq. } \eqref{lamasineq} \mbox{ holds for } E_i \mbox{ on } B_r(p) \mbox{ with constant } C \}\, .
\end{equation*}
Define also the "nonlocal index concentration set" of the $E_i$ (or equivalently of the associated sequence of indices $\{i\}$) as the set
\begin{equation*}
IC(\{i\})=\{ p\in M\ :\ \limsup_C\, \liminf_i \,R(p,i,C)\,=\,0\}.
\end{equation*}
\noindent \textbf{Step 1}. We show that, up to passing to a subsequence, the nonlocal index concentration set has at most $m$ points.

Indeed, assume by contradiction that every subsequence has at least $m+1$ points in its concentration set. In particular, for our entire sequence $E_i$ there exist $p_1,...,p_{m+1}$ in the nonlocal index concentration set. We focus on $p_1$. Given $j\in\N$, define $C_j=j$ and $\ep_j=1/j$; by definition, there exists a subsequence of indices $\{i_j^1\}_j$ such that $R(p_1,i_j^1,C_j)<\ep_j$, or in other words $R(p_1,i_j^1,j)<1/j$. This implies that $q_1:=p_1$ belongs to the nonlocal index concentration set of the subsequence of surfaces $E_{i_j}$, i.e. $q_1\in IC(\{i_j^1\}_j)$.

Now, if $m>0$ (otherwise the Lemma is trivial), by our contradiction assumption $E_{i_j^1}$ needs to have another point $q_2$ in its index concentration set, since it must have at least $(m+1)\geq 2$ such points in total. Therefore, there exists some subsequence of the $\{i_j^1\}_j$, denoted by $\{i_j^2\}_j$, such that $R(q_2,i_j^2,j)<1/j$. Moreover, since $i_j^2\geq i_j^1$ it is immediate to see that $R(q_1,i_j^2,j)<1/j$ still holds.

Now, both $q_1$ and $q_2$ are in the index concentration set $IC(\{i_j^2\}_j)$. If $(m+1)>2$, by our contradiction assumption there must exist another point $q_3$ which is also in the set, and we can proceed as before; iterating this argument $(m+1)$ times, we find $q_1,...,q_{m+1}$ and a subsequence $\{i_j^{m+1}\}_j$ as above such that $R(q_1,i_j^{m+1},j)<1/j,...,R(q_{m+1},i_j^{m+1},j)<1/j$.

Select now $\delta$ small enough so that the balls $B_{\delta}(q_1),...,B_{\delta}(q_{m+1})$ are at a positive distance $d>0$ from each other. The almost-stability inequality in Lemma \ref{almoststab}, applied with $E=E_{i_j^{m+1}}$ and $\mathcal{U}_k=B_\delta(q_k)$ for $k=1,...,m+1$, says that $E_{i_j^{m+1}}$ is $\Lambda$-almost stable on one of the balls, say $B_{\delta}(q_{k_j})$, with a constant $\Lambda<\infty$ independent of $j$. Then, up to passing to a subsequence, we can assume that the $E_{i_j^{m+1}}$ are actually all $\Lambda$-almost stable on the same ball, say $B_{\delta}(q_1)$. In other words, $\delta\leq R(q_1,i_j^{m+1},\Lambda)$ for all $j$. Taking $j$ large enough so that $j\geq\Lambda$, we deduce that 
$$\delta \leq R(q_1,i_j^{m+1},\Lambda)\leq R(q_1,i_j^{m+1},j)<1/j\, .$$
But then, by taking $j$ large enough so that also $1/j<\delta$, we reach a contradiction. This concludes the proof of the claim in Step 1.

\noindent \textbf{Step 2}. Conclusion of the proof.

Let $\{i_j\}_j$ be the subsequence given by the statement of Step $1$, and let $p_1,...,p_l$ be all the points in the nonlocal index concentration set of the $E_{i_j}$, so that $l\leq m$. Consider a point $p$ different from the $p_1,...,p_l$, so that it is not in the nonlocal index concentration set. Then, by definition, there exist a radius $r_p>0$ and numbers $\Lambda(p)$ and $j_0(p)$ such that, for all $j\geq j_0$, we have that $R(p,i_j,\Lambda)\geq r_p$, or in other words the $E_{i_j}$ are eventually $\Lambda(p)$-almost stable on $B_{r_p}(p)$. This concludes the proof of the lemma.
\end{proof}
The next proposition shows the regularity of $\Sigma$, the support of $V$, outside of the finitely many points where the index can concentrate.
\begin{proposition}\label{LocStableConv} Let $V$ be the limit integral stationary varifold obtained in Theorem \ref{limvarcor} as a subsequential limit of the $\partial E_i$. Then there exist points $p_1,...,p_l$ in $M$, with $l\leq m$, such that $\Sigma={\rm supp}(V)$ corresponds to a smooth classical minimal surface outside of $p_1,...,p_l$. Moreover, up to passing to a subsequence, the $\partial E_i$ converge as (multi-sheeted) $C^{2}$ local normal graphs over any compact subset of $\Sigma\setminus\{p_1,...,p_l\}$.
\end{proposition}
\begin{proof}
Let $p_1,...,p_l$ be as in Lemma \ref{aspoints}, and pass to a subsequence (still labeled $E_i$) verifying the conclusions of both Theorem \ref{limvarcor} and Lemma \ref{aspoints}. Then, given $p\neq p_1,...,p_l$, for $i$ large enough the $E_i$ are $\Lambda(p)$-almost stable in $B_{r_p}(p)$ for some $\Lambda(p)<\infty$ and $r_p>0$. By Theorem \ref{C2alpha}, which (after scaling) we can apply with some $\varphi_p$ up to choosing a smaller value for $r_p$, the $E_i$ satisfy then uniform $C^{2,\alpha}$ estimates in $B_{r_p}(p)$. Moreover, by Proposition \ref{LocBoundFMI} (rescaled appropriately), the $E_i$ have perimeter bounded by $Cr_p^{n-1}$ inside the ball.

    Therefore, by Lemma \ref{Curvimgraph}, up to making $r_p$ even smaller the set $\varphi^{-1}(\partial E_i)$ is made of only a finite, bounded number $d_i$ of uniformly $C^{2,\alpha}$ graphs $\{\Gamma_{i,l}\}_{l=1}^{d_i}$. Moreover, by \eqref{wiowhoih3} the mean curvatures of the $\varphi(\Gamma_{i,l})$ are going to $0$ as $s\to 1$. Then, by standard Ascoli-Arzelà arguments, the $E_i$ converge as $C^2$ submanifolds to $\Sigma={\rm supp}(V)$ in $B_{r_p}(p)$, which shows that $\Sigma$ is a union of ordered $C^2$ minimal hypersurfaces around $p$. Then, the maximum principle implies that these minimal hypersurfaces are actually disjoint as well. Moreover, by the usual regularity estimates for the classical minimal surface equation, we can bootstrap the $C^2$ regularity and deduce that $\Sigma$ is actually a union of ordered, disjoint smooth ($C^\infty$) submanifolds in $B_{r_p}(p)$. Since the choice of $p\neq p_1,...,p_l$ was arbitrary, we conclude that $\Sigma$ is smooth everywhere except possibly at the points $p_1,...,p_l$.\\
    
    Let now $K$ be a compact subset $K\subset\Sigma\setminus\{p_1,...,p_l\}$. By the above, the $E_i$ eventually satisfy $C^{2,\alpha}$ estimates around any point different from $p_1,...,p_l$; by compactness of $K$, covering it with finitely many balls where such $C^{2,\alpha}$ estimates hold we deduce that, actually, the $E_i$ eventually satisfy uniform $C^{2,\alpha}$ estimates (depending on $K$) in a whole neighbourhood of $K$. Since the $E_i$ are converging to $\Sigma$, we conclude that we can eventually write the $E_i$, locally, as normal graphs over $K$. This concludes the proof.
\end{proof}

We are now ready to deduce the crucial result that the mass of our limit varifold $V$ is the limit of the fractional perimeters of the $E_i$, which we had not addressed up to now. We more generally prove the stronger result that the variations of the varifold mass of $V$ of any order are the limits of the corresponding variations of the fractional perimeters of the $E_i$. The proof is delicate, since it needs to account for the potentially singular convergence around the $p_1,...,p_l$ as well as the risk (when localising the fractional perimeter) of concentration of mass/classical perimeter close to boundaries of domains but which the fractional perimeter would not be able to capture. 
\begin{proposition}\label{globperconv}
    In the setting of Theorem \ref{limvarcor},
    $$\lim_i (1-s_i)\textnormal{Per}_s(E_i;M)=\gamma_n\|V\|_M\,,
    $$
    with $\gamma_n$ as in \eqref{gammadef}.\\
    More generally, we have the following. Let $X$ be a vector field of class $C^\infty$ on $M$, and let $ \psi^t := \psi_X^t$ denote its flow at time $t$. Then, for every $k\in\N$,
    \begin{equation}\label{q5etgqwa4e5}
        \frac{d^k}{d t^k}(1-s_i)\textnormal{Per}_s(\psi^t(E_i);M)\to \frac{d^k}{d t^k}\|\psi_\#^t(V)\|_M\, \quad \mbox{as } i\to\infty\, .
    \end{equation}
\end{proposition}
    \begin{proof}
    \textbf{Step 1.} We prove that the fractional perimeters of the $E_i$ converge to the varifold mass of $V$.

    Let $\delta>0$, which will be fixed until the end of the proof and then made to go to $0$. Let $\Sigma$ be as in Proposition \ref{LocStableConv}, so that the $E_i$ converge as uniform $C^2$ normal graphs over $K:=\Sigma\setminus\{B_\delta(p_1)\cup...\cup B_\delta(p_l)\}$.  Define the compact manifold with boundary $N:=M\setminus\{B_{4\delta}(p_1)\cup...\cup B_{4\delta}(p_l)\}$; by slightly increasing the value of $\delta$ if necessary, we can moreover assume that $\Sigma$ intersects $\partial N$ transversely and only in its smooth part.
    Let now $\{\mathcal C_j\}$ be a decomposition of $N$ as in Lemma \ref{LipDecomp}, with associated $(1+\delta)$-biLipschitz equivalences $F_j: \mathcal{C}_j \to C_j \subset\R^n$ which are given by the restriction of geodesic normal coordinate maps $\widetilde F_j:B_{r_j}\to \B_{r_j}$ to the $C_j$. By the proof of Lemma \ref{LipDecomp} and the assumption that $\Sigma$ is transverse to $\partial N$, it is simple to see that we can moreover require that $\Sigma$ intersect each of the $\mathcal{C}_j$ transversely and only in the smooth part of $\mathcal{C}_j$.
    
    It is clear then that, since the $E_i$ are converging (with multiplicity) as $C^2$ submanifolds to $\Sigma$ in $N$, for all $i$ large enough the same transversality properties are also true for the $E_i$, and in a quantifiable way. We record in a claim the precise property that we will need, and which follows directly from this discussion. We will use the notation $\mathcal N_r F$ to denote the set (neighborhood) of points at distance at most $r$ from a set $F$.
    
    \noindent \textbf{Claim.} There exists a constant $C$ such that, for all $i$ large enough and $r$ small enough, $\partial E_i\cap\bigcup_j\mathcal N_{3r} \partial\mathcal C_j$ can be covered with $Cr^{2-n}$ (or less) balls $\{B_{r/2}(q_k)\}_k$ of radius $r/2$.
    
    This is to be read as ``since $\Sigma$ is smooth and quantitatively transversal to the $\mathcal C_j$, it accumulates area close to the boundaries of the $\mathcal C_j$ comparably to how a hyperplane passing through the center of a unit ball does so relative to the boundary of said ball. Since the $\partial E_i$ converge in $C^2$ to $\Sigma$ in $K$, the same is true of the $\partial E_i$". In particular, if $\{B_{r/2}(q_k)\}_k$ is a collection of balls of radius $r$ as in the claim and $r$ is small enough, by the classical perimeter bound in \eqref{perindbound} for $i$ large enough we can bound
    $$\textnormal{Per}(E_i;\mathcal N_r \partial\mathcal C_j)\leq \sum_k \textnormal{Per}(E_i;B_{r/2}(q_k))\leq C r^{2-n}r^{n-1}=Cr\, ,
    $$
    which goes to $0$ as $r\to 0$ and thus shows that the $E_i$ do not accumulate perimeter close to the boundary of the $\mathcal C_j$. This simple technical assumption, which is satisfied by our $E_i$, is essential to guarantee that the fractional and classical perimeters of our objects inside the $\mathcal C_j$ are comparable.\\
    
    We are now ready to begin the proof of Step 1. We will omit the constant $\gamma_n$ in what follows, and every instance of a fractional perimeter should be understood as incorporating this multiplicative constant in front.
    
    Observe that, by transversality, we can split
    $$\textnormal{Per}(E_i;M)=\textnormal{Per}(E_i;M\setminus N)+\textnormal{Per}(E_i;N)=\textnormal{Per}(E_i;M\setminus N)+\sum_j\textnormal{Per}(E_i;\mathcal C_j)\, .$$
    An iterated application of the triangle inequality then leads to the bound
    \begin{align}\label{perdsumf}
        &\Big|\textnormal{Per}(E_i;M)-(1-s_i)\textnormal{Per}_{s_i}(E_i;M)\Big|=\\
        &=\Big|\textnormal{Per}(E_i;M\setminus N)+\sum_j\textnormal{Per}(E_i;\mathcal C_j)-(1-s_i)\textnormal{Per}_{s_i}(E_i;M)\Big| \nonumber\\
        &\leq \Big|\textnormal{Per}(E_i;M\setminus N)-(1-s_i)\textnormal{Per}_{s_i}(E_i;M\setminus N)\Big|+\sum_j\Big|\textnormal{Per}(E_i;\mathcal C_j)-(1-s_i)\textnormal{Per}_{s_i}(E_i;\mathcal C_j)\Big| \nonumber\\
        &\ \ \ +\Big|(1-s_i)\textnormal{Per}_{s_i}(E_i;M)-[(1-s_i)\textnormal{Per}_{s_i}(E_i;M\setminus N)+ \sum_j (1-s_i)\textnormal{Per}_{s_i}(E_i;\mathcal C_j)]\Big|\, . \nonumber
    \end{align}
    Since $M\setminus N\subset \cup B_{4\delta}(p_i)$, we immediately see that we can bound the first term as
    \begin{align}\label{eqfpsm1}
        \Big|\textnormal{Per}(E_i;M\setminus N)-(1-s_i)\textnormal{Per}_{s_i}(E_i;M\setminus N)\Big|\leq C(\delta^{n-1}+\delta^{n-s})+\ep(\sg)\leq C\delta^{n-1}+\ep(\sg)\, ,
    \end{align}
    simply by applying \eqref{perindbound} and \eqref{fpergrowth} on each of the balls $B_{4\delta}(p_i)$ and bounding the remaining nonlocal interactions between different balls by $\ep(\sg)$ thanks to disjointness (up to choosing $\delta$ small enough) of the $B_{4\delta}(p_i)$.
    
    The second term is the main one. To obtain a bound for $\Big|\textnormal{Per}(E_i;\mathcal C_j)-(1-s_i)\textnormal{Per}_{s_i}(E_i;\mathcal C_j)\Big|$, we can argue as in Proposition \ref{StableLocPerComp}, using $F_j^{-1}: C_j\subset\R^n\to \mathcal C_j$ as a chart parametrisation and the fact that $F_j$ is a $(1+\delta)$-biLipschitz equivalence. The only difference is that $C_j$ is not a cylinder, so that the $F_j(E_i)$ cannot be exactly written as a union of subgraphs inside $C_j$; on the other hand, the $F_j$ are the restriction of geodesic normal coordinate maps $\widetilde F_j:B_{r_j}\to \B_{r_j}$ to the $C_j$, so that (up to choosing $r_j$ small enough) the $\widetilde F_j(E_i)$ are indeed a union of subgraphs $A_k$ inside $\B_{r_j}$. We can then decompose $F_j(E_i)=\cup_k (A_k\cap C_j)$, where (as argued at the beginning) the $A_k\cap C_j$ are quantitatively transversal the $\partial C_j$. This allows us to apply \cite[Lemma 11]{CV11} to each of the $A_k\cap C_j$, exactly as in the proof of Proposition \ref{StableLocPerComp}. Anyhow, we conclude that
    \begin{align}
        \sum_j\Big|\textnormal{Per}(E_i;\mathcal C_j)-(1-s_i)\textnormal{Per}_{s_i}(E_i;\mathcal C_j)\Big|&\leq \ep(\delta)[\sum_j \textnormal{Per}(E_i;\mathcal C_j)+\sum_j(1-s_i)\textnormal{Per}_s(E_i;\mathcal C_j)] +\ep(\sg)\nonumber\\
        &\leq\ep(\delta) +\ep(\sg)\,,\label{eqfpsm2}
    \end{align}
    where we have bounded the sum of the classical and fractional perimeters inside the $\mathcal C_j$ by the total ones in $M$ (which are uniformly bounded thanks to our assumptions and Proposition \ref{GlobBoundFMI}).\\
    
    We are left with bounding the third term in the RHS of \eqref{perdsumf}, i.e. seeing that $(1-s_i)\textnormal{Per}_{s_i}(E_i;M)$ is very close to the sum of fractional perimeters in $M\setminus N$ and in the $\mathcal C_j$. It is here where we need the quantitative transversality property stated above to get rid of the interactions. Let $I(A,B):=\iint_{E_i\cap A\times E_i^c\cap B} K(p,q)$. Then, we can compute
    \begin{align*}
       \textnormal{Per}_{s_i}(E_i;M) &=\iint_{E_i\cap M\times E_i^c\cap M} K(p,q)=I(M,M)\\
       &=I(M\setminus N,M\setminus N)+I(M\setminus N,N)+\sum_j I(\mathcal C_j,\mathcal C_j)+\sum_j I(C_j,M\setminus \mathcal C_j)\\
       &=\textnormal{Per}_{s_i}(E_i;M\setminus N)+\sum_j\textnormal{Per}_{s_i}(E_i;\mathcal C_j)+I(M\setminus N,N)+\sum_j I(C_j,M\setminus \mathcal C_j)\, .
    \end{align*}
    We can then bound
    \begin{align}
       &\Big|\textnormal{Per}_{s_i}(E_i;M)-[\textnormal{Per}_{s_i}(E_i;M\setminus N) +\sum_j\textnormal{Per}_{s_i}(E_i;\mathcal C_j)]\Big|=\nonumber\\
       &= I(M\setminus N,N)+\sum_j I(C_j,M\setminus \mathcal C_j)\nonumber\\
       &\leq I(M\setminus N,N) + \sum_j\Big[I(\mathcal N_r\partial \mathcal C_j,\mathcal N_r\partial \mathcal C_j)+I(\mathcal C_j\setminus \mathcal N_r\partial \mathcal C_j,M\setminus C_j)+I(\mathcal C_j, (M\setminus \mathcal C_j)\setminus \mathcal N_r\partial \mathcal C_j)\Big]\,. \label{toteqfpf}
    \end{align}
    Since $M\setminus N$ is formed of $l$ balls of radius $\delta$, we can bound
    $$(1-s_i)I(M\setminus N,N)\leq(1-s_i)\textnormal{Per}_{s_i}(M\setminus N;M)\leq lC\delta^{n-s}\,.$$
    Moreover, since $\mathcal C_j\setminus \mathcal N_r\partial \mathcal C_j$ and $M\setminus \mathcal C_j$ are at a positive distance $r$ from each other, we have that
    $$(1-s_i)I(\mathcal C_j\setminus \mathcal N_r\partial \mathcal C_j,M\setminus C_j)\leq(1-s_i)\sum_j\iint_{ (\mathcal C_j\setminus \mathcal N_r\partial \mathcal C_j)\times (M\setminus \mathcal C_j)} K(p,q)\leq (1-s_i)C(r,\delta)=C(r,\delta)\ep(\sg)\, ,
    $$
    and likewise
    $$
    (1-s_i)I(\mathcal C_j, (M\setminus \mathcal C_j)\setminus \mathcal N_r\partial \mathcal C_j)\leq C(r,\delta)\ep(\sg)\, .
    $$
    Therefore, by \eqref{toteqfpf} we see that
    \begin{align}
       &(1-s_i)\Big|\textnormal{Per}_{s_i}(E_i;M)-[\textnormal{Per}_{s_i}(E_i;M\setminus N) +\sum_j\textnormal{Per}_{s_i}(E_i;\mathcal C_j)]\Big|\leq \nonumber\\
       &\hspace{5cm} \leq lC\delta^{n-s}+C(r,\delta)\ep(\sg)+\sum_j I(\mathcal N_r\partial \mathcal C_j,\mathcal N_r\partial \mathcal C_j)\,. \label{toteqfpf2}
    \end{align}
    It remains to bound $I(\mathcal N_r\partial \mathcal C_j,\mathcal N_r\partial \mathcal C_j)$. For that, recall that $\partial E_i\cap \mathcal N_{2r}\partial \mathcal C_j$ is contained in a union $\cup_k B_{r/2}(q_k)$ of at most $Cr^{2-n}$ balls of radius $r/2$, as we recorded in the claim at the beginning of the proof. Considering the balls $B^k:=B_r(q_k)$ of radius $r$ instead, observe that if $p\in E_i\cap \mathcal N_r\partial \mathcal C_j$ and $q\in E_i^c\cap \mathcal N_r\partial \mathcal C_j$, and either  $p\notin \cup_k B^k$ or $q\notin \cup_k B^k$, then either ${\rm dist}(p,\cup_k B_{r/2}(q_k))\geq r/2$ or ${\rm dist}(q,\cup_k B_{r/2}(q_k))\geq r/2$. Therefore, we can bound
    \begin{align*}
    {\rm dist}(p,q)&\geq {\rm dist}(p,\partial E_i)+{\rm dist}(\partial E_i,q)\\
    &= \min\{{\rm dist}(p,\partial E_i\cap \mathcal N_{2r}\partial \mathcal C_j),{\rm dist}(p,\partial E_i\cap (\mathcal N_{2r}\partial \mathcal C_j)^c)\}\\
    &\hspace{1cm} +\min\{{\rm dist}(q,\partial E_i\cap \mathcal N_{2r}\partial \mathcal C_j),{\rm dist}(q,\partial E_i\cap (\mathcal N_{2r}\partial \mathcal C_j)^c)\}\\
    &\geq \min\{{\rm dist}(p,\cup_k B_{r/2}(q_k)),r\}\\
    &\hspace{1cm} +\min\{{\rm dist}(q,\cup B_{r/2}(q_k)),r\}\\
    &\geq r/2\,.
    \end{align*}
    Since $K_s(p,q)$ is bounded for $(p,q)$ with ${\rm dist}(p,q)\geq r/2>0$, for example thanks to \eqref{kerkdtcmp}, this shows that there is a constant $C(r)$ such that
    $$I(\mathcal N_r\partial \mathcal C_j\setminus\cup_k B^k,\mathcal N_r\partial \mathcal C_j)+I(\mathcal N_r\partial \mathcal C_j,\mathcal N_r\partial \mathcal C_j\setminus\cup_k B^k)\leq C(r)\,,$$
    so that we can bound
    \begin{align}
        I(\mathcal N_r\partial \mathcal C_j,\mathcal N_r\partial \mathcal C_j)&\leq I(\mathcal N_r\partial \mathcal C_j\cap\cup_k B^k,\mathcal N_r\partial \mathcal C_j\cap\cup_k B^k)+C(r)\nonumber\\
        &\leq I(\cup_k B^k,\cup_k B^k)+C(r) \label{equbkc}\,.
    \end{align}
    Observe that
    \begin{align*}
        I(\cup_k B^k,\cup_k B^k)&=\iint_{(E_i\cap \cup_k B^k)\times (E_i^c\cap  \cup_k B^k)} K(p,q)\leq\iint_{(E_i\cap \cup_k B^k)\times (E_i\cap \cup_k B^k)^c} K(p,q)=\textnormal{Per}_{s_i}(E_i\cap\cup_k B^k;M)\,.
    \end{align*}
    Then, applying the interpolation result of \eqref{interpropM}, we can bound
    \begin{align*}
        I(\cup_k B^k,\cup_k B^k)&\leq\textnormal{Per}_{s_i}(E_i\cap\cup_k B^k;M)\leq \frac{C}{1-s_i}\big[\textnormal{Per}(\partial(E_i\cap\cup_k B^k);M)\big]^s\\
        &\leq \frac{C}{1-s_i}\big[\textnormal{Per}(\partial E_i\cap\cup_k B^k;M)+\textnormal{Per}(\cup_k \partial B^k;M)\big]^s\\
        &\leq \frac{C}{1-s_i}\big[\sum_k \textnormal{Per}(\partial E_i\cap B^k;M)+\sum_k\textnormal{Per}(\partial B^k;M)\big]^s\,.
    \end{align*}
    The computation after the claim at the beginning of the proof shows then that both sums can be bounded by $Cr$. Substituting into \eqref{equbkc},
    \begin{align*}
        (1-s_i)I(\mathcal N_r\partial \mathcal C_j,\mathcal N_r\partial \mathcal C_j)&\leq Cr^s+(1-s_i)C(r)\,,
    \end{align*}
    
    which by \eqref{toteqfpf2} shows that
    \begin{align*}
       (1-s_i)\Big|\textnormal{Per}_{s_i}(E_i;M)-[\textnormal{Per}_{s_i}(E_i;M\setminus N) +\sum_j\textnormal{Per}_{s_i}(E_i;\mathcal C_j)]\Big|&\leq C\delta^{n-s}+Cr^s+C(r)\ep(\sg)
       \,.
    \end{align*}
    Substituting this bound as well as \eqref{eqfpsm1}, \eqref{eqfpsm2} into \eqref{perdsumf}, we can finally estimate
    \begin{align*}
       \Big|\textnormal{Per}(E_i;M)-(1-s_i)\textnormal{Per}_{s_i}(E_i;M)\Big|&\leq C\delta^{n-1}+Cr^s+C(r)\ep(\sg)+\ep(\delta)\,.
    \end{align*}
    Fix $\delta>0$. Choosing $r$ small enough so that $Cr^s\leq \delta$, and then choosing $i$ large enough (thus $\sg$ small enough) so that $C(r)\ep(\sg)\leq \delta$, we deduce that
   \begin{align*}
       \Big|\textnormal{Per}(E_i;M)-(1-s_i)\textnormal{Per}_{s_i}(E_i;M)\Big|&\leq C\delta
    \end{align*}
    for all $i$ large enough, or in other words
     \begin{align*}
       \limsup_{i\to\infty} \Big|\textnormal{Per}(E_i;M)-(1-s_i)\textnormal{Per}_{s_i}(E_i;M)\Big|&\leq C\delta
       \,.
    \end{align*}
    By arbitrariness of $\delta$ we find that $\lim_{i} (1-s_i)\textnormal{Per}_{s_i}(E_i;M)=\lim_{i} \textnormal{Per}(E_i;M)$,
    which since the $\partial E_i$ converge in the varifold sense to $V$ shows that $\lim_{i} (1-s_i)\textnormal{Per}_{s_i}(E_i;M)=\|V\|_M$
    as desired.
    
    \noindent \textbf{Step 2.} We now show \eqref{q5etgqwa4e5}, which states the convergence of the derivatives of the fractional perimeters along the flow of a vector field to the derivatives of the mass of $V$.
    
    The argument is as in the proof of \eqref{convItep} simpler version of the proof of \eqref{q5etgqwa4e50}. Let $X$ be a vector field on $M$, and let $\psi^t:=\psi_X^t$ denote its flow at time $t$. Define $f_i(t)=(1-s_i)\textnormal{Per}_{s_i}(\psi^t(E_i);M)$, and likewise $f(t)=\|\psi_\#^t(V)\|_M$. In step $1$ we have precisely shown that $f_i(0)\to f(0)$ as $i\to\infty$. It is very simple to see that, given any other $t$, we have that $f_i(t)\to f(t)$ as well: Indeed, we have only used the regular convergence of the $E_i$ to $\Sigma$ outside of a small set and the estimates such as \eqref{fpergrowth} for the $E_i$ on that small set; considering the translated sets $\psi^t(E_i)$ for a fixed $t$ preserves all these properties, since $\psi^t$ is just a fixed diffeomorphism of $M$, and in particular the small set of possibly bad convergence of the $\psi^t(E_i)$ towards $\psi^t(\Sigma)$ is just the flow at time $t$ of the corresponding set at time $0$.\\
    Therefore, we see that the single-variable functions $f_i(t)$ converge pointwise to $f(t)$. By Lemma \ref{enboundslemma2} and the standing assumption that the $E_i$ have uniformly bounded fractional perimeters (upon multiplication by $(1-s_i)$), we also know that the functions $\frac{d^k}{dt^k}f_i(t)$ are locally uniformly bounded. By Arzel\`a-Ascoli, we deduce that $\frac{d^k}{dt^k}f_i(t)$ converges locally uniformly and moreover the limit needs to be $\lim_i \frac{d^k}{dt^k}f_i(t)=\frac{d^k}{dt^k}f(t)$. This concludes the proof of \eqref{q5etgqwa4e5}.
\end{proof}
We can now finally prove the regularity of the limit surface $\Sigma$.
\begin{theorem}\label{smoothlim}
Let $V$ be the limit integral stationary varifold obtained in Theorem \ref{limvarcor}. Then, $\Sigma={\rm supp}(V)$ is a smooth surface.
\end{theorem}

\begin{proof}
Let $p_1,...,p_l$ and $\Sigma$ be as in Proposition \ref{LocStableConv}.

\noindent \textbf{Step 1.} Stability on a punctured ball: Given any $p\in M$, we show that there exists $r_p>0$ such that $\Sigma$ is stable on $B_{r_p}\setminus \{p\}$.

Assume first that $p\neq p_1, p_2,...,p_l$. By Lemma \ref{aspoints}, there exist $r_p>0$ and $\Lambda_p>0$ such that the $E_{i}$ are eventually $\Lambda$-almost stable on $B_{r_p}(p)$. This implies that $\Sigma$ is stable in $B_{r_p}(p)$: Let $X$ be a smooth vector field supported in $B_{r_p}(p)$, and let and let $ \psi^t := \psi_X^t$ denote its flow at time $t$. Then, by \eqref{q5etgqwa4e5} and the definition of $\Lambda$-almost stability (see Definition \ref{almoststab}, and recall the last line in Definition \ref{2quaddef}) we have that
\begin{align*}
        \frac{d^2}{d t^2}\|\psi_\#^t(V)\|_M&=\lim_i \frac{d^2}{d t^2}(1-s_i)\textnormal{Per}_s(\psi^t(E_i);M)\geq -\lim_i (1-s_i)\Lambda_p  \Big(\int_{\partial E_i\cap B_{r_p}(p)}|X\cdot \nu_{\Sigma}|\Big)^2\\
        &\geq -\lim_i (1-s_i)\Lambda_p  [\|X\|_{L^{\infty}(M)}\textnormal{Per}(\partial E_i;M)]^2= 0\, ,
\end{align*}
as desired.

Let now $p$ be one of the $p_1,...,p_l$ instead. If (arguing by contradiction) we assume that there exists no $r_p>0$ such that $\Sigma$ is stable on $B_{r_p}(p)\setminus \{p\}$, given $\delta>0$ we can in particular find $(m+1)$ vector fields $X_1$,..., $X_{m+1}$, supported respectively on small disjoint annuli $A_1,...,A_{m+1}$ centered at $p$ and contained in $B_\delta(p)$, such that $\frac{d^2}{dt^2}\Big|_{t=0}\|(\psi_{X_k}^t)_\# V\|_M<0$ for every $k$.

On the other hand, for each $i$  we can find (by Lemma \ref{asineq}) one of the $A_1,...,A_{m+1}$ such that $E_i$ is $\Lambda$-almost stable in it, for some $\Lambda<\infty$ depending only on the fixed $A_1,...,A_{m+1}$. Thus, by passing to a subsequence (which we do not relable), we can assume that the $E_i$ are all $\Lambda$-almost stable on the same annulus, say $A_1$. From this we deduce (see \eqref{lamasineq}) that $\frac{d^2}{dt^2}\big|_{t=0}\text{Per}_s(\psi_{X_1}^t(E_i))\geq -C$ for some constant $C<\infty$ depending on $\Lambda$, $X_1$ and $\sup_i\textnormal{Per}(\partial E_i;M)$, the supremum of classical perimeters being bounded thanks to Proposition \ref{GlobBoundFMI}. By Proposition \ref{globperconv}, we can then pass to the limit as $i\to\infty$ to deduce that
\begin{align*}
    \frac{d^2}{dt^2}\Big|_{t=0}\|(\psi_{X_1}^t)_\# V\|_M&=\lim_i (1-s_i)\frac{d^2}{dt^2}\big|_{t=0}\text{Per}_s(\psi_{X_1}^t(E_i))\geq -\lim_i (1-s_i)C=0\, ,
\end{align*}
which is a contradiction.

This concludes the proof of Step $1$.

\noindent \textbf{Step 2.} We prove that $\Sigma={\rm supp}(V)$ is smooth everywhere.

In Proposition \ref{LocStableConv} we have proved that the support $\Sigma$ of $V$ is a smooth (minimal) surface in all of $M\setminus \{p_1,...,p_l\}$. In particular, we have that (with $n=3$)
$$\mathcal H^{n-3}(\textnormal{Sing}(\Sigma)) =\mathcal H^{0}(\textnormal{Sing}(\Sigma)) \leq l < \infty \, .
$$
We will now show that if $p$ is one of the $p_1,...,p_l$, then $\Sigma$ is also smooth around $p$. First, by Step 1, we deduce that $V$ is a stable integral varifold on a punctured ball $B_{r_p}\setminus\{p\}$. Moreover, we have just argued that $\mathcal H^{n-3}(\textnormal{Sing}(\Sigma)) < \infty$ with $n=3$. The classical result\footnote{Recall that $n=3$ refers to ambient dimension in the present article. This corresponds to $n=2$ in the notation employed in \cite{SS}.} in \cite{SS} on the regularity of stable minimal surfaces with small singular set (points, in our case, which have null $2$-capacity on a surface) can then be used to show that $\Sigma$ is actually smooth also at the point $p$. To be precise, for $n>3$ this would be precisely the result in \cite{SS}, and in the present case $n=3$ the removability of point singularities was first given in \cite{Isol} by using a tangent cone analysis on top of the results in \cite{SS}. See also the appendix of \cite{Sharp} for how a simple capacitary argument allows to apply \cite{SS} directly instead.
\end{proof}
We can now complete the proof of Theorem \ref{FMIConv}. We have essentially already obtained most of its points.
\begin{proof}[Proof of Theorem \ref{FMIConv}]
    \textbf{Step 1.} Convergence to a limit.
    
    By Theorem \ref{limvarcor}, we have the convergence of a subsequence of the $\partial E_i$ to a limit stationary integral varifold $V$ in the varifold sense. Moreover, by Theorem \ref{smoothlim}, this varifold is supported on a smooth surface $\Sigma$. Let $\Sigma_1,...,\Sigma_m$ denote the connected components of $\Sigma$; then, by the Constancy Theorem (see \cite[Theorem 41.1]{SimonNotes}), the $\Sigma_k$ are minimal surfaces and there are integers $n_k$ such that $V=\sum_{k=1}^{m} n_k [\Sigma_k]$, where $[\Sigma_k]$ denotes the integer rectifiable varifold associated to the smooth surface $\Sigma_k$. Combining this with Proposition \ref{globperconv}, we deduce that
    $$\lim_i (1-s_i)\textnormal{Per}_s(E_i;M)=\lim_i\textnormal{Per}(E_i;M)=\sum_{k=1}^{m} n_k\rm{Per}(\Sigma_k)\,.$$
    \noindent \textbf{Step 2.} Constancy of the multiplicity in the convergence.
    
    Let $p_1,...,p_l$ be as in Proposition \ref{LocStableConv}. We want to show that a subsequence of the $E_i$ converges towards $\Sigma\setminus\{p_1,...,p_l\}$ with constant multiplicity over each component. Let then $K:=K_\delta=\Sigma_k\setminus\{B_\delta(p_1)\cup...\cup B_\delta(p_l)\}$. If $\delta$ is small enough, then $K_\delta$ is still a connected set\footnote{Thanks to $\Sigma_k$ being smooth and connected, so that locally around the $p_j$, topologically we are just removing a disk from a bigger disk.}. Moreover, by Proposition \ref{LocStableConv}, the $E_i$ are eventually graphs over $K_\delta$. Fix $i$; then, since (given a natural number $d\in\N$) having a fixed number of leaves $d$ is both an open and closed condition, by connectedness the number of leaves of $E_i$ over the limit is constant on $K_\delta$. Now, denoting this multiplicity by $d=d_i$, a priori it could vary when varying the index $i$; on the other hand, $d_i$ is uniformly bounded, since $\limsup_i d_i\textnormal{Per}(\Sigma\cap K_\delta)\leq \limsup_i\textnormal{Per}(E_i)$ and the latter quantity is uniformly bounded by Proposition \ref{GlobBoundFMI}. Therefore, passing to a subsequence we can take $d_i$ to be identical for all $i$. Finally, observe that choosing any $\delta'<\delta$ does not change the result, since $K_\delta\subset K_{\delta'}$ and as argued above the number of leaves over $K_\delta'$ needs to be constant, therefore (for $i$ large enough) this number coincides with the one over $K_\delta$.

    \noindent \textbf{Step 3.} The minimal surface $\Sigma$ has Morse index bounded by $\p$.

    To check this, consider $(\p+1)$ vector fields $X_0,...,X_\p$ of class $C^\infty$ on $M$.
    
Letting $a:=(a_0,a_1,...,a_\p)\in\R^{\p+1}$ and $X[a]=a_0X_0+...+a_\p X_\p$, we can define the quadratic form $Q_{i}(a):=(1-s_i)\frac{d^2}{dt^2}\Big|_{t=0}\textnormal{Per}_s(\psi^{t}_{X[a]}(E_i))$, which we can write as $Q_{i}(a)=Q_{i}^{kl}a_ka_l$ for some coefficients $Q_{i}^{kl}$. From \eqref{q5etgqwa4e5} and the polarization identity for a quadratic form, it is immediate to see that $Q_{\ep_j}^{kl}\to Q_{0}^{kl}$ as $j\to \infty$, where $Q_0(a):= \frac{d^2}{dt^2}\Big|_{t=0}\|(\psi_{X_{[a]}}^t)_\# V\|_M=Q_{0}^{kl}a_ka_l$.

Now, since the $E_i$ have Morse index $\leq \p$, by definition we know that for every $i$ there must exist some $a^{i}\in\Sp^{\p}$ such that
    \begin{equation}
    Q_{i}(a^{i})=(1-s_i)\frac{d^2}{dt^2}\Big|_{t=0}\textnormal{Per}_s(\psi^{t}_{X[a^i]}(E_i))\geq 0\, ;
    \end{equation}
the convergence of the coefficients $Q_{i}^{kl}$ to $Q_0^{kl}$ then immediately shows that $Q_0(a)\geq 0$ for some $a\in\Sp^\p$ as well, which is precisely a restatement of the fact that $\Sigma$ have Morse index at most $\p$.
\end{proof}
\section{Density and equidistribution of classical minimal surfaces -- Proofs of Corollary \ref{FracYau4} and Theorems \ref{Density}-\ref{Equi}}
With the results up to now, we can give:
\begin{proof}[Proof of Corollary \ref{FracYau4}]
    Given $\p\in\N^+$, let $\{s_i\}_i$ be a sequence of numbers in $(0,1)$ such that $\lim_i s_i=1$ and
    $$\lim_i (1-s_i)l_{s_i}(\p,M)=\liminf_{s\to 1} (1-s)l_{s}(\p,M)=l_{1}(\p,M)\, ,
    $$
    which exists by definition \eqref{deflp1} of $l_{1}(\p,M)$. Let moreover $E_{\p}^{s_i}$ be given by Theorem \ref{FracYau3}. Applying Theorem \ref{FMIConv} to the $\partial E_{\p}^{s_i}$, which are smooth $s_i$-minimal surfaces with Morse index bounded by $\p$ and which satisfy the fractional perimeter bound
    $$(1-s_i){\rm Per}_{s_i}(\partial E_{\p}^{s_i})=(1-s_i)l_{s_i}(\p,M)\leq C\p^{s_i/n}\,,$$
    we conclude the desired result.
\end{proof}

In what follows, it will be useful to keep in mind the notation and properties in \eqref{minmaxval}--\eqref{lpsbounds} and \eqref{deflp11}--\eqref{lp1bounds}. To lighten the notation, we will write $l_{s,\p}(N,h)$ instead of $l_s(\p,(N,h))$ when metrics appear.

We start with the following simple lemma, an analogue of \cite[Lemma 1]{MNS} and which shows that the $\p$-widths are continuous (more precisely, locally Lipschitz) with respect to the metric:
\begin{lemma}\label{lipcompwidth}
    Let $\widetilde g$ be a $C^{\infty}$ Riemannian metric on $M$, and let $C_1<C_2$ be positive constants. Then, there is $K=K(\widetilde g, C_1, C_2)$ such that
    $$|\p^{-1/n}l_{1,\p}(M,g')-\p^{-1/n}l_{1,\p}(M,g)|\leq K |g'-g|_{\widetilde g}\,,
    $$
    for any $g,g'\in\{h \mbox{ smooth metric on } M : C_1\widetilde g \leq h \leq C_2\widetilde g$ and any $p\in \N$.
\end{lemma}
\begin{proof}
    By Step 2 in the proof of Theorem \ref{clasweylthm}, we know that $l_{1,\p}(M,g)=l_{1,\p}^d(M,g)$. Therefore, the lemma is equivalent to proving that
    $$|\p^{-1/n}l_{1,\p}^d(M,g')-\p^{-1/n}l_{1,\p}^d(M,g)|\leq K |g'-g|_{\widetilde g}\,.
    $$
    To prove that this, we start by looking at the fractional Allen-Cahn functional for some $s\in(0,1)$. We consider Lemma \ref{Lipineq}, with $(M_1,g_1)=(M,g')$ and $(M_2,g_2)=(M,g)$, $\Omega_1=M_1$ and $\Omega_2=M_2$, and $F=id_M$ viewed as a map $F:M_1\to M_2$. Then, if $u\in H^{s/2,d}(M_1)=H^{s/2,d}(M_2)$, it holds that
$$
\mathcal E_{\ep,s}^d(u\circ F,M_1)\leq \text{Lip}(F)^{3n+s}\mathcal E_{\ep,s}^d(u,M_2)\, ,
$$
and then, by equality of the spaces $H^{s/2,d}(M_1)$ and $H^{s/2,d}(M_2)$, taking $p$-th min-max values according to \eqref{minmaxvald} we get that
$$
c_{\ep,s}^d(\p,M_1)\leq \text{Lip}(F)^{3n+s}c_{\ep,s}^d(\p,M_2)\, .
$$
Sending $\ep$ to $0$, we deduce that
$$
l_{s,\p}^d(M_1)\leq \text{Lip}(F)^{3n+s}l_{s,\p}^d(M_2)\, .
$$

Now, we can bound
\begin{align*}
\rm{Lip}(F)^{3n+s}&=\Bigg(\sup_{(p,q)\in M\times M}\frac{\rm{dist}_{M_2}(p,q)}{\rm{dist}_{M_1}(p,q)}\Bigg)^{3n+s}\leq \Bigg(\sup_{(p,q)\in M\times M}\frac{\sup_{v\in TM}\frac{|v|_{g_2}}{|v|_{g_1}}\rm{dist}_{M_1}(p,q)}{\rm{dist}_{M_1}(p,q)}\Bigg)^{3n+s}\\
&=\Bigg(\sup_{v\in TM}\frac{|v|_{g_2}}{|v|_{g_1}}\Bigg)^{3n+s}=\Bigg(1+\sup_{v\in TM}\frac{|v|_{g_2}-|v|_{g_1}}{|v|_{g_1}}\Bigg)^{3n+s}\\
&\leq \Bigg(1+\sup_{v\in TM}\frac{|v|_{\widetilde g}}{|v|_{g_1}}\sup_{v\in TM}\frac{|v|_{g_2}-|v|_{g_1}}{|v|_{\widetilde g}}\Bigg)^{3n+s}\leq \Bigg(1+\frac{1}{C_1}|g_2-g_1|_{\widetilde g}\Bigg)^{3n+s}\\
&=\Bigg(1+\frac{1}{C_1}|g'-g|_{\widetilde g}\Bigg)^{3n+s}\, .
\end{align*}

Substituting, and using again the definition of $M_1$ and $M_2$, we get that
\begin{equation}\label{inlsmet}
    l_{s,\p}^d(M,g')\leq \Bigg(1+\frac{1}{C_1}|g'-g|_{\widetilde g}\Bigg)^{3n+s}l_{s,\p}^d(M,g)\, ,
\end{equation}
or in other words
$$
l_{s,\p}^d(M,g')-l_{s,\p}^d(M,g) \leq\Bigg[\Bigg(1+\frac{1}{C_1}|g'-g|_{\widetilde g}\Bigg)^{3n+s}-1\Bigg]l_{s,\p}^d(M,g)\, .
$$
Now, applying first \eqref{inlsmet} with $(M,g)$ and $(M,\widetilde g)$ instead on the right hand side, and using then the upper bound in \eqref{lpsboundsd} with $\Omega$ equal to the Riemannian manifold $(M,\widetilde g)$, we get that
$$
(1-s)l_{s,\p}^d(M,g')\p^{-s/n}-(1-s)l_{s,\p}^d(M,g)\p^{-s/n} \leq C(\widetilde g,C_1,C_2)\Bigg[\Bigg(1+\frac{1}{C_1}|g'-g|_{\widetilde g}\Bigg)^{3n+s}-1\Bigg]\, ,
$$
or equivalently
$$
(1-s)l_{s,\p}^d(M,g')\p^{-s/n} \leq (1-s)l_{s,\p}^d(M,g)\p^{-s/n} + C(\widetilde g,C_1,C_2)\Bigg[\Bigg(1+\frac{1}{C_1}|g'-g|_{\widetilde g}\Bigg)^{3n+s}-1\Bigg]\,.
$$
Taking finally $\liminf_s$ on both sides and substracting again, by definition of $l_{1,\p}^d$ we get that
$$
l_{1,\p}^d(M,g')\p^{-1/n}-l_{1,\p}^d(M,g)\p^{-1/n} \leq C(\widetilde g,C_1,C_2)\Bigg[\Bigg(1+\frac{1}{C_1}|g'-g|_{\widetilde g}\Bigg)^{3n+s}-1\Bigg]\, ,
$$
from which the desired result follows by a Taylor expansion.
\end{proof}
The short proof of Theorem \ref{Density} on the density for generic metrics then follows as in \cite{IMN} (see also \cite{GG2}). We indicate the few modifications needed.
\begin{proof}[Proof of Theorem \ref{Density}]
The proof is just as in \cite[Proposition 3.1]{IMN}, which uses the minimal surfaces associated to the Almgren-Pitts spectrum and the Weyl Law associated to them. One simply need substitute the continuity of the $\p$-widths in \cite[Lemma 2.1]{IMN} by Lemma \ref{lipcompwidth}, the attainability of the min-max widths in \cite[Proposition 2.2]{IMN} by Corollary \ref{FracYau4}, and the Weyl Law from \cite{LMN} by Theorem \ref{clasweylthm}.
\end{proof}
Likewise, one can obtain Theorem \ref{Equi} on the equidistribution of classical minimal surfaces for generic metrics:
\begin{proof}[Proof of Theorem \ref{Equi}]
    The proof in \cite{MNS}, which again uses the minimal surfaces and Weyl Law associated to the Almgren-Pitts spectrum, can be replicated in our case. The result in \cite[Lemma 1]{MNS} on the Lipschitz continuity of the $\p$-widths is replaced by Lemma \ref{lipcompwidth} from the present paper, the attainability of the min-max widths in \cite[Proposition 2.2]{IMN} (which is used in the proof of \cite[Lemma 2]{MNS}) is replaced by Corollary \ref{FracYau4}, and the Weyl Law from \cite{LMN} is replaced by Theorem \ref{clasweylthm}.
\end{proof}

\appendix

\section{Proof of Theorem \ref{mmbounds}}\label{app:minmaxthm}
\begin{proof}[Proof of Theorem \ref{mmbounds}]
    The existence of a critical point $u_{\ep,\p}$ with energy $c_{\ep,s}$ follows from applying a min-max theorem to the functional $\mathcal{E}_{\ep,s}$ with $\widetilde{\mathcal{F}}_\p$ as the set of min-max families, exactly as in the proof of \cite[Theorem 3.4]{CFS}. The fact that $u_{\ep,\p}$ has index bounded by $\p$ comes from the fact that min-max critical points of families of the type \eqref{cohom} satisfy this upper index bound; this is thanks to the Hausdorff measure condition $\mathcal{H}^\p(A)<\infty$, as described in \cite{LS}, which intuitively corresponds to considering sets $A$ with at most $\p$ parameters in the min-max.
    
    Finally, the upper and lower energy bounds in \eqref{critlevbound} follow the same strategy as in the proof of \cite[Theorem 3.2]{CFS}. To be precise, in \cite{CFS} we considered the family
    \begin{equation*}
    \mathcal{G}_\p:= \big{\{} A = f(\Sp^\p) \ : \ f \in C^0 (\Sp^p;  H^{s/2}(M)\backslash \{0\} ) \s \text{and} \s f(-x)=-f(x) \  \forall \ x\in\Sp^\p \big{\}}
    \end{equation*}
    instead.
    
    By the Borsuk--Ulam theorem, given $A\in \mathcal G_\p$ (and assuming that the Hausdorff condition holds) then
    $A\in \mathcal F_\p$ as well. This inclusion was, in fact, the only property that was used to prove the lower bounds, thus they hold for $\mathcal F_\p$ as well.\\
    On the other hand, the upper bounds were shown by constructing an explicit $f:\Sp^\p\to H^{s/2}(M)\setminus{0}$ such that all elements in $A=f(\Sp^\p)\in\mathcal G_\p$ have low energy. If the Hausdorff measure condition $\mathcal{H}^\p(A)<\infty$ holds, then $A\in \mathcal F_\p$ and thus the same set $A$ works in our case. Unfortunately, the only reasonable way of showing the measure condition is to see that $f:\Sp^\p\to H^{s/2}(M)\setminus{0}$ is Lipschitz, which is not true for the function $f$ considered in \cite[Theorem 3.2]{CFS} -- in fact, it is only $\beta$-Holder continuous for an intermediate exponent $\beta\in (0,1)$ in general. We now explain the construction and how to adapt it to our case to obtain a Lipschitz map.
    
    Essentially, via the exponential map and a finite covering of our manifold with $\p$ small balls, in \cite{CFS} the construction is reduced to the case $M=\B_{r_\p}(0)$, the Euclidean ball with radius $r_\p=C\p^{-\frac{1}{n}}$. Then, $f$ is defined as the map which sends $a=(a_0, a_1, \dotsc, a_\p) \in \Sp^\p$ to the function $u_a\in H^{s/2}(M)\setminus \{0\}$ defined by $u_a(x)={\rm sgn}(P_a(x_n))$ with $P_a(z)=a_0+a_1 z +\dotsc +a_p z^\p$. Therefore, $u_a$ is a characteristic function with values $\pm 1$ depending on the sign of $P_a(x_n)$, and it jumps at most $\p$ times when $x_n$ coincides with a root of $P_a$ with odd multiplicity. This shows that $u_a$ has BV norm at most $C\p\cdot r_\p^{n-1}=C\p^{1/n}$, which by the interpolation result in \eqref{ipr11} shows that it has $H^{s/2}$ norm at most $\frac{C}{1-s}\p^{s/n}$ as desired.
    
    The only reason why this assignation $f:\Sp^\p\to H^{s/2}(M)\setminus\{0\}$ might not be Lipschitz is that, when changing $a$, two roots of $P_a$ might approach each other (or approach $\pm r_\p$). Just as an example, consider $t\in N:=(0,1)$ and $P_t(z)=z(z-t)$, which has the roots $0$ and $t$. Define the single-variable function $u_t(z)={\rm sgn}(P_t(z))$ for $z\in(-1,1)$, so that $u_t$ takes the value $-1$ on $(-1,0)$, $+1$ on $(0,t)$, and $-1$ again on $(t,1)$. As $t\to 0$, it is clear that $[u_t]_{H^{s/2}(-1,1)}\to 0$: for example, by applying \eqref{ipr11} to $u_t+1$, we see that $[u_t]_{H^{s/2}(-1,1)}^2\leq Ct^{1-s}$, or equivalently $[u_t]_{H^{s/2}(-1,1)}\leq Ct^{\frac{1-s}{2}}$. On the other hand, as a function of $t$, $t\mapsto Ct^{\frac{1-s}{2}}$ is not Lipschitz, and in general this bound cannot be improved.\\
    
    Going back to the case $M=\B_{r_\p}$, to solve this issue we regularise our functions in a simple way. Consider a smooth, nondecreasing, odd function $\eta$ such that $\eta(x)={\rm sgn}(x)$ for $|x|\geq 1$ and $\eta(x)=0$ only at $x=0$. Define $\eta_\delta(x)=\eta(x/\delta)$, which as $\delta\to 0$ converges to ${\rm sgn}(x)$, and set $u_{a,\delta}(x)=\eta_\delta\circ P_a(x_n)$, which is a smoothed out version of the function $u_a(x)$. Therefore, the $BV$ norm of $u_{a,\delta}$ is bounded in the same way as for $u_a$, since the function $u_{a,\delta}(x)$ varies in total just as much (it is monotone where $P_a$ is, and there are at most $(\p-1)$ values of $x_n$ where $P_a(x_n)$ changes monotonicity since it is a polynomial of degree at most $\p$). The interpolation result \eqref{ipr11} then shows once again that $[u_{a,\delta}]_{H^{s/2}(\B_{r_\p})}\leq \frac{C}{1-s}\p^{s/n}$. Moreover, since $W(\pm 1)=0$ and $W(s)\leq\frac{1}{4}$ for $s\in[-1,1]$, we see that
    $$\lim_{\delta\to 0}\sup_{a\in\Sp^\p}\frac{1}{\ep^s}\int_{\B_{r_\p}} W(u_{a,\delta})\leq \frac{1}{4\ep^s}\lim_{\delta\to 0}\sup_{a\in\Sp^\p} \big|\{x:|P_a(x)|\leq \delta\}\cap\B_{r_\p}\big|=0\,,$$
    where the last equality follows from the fact that the set of $x\in\B_{r_\p}$ where the polynomial $P_a(x_n)$ ---whose vector of coefficients $a$ has norm one--- is close to zero has very small measure. This elementary fact is easy to see, for example by a compactness+contradiction argument.
    
    Then, by choosing (fixing) $\delta>0$ small enough depending on $\p$ and $\ep$, we find that $\mathcal E_{\ep,s}(u_{a,\delta},\B_{r_\p})\leq \frac{C}{1-s}\p^{s/n}$. Moreover, by the definition of $\eta$ we see that $u_{a,\delta}$ is never the zero function, since neither is $P_a$, and moreover $u_{-a,\delta}=-u_{a,\delta}$ by oddness of $\eta$. Let $f:\Sp^\p\to H^{s/2}(M)\setminus\{0\}$, $f(a)= u_{a,\delta}$. To see that $A=f(\Sp^\p)$ is in $\mathcal F_\p$, it remains only to see that $f$ is a Lipschitz function, since it will imply that $\mathcal H^\p(A)<\infty$; with this, we will conclude the proof of the desired upper bounds in \eqref{critlevbound}.
    
    Now, letting $g(a,x):=u_{a,\delta}(x)$, by smoothness in $(a,x)$ we can bound $\|g\|_{C^2}\leq C(\delta)$ for some constant $C(\delta)$, the value of which is not important (as we have already fixed $\delta$). This shows that, given $a$ and $b=a+h$ in $\Sp^\p$,
    \begin{align*}
        [u_{a+h,\delta}-u_{a,\delta}]_{H^{s/2}(M)}^2&=\iint_{\B_{r_\p}\times\B_{r_\p}} \Big((u_{a+h,\delta}-u_{a,\delta})(x)-(u_{a+h,\delta}-u_{a,\delta})(y)\Big)^2K(x,y)\\
        &=\iint_{\B_{r_\p}\times\B_{r_\p}} \Big(\int_0^{|h|}\frac{\partial}{\partial s} \big[u_{a+s\frac{h}{|h|},\delta}(x)-u_{a+s\frac{h}{|h|},\delta}(y)\big]ds\Big)^2K(x,y)\\
        &\leq |h|^2\iint_{\B_{r_\p}\times\B_{r_\p}} \sup_s\big[\frac{\partial}{\partial s} u_{a+s\frac{h}{|h|},\delta}(x)-\frac{\partial}{\partial s}u_{a+s\frac{h}{|h|},\delta}(y)\big]^2K(x,y)\\
        &\leq |h|^2\|g\|_{C^2}\iint_{\B_{r_\p}\times\B_{r_\p}} |x-y|^2K(x,y)\\
        &\leq |h|^2 C(\delta)\,,
    \end{align*}
    which after taking square roots on both sides proves that $a\mapsto u_{a,\delta}$ is a Lipschitz map from $\Sp^\p$ to $H^{s/2}(\B_{r_\p})$. This concludes the proof in the case $M=\B_{r_\p}$. As mentioned at the beginning, in general one actually decomposes $M$ as a union of $\p$ subsets which (passing to normal coordinates) are of this type; see the proof of \cite[Theorem 3.2]{CFS}. The additional interactions between the respective subsets can simply be bounded directly by the fractional perimeters of the subsets, including in the proof of the Lipschitz property above in the general case.
\end{proof}
\section{Step 4 of the proof of Proposition \ref{supleq2}}\label{app:hausdorffmod}
In Step 1, given $A\in \tilde{\mathcal{F}}_\p^d(\Omega)$, we constructed $A'\subset H^{s/2}(M)\setminus\{0\}$ with energy not surpassing that of $A$ by much (in the sense that every $U\in A'$ is obtained from some $u\in A$ by $U\big|_{C_i}=u\circ F_i$, and the bound \eqref{Uucin} holds). On the other hand, we explained that $f:H^{s/2,d}(\Omega)\to H^{s/2,d}(M)$ defined by $f(u):=U$ may not necessarily be a Lipschitz map; in case it were, it would imply that $A'\in\tilde{\mathcal{F}}_\p^d(M)$ and we would conclude the proof.

Fix $A\in \tilde{\mathcal{F}}_\p^d(\Omega)$. We want to regularise the map $f$ (restricted to $A$) by a very close equivariant map $f_{\delta}:A\to H^{s/2}(M)\setminus\{0\}$ which is moreover Lipschitz, which will allows us to conclude the result. Let $\xi:[0,+\infty)\to[0,1]$ be a smooth, increasing function such that $\xi(0)=0$ and $\xi(x)=1$ for $x\geq 1$. Let $\delta>0$, which we will choose later \textit{depending on $A$}, and set $$\xi_\delta(t)=\xi(t/\delta)$$ and $\eta_\delta(p)=\xi_\delta\circ {\rm dist}(p,\cup_i \partial C_i))$. The assignation $f_{\delta}(u)=U_\delta:=\eta_\delta U$ is still equivariant, and it is not zero if $u$ isn't either. Define $A_\delta'=f_{\delta}(A)$.

We now compute:
\begin{align*}
    [U_\delta]_{H^{s/2}(M)}&=\sum_{i,j} \iint_{C_i\times C_j}(U_\delta(p)-U_\delta(q))^2{\rm dist}_M(p,q)^{-(n+s)}\,dp\,dq\\
    &= \sum_i \iint_{C_i\times C_i} ((\eta_\delta U)(p)-(\eta_\delta U)(q))^2{\rm dist}_M(p,q)^{-(n+s)}\,dp\,dq \\
    &\qquad+ \sum_{i\neq j} \iint_{C_i\times C_j}((\eta_\delta U)(p)-(\eta_\delta U)(q))^2{\rm dist}_M(p,q)^{-(n+s)}\,dp\,dq\\
    &\leq 2\sum_i \iint_{C_i\times C_i} \eta_\delta^2(p) (U(p)- U(q))^2{\rm dist}_M(p,q)^{-(n+s)}\,dp\,dq\\
    &\hspace{0.5cm} +2\sum_i\iint_{C_i\times C_i} U(p)^2(\eta_\delta(p) -\eta_\delta(q))^2{\rm dist}_M(p,q)^{-(n+s)}\,dp\,dq\\
    &\hspace{0.5cm} +4\sum_{i} \iint_{C_i\times(M\setminus C_i)}\eta_\delta^2(p) U^2(p){\rm dist}_M(p,q)^{-(n+s)}\,dp\,dq\, .
\end{align*}
Using that $\eta_\delta\leq 1$ and $|\eta_\delta(p)-\eta_\delta(q)|\leq C(\delta){\rm dist}_M(p,q)$, together with \eqref{kerkdtcmp} we can bound
\begin{align*}
    &2\sum_i \iint_{C_i\times C_i} \eta_\delta^2(p) (U(p)- U(q))^2{\rm dist}_M(p,q)^{-(n+s)}\,dp\,dq\\
    &\qquad+2\sum_i\iint_{C_i\times C_i} U(p)^2(\eta_\delta(p) -\eta_\delta(q))^2{\rm dist}_M(p,q)^{-(n+s)}\,dp\,dq\leq\\
    &\leq 2\sum_i \iint_{C_i\times C_i} (U(p)- U(q))^2{\rm dist}_M(p,q)^{-(n+s)}\,dp\,dq+C(\delta)\sum_i\iint_{C_i\times C_i} U(p)^2{\rm dist}_M(p,q)^{-(n+s-2)}\,dp\,dq\\
    &\leq 2\sum_i \iint_{C_i\times C_i} (U(p)- U(q))^2{\rm dist}_M(p,q)^{-(n+s)}\,dp\,dq+C(\delta)\sum_i\int_{C_i}U(p)^2\,dp\\
    &\leq C(\delta)\sum_i\|U\|_{H^{s/2}(C_i)}^2\\
    &\leq C(\delta)\sum_i\|u\|_{H^{s/2}(\mathcal C_i)}^2.
\end{align*}
We now bound the outer interaction term, using \eqref{kerkdtcmp} and the fact that $\eta_\delta(p)\leq C(\delta)\text{dist}_M(p,\partial C_i)$:
\begin{align*}
    \sum_i&\iint_{C_i\times(M\setminus C_i)}\eta_\delta^2(p) U^2(p){\rm dist}_M(p,q)^{-(n+s)}\,dp\,dq\\
    &\leq C\iint_{C_i\times(M\setminus C_i)}\eta_\delta^2(p) U^2(p)\text{dist}_M(p,q)^{-(n+s)}\,dp\,dq\leq C\sum_i\int_{C_i}\eta_\delta^2(p)U^2(p)\text{dist}_M(p,\partial C_i)^{-s}\\
    &\leq C(\delta)\sum_i\int_{C_i}U^2(p)\text{dist}_M(p,\partial C_i)^{2-s}\leq C(\delta)\sum_i\int_{C_i}U^2(p)\\
    &\leq C(\delta)\sum_i\|U\|_{H^{s/2}(C_i)}^2\leq C(\delta)\sum_i\|u\|_{H^{s/2}(\mathcal C_i)}^2\,.
\end{align*}
In other words, we have found that $\|U_\delta\|_{H^{s/2}(M)}^2=\|U_\delta\|_{L^2(M)}^2+[U_\delta]_{H^{s/2}(M)}^2\leq C(\delta)\sum_i\|u\|_{H^{s/2}(\mathcal C_i)}^2$. Applying this to $u_1-u_2$ instead, we conclude that our assignation $f_{\delta}(u)=U_\delta$ is Lipschitz. This shows that $\mathcal H^{n-1}(A_\delta')\leq C(\delta)\mathcal H^{n-1}(A)<\infty$ and therefore $A_\delta'\in\mathcal F_\p$.

We will now show that for $\delta>0$ small enough depending on $A$, we have that
\begin{equation}\label{ACUUddif}
    \sup_{u\in A} |\mathcal E_\ep^d(U,M)-\mathcal E_\ep^d(U_\delta,M)|\leq \frac{C}{1-s}
\end{equation}
for some constant $C$ depending only on the $C_i\subset M$ (but not on $A$). For that, we will separately show that $$\|U_\delta-U\|_{H^{s/2}(M)}\leq \frac{C}{1-s}$$ and
$$\sup_{u\in A}|\mathcal E_\ep^{\rm Pot}(U_\delta,M)-\mathcal E_\ep^{\rm Pot}(U,M)|\leq C$$
for $\delta>0$ small enough (in fact, the second error can be made to go to zero).

Let $V=U-U_\delta=(1-\eta_\delta)U$. The computations above with $V$ in place of $U_\delta$ give that
\begin{align*}
    [V]_{H^{s/2}(M)}
    &\leq 2\sum_i \iint_{C_i\times C_i} (1-\eta_\delta)^2(p) (U(p)- U(q))^2{\rm dist}_M(p,q)^{-(n+s)}\,dp\,dq\\
    &\hspace{0.5cm} +2\sum_i\iint_{C_i\times C_i} U(p)^2(\eta_\delta(p) -\eta_\delta(q))^2{\rm dist}_M(p,q)^{-(n+s)}\,dp\,dq\\
    &\hspace{0.5cm} +4\sum_{i} \iint_{C_i\times(M\setminus C_i)}(1-\eta_\delta)^2(p) U^2(p){\rm dist}_M(p,q)^{-(n+s)}\,dp\,dq\, .
\end{align*}
Since $|U|\leq 1$, we can just bound the last term by $\frac{C}{1-s}$ with a $C$ constant depending on the fractional perimeters of the $C_i$. Moreover, using \eqref{itrop2} and the fact that $\|\eta_\delta\|_{BV(M)}$ is uniformly bounded independently of $\delta$ (which is clear from its definition as a regularisation of a characteristic function) we can bound the second term:
\begin{align*}
    \sum_i\iint_{C_i\times C_i} U(p)^2(\eta_\delta(p) -\eta_\delta(q))^2{\rm dist}_M(p,q)^{-(n+s)}\,dp\,dq&\leq C\iint_{M\times M} (\eta_\delta(p) -\eta_\delta(q))^2{\rm dist}_M(p,q)^{-(n+s)}\,dp\,dq\\
    &\leq \frac{C}{1-s}\,.
\end{align*} 
Furthermore, from the definition of $\eta_\delta$ we can bound
\begin{align*}
    \|U\|_{L^2(M)}&=\sum_i \iint_{C_i} (1-\eta_\delta)^2(p) U(p)^2\,dp\leq \sum_i \int_{C_i\cap\{p:{\rm dist}(p,\partial C_i)\leq \delta\}} U(p)^2\,dp
\end{align*}
and
\begin{align*}
    &\sum_i \iint_{C_i\times C_i} (1-\eta_\delta)^2(p) (U(p)- U(q))^2{\rm dist}_M(p,q)^{-(n+s)}\,dp\,dq\leq\\
    &\qquad\leq \sum_i \iint_{(C_i\cap\{p:{\rm dist}(p,\partial C_i)\leq \delta\})\times C_i} (U(p)- U(q))^2{\rm dist}_M(p,q)^{-(n+s)}\,dp\,dq\,,
\end{align*}
both of which go uniformly to zero as $\delta\to 0$ thanks to the equi-integrability in $H^{s/2}$ of the $u$ in $A$ (which follows immediately from the compactness of $A$, or more precisely the equivalent characterisation of being a totally bounded set of functions, together with the fact that a finite family of functions is always equi-integrable).
Putting everything together, we see that 
$$\|U_\delta-U\|_{H^{s/2}(M)}=\|V\|_{H^{s/2}(M)}\leq \frac{C}{1-s}$$
for $\delta=\delta(A)$ small enough.

Finally, regarding the difference in potential energies, we can write
\begin{align*}
    \mathcal E_\ep^{\rm Pot}(U_\delta,M)&=\frac{1}{\ep^s}\int_M W(U_\delta(p))\,dp=\frac{1}{\ep^s}\int_M W(\eta_\delta(p) U(p))\,dp\,,
\end{align*}
so that (since $|U|,|U_\delta|,|\eta_\delta|\leq 1$)
\begin{align*}
    |\mathcal E_\ep^{\rm Pot}(U_\delta,M)-\mathcal E_\ep^{\rm Pot}(U,M)|&\leq \frac{C}{\ep^s}|\{p:U_\delta(p)\neq U(p)\}|=\frac{C}{\ep^s}|\{p:{\rm dist}(p,\cup_i \partial C_i)\leq \delta\}|\leq C
\end{align*}
for $\delta$ small enough depending on $\ep$.\\

With this, we end the proof of \eqref{ACUUddif}. Using the $A_\delta$ instead of $A$ in Steps 1--3, the only difference being that we need to consider an additional constant error term on the RHS of \eqref{Uucin} (and which immediately gets killed after \eqref{i4in} since it gets divided by $\p^{s/n}$ before letting $\p\to\infty$), we finish the proof of Proposition \ref{supleq2}.

\section{Some Riemannian computations}\label{sec:riemcomput}
\begin{proof}[Proof of \eqref{inprodcomp} from \eqref{inprodcomp0}]
For a function $f(q')$ on $M$, its gradient $\nabla_{q'}^M f$ at $q'=\varphi_q(y)$ is represented in coordinates by the Euclidean vector $\varphi_q^{*}(\nabla_{q'}^M f)=g^{-1}(y)\nabla_y^{\R^n} f(\varphi(y))$. Here $g^{-1}$ denotes the inverse of the matrix $g=(g_{ij})_{ij}$ corresponding to the metric of $M$ in coordinates, and the product is to be understood as a matrix product. This shows, then, that $\nabla_{q'}^M K_{s-2}(\varphi_q(0),\varphi_q(y))$ corresponds to $g^{-1}\nabla_y^{\R^n}K_{s-2}(\varphi_q(0),\varphi_q(y))$ in coordinates.\\

Regarding the normal vector $\nu_{\varphi_q(\hat\Gamma)}(\varphi_q(y))$, it corresponds to 
\begin{align}\label{nucriem}
\varphi^{*}(\nu_{\varphi_q(\hat\Gamma)})=\frac{1}{\sqrt{\nu_{\hat\Gamma}^tg^{-1}\nu_{\hat\Gamma}}}g^{-1}\nu_{\hat\Gamma}(y)
\end{align}
in coordinates: Indeed, let $v\in T_{q'}M$ be tangent to the surface $\varphi_q(\hat\Gamma)$, which means that its Euclidean representative $\varphi^{*}(v)$ is tangent to $\hat\Gamma$ at $y$. Then,
    \begin{align*}
        \langle v, \varphi_*(\frac{1}{\sqrt{\nu_{\hat\Gamma}^tg^{-1}\nu_{\hat\Gamma}}}g^{-1}\nu_{\hat\Gamma})\rangle_g=\frac{1}{\sqrt{\nu_{\hat\Gamma}^tg^{-1}\nu_{\hat\Gamma}}}[(\varphi_*)^{-1}(v)]^tgg^{-1}\nu_{\hat\Gamma}=\frac{1}{\sqrt{\nu_{\hat\Gamma}^tg^{-1}\nu_{\hat\Gamma}}}[(\varphi_*)^{-1}(v)]^t\nu_{\hat\Gamma}=0
    \end{align*}
    by orthogonality. Furthermore,
    \begin{align*}
        |\varphi_*(\frac{1}{\sqrt{\nu_{\hat\Gamma}^tg^{-1}\nu_{\hat\Gamma}}}g^{-1}\nu_{\hat\Gamma})|_g=\frac{1}{\nu_{\hat\Gamma}^tg^{-1}\nu_{\hat\Gamma}}\nu_{\hat\Gamma}^tg^{-1}gg^{-1}\nu_{\hat\Gamma}=\frac{1}{\nu_{\hat\Gamma}^tg^{-1}\nu_{\hat\Gamma}}\nu_{\hat\Gamma}^tg^{-1}\nu_{\hat\Gamma}=1\, .
    \end{align*}
These two properties imply that $\varphi_*(\frac{1}{\sqrt{\nu_{\hat\Gamma}^tg^{-1}\nu_{\hat\Gamma}}}g^{-1}\nu_{\hat\Gamma})$ is the unit vector orthogonal to $\varphi_q(\hat\Gamma)$, as we wanted to see.\\

Combining our computations so far, we conclude that
\begin{align*}
\langle\nu_{\varphi_p(\hat\Gamma)}(\varphi_p(y)),\nabla_{q'}^M K_{s-2}(\varphi_p(0),\varphi_p(y))\rangle_{g(\varphi_p(y))}&=\frac{1}{\sqrt{\nu_{\hat\Gamma}^t g^{-1} \nu_{\hat\Gamma}}}\big[g^{-1}\nu_{\hat\Gamma}(y)\big]^t g \big[g^{-1}\nabla_y^{\R^n}K_{s-2}(\varphi_p(0),\varphi_p(y))\big]\\
&=\frac{1}{\sqrt{\nu_{\hat\Gamma}^t g^{-1} \nu_{\hat\Gamma}}}\big[\nu_{\hat\Gamma}(y)\big]^t g^{-1}(y) \nabla_y^{\R^n}K_{s-2}(\varphi_p(0),\varphi_p(y))\,,
\end{align*}
which leads to \eqref{inprodcomp} from \eqref{inprodcomp0}.
\end{proof}

\begin{proof}[Proof of \eqref{mcrequl}]
We now show \eqref{mcrequl}, i.e. that the mean curvature of $\varphi_p(\hat\Gamma)$ at $p\in M$ is the same as the mean curvature of $\hat\Gamma$ at $0\in\R^n$ thanks to having taken normal coordinates. Let $X_i=\varphi_*(e_i)$, where $e_1,...,e_n$ is the standard basis of $\R^n$. Consider the second fundamental form of $\varphi_p(\hat\Gamma)$ at $p$, represented by the matrix $\Big(\langle \nabla_{X_i} \nu_{\varphi_p(\hat\Gamma)},X_j\rangle_g(p)\Big)_{ij}$ in our coordinates. Here $\nabla_V W$ indicates covariant differentiation with respect to the Levi-Civita connection, which letting $\varphi^*(V)=V^ie_i$ and $\varphi^*(W)=W^je_j$ has the coordinate expression
$$
\varphi^*(\nabla_V W)=V^i\frac{\partial W^j}{\partial x^i}e_j+V^iW^j\Gamma_{ij}^ke_k\,.
$$
Now, in normal coordinates, the metric is Euclidean at the origin and its first derivatives vanish there; this also gives that the Christoffel symbols $\Gamma_{ij}^k$ of the Levi-Civita connection vanish at $0$. Therefore, we can compute
\begin{align*}
    \langle \nabla_{X_i} \nu_{\varphi_p(\hat\Gamma)},X_j\rangle_g(p)&=\big([\varphi^*(\nabla_{X_i} \nu_{\varphi_p(\hat\Gamma)})]^tg[e_j]\big)(0)=\frac{\partial\varphi^*(\nu_{\varphi_p(\hat\Gamma)})}{\partial x^i}\Big|_{x=0}\cdot e_j\,.
\end{align*}
Now, differentiating \eqref{nucriem} and recalling once again that $g$ and its first derivatives are Euclidean at the origin, we see that actually $\frac{\partial\varphi^*(\nu_{\varphi_p(\hat\Gamma)})}{\partial x^i}\Big|_{x=0}=\frac{\partial\nu_{\hat\Gamma}}{\partial x^i}\Big|_{x=0}$, thus we arrive at
\begin{align*}
    \langle \nabla_{X_i} \nu_{\varphi_p(\hat\Gamma)},X_j\rangle_g(p)&=\frac{\partial\nu_{\hat\Gamma}}{\partial x^i}\Big|_{x=0}\cdot e_j\,.
\end{align*}
On the other hand, the expression on the right is actually the $(i,j)$-th entry of the second fundamental form of $\hat\Gamma$ at $0$, which shows that actually both matrices are equal. Since the traces of the second fundamental forms are precisely the corresponding mean curvatures, this concludes our proof.
\end{proof}

\section{Gradient behaviour of the kernel $K_s$ -- Proof of Proposition \ref{gradcomp}}\label{app:gradk}
We will first show three intermediate results on the gradient behaviour of the heat kernel $H_M$, after which the proof of Proposition \ref{gradcomp} will be given.
\begin{proposition}\label{prop:hkexploc}
    Let $M$ be a closed, $n$-dimensional Riemannian manifold. Assume that the flatness assumptions ${\rm FA}_\ell(M,g,4,p_0,\varphi)$ are satisfied, with $l=n+2$. Let $H_M(p,q,t)$ be the heat kernel of $M$, and let $E_M:=\frac{1}{(4\pi t)^{n/2}}e^{-\frac{{\rm dist}_M^2(p,q)}{4t}}$. Then, writing $H(x,y,t): = H_M(\varphi(x),\varphi(y),t)$ and $E(x,y,t):=E_M(\varphi(x),\varphi(y),t)$, there are constants $C,c,r>0$ depending only on $n$ such that
    \begin{align*}
        |\nabla_y (H-E)|(x,y,t)\leq C|x-y|\Big(1+\frac{|x-y|^2}{t}\Big)E+Ct^{l+1}+Ce^{-c/t}
    \end{align*}
    for all $(x,y,t)\in\B_{r}(0) \times \B_{r}(0)\times (0,1]$.
\end{proposition}
\begin{proof}[Proof of Proposition \ref{prop:hkexploc}]
\textbf{Step 1.} Reduction to the case of a torus with perturbed metric.\\
As usual, $g_{ij}$ will denote the coefficients of the metric $g$ in the coordinates given by $\varphi$. We want to embed isometrically a piece of our manifold $M$ inside a torus with perturbed metric (from its original flat one). For that, let $\eta:\R^n\to\R$ be a standard cutoff, with $\chi_{\B_1}\le \eta \le \chi_{\B_2}$. Consider the $n$-dimensional flat torus $(\mathbb T^n,\delta_{ij})$ of side length $8$, obtained from the hypercube of side length $8$ centered at $0\in\R^n$ by identifying opposite faces, and let $\varphi':\R^n\to\mathbb T^n$ denote the quotient map defining $\mathbb T^n$. Define a new Riemannian manifold $M'=(\mathbb T^n,g')$, where $g'_{ij}:=\eta g_{ij}+(1-\eta)\delta_{ij}$ is just the usual flat metric except on $\varphi'(\B_2)\subset \mathbb T^n$. It is obvious then that $M'$ satisfies the flatness assumptions ${\rm FA}_\ell(M',g',1,p_0'=0,\varphi')$, since by construction $g_{ij}\equiv g'_{ij}$ in $\B_1$ in the coordinates given by $\varphi$ and $\varphi'$. Thanks to this, as recorded in \cite[Lemma 2.17]{FracSobPaper} the heat kernels of both manifolds in $\B_{1/2}$ are actually quantitatively close, in the sense that 
\begin{align}\label{hhpb}
        |\nabla_y (H-H')|(x,y,t)\leq Ce^{-c/t}\, ;
\end{align}
we reproduce here the precise statement for the convenience of the reader.
\begin{lemma}[\textbf{Localisation principle}, \cite{FracSobPaper}]\label{lemaux2}
Let $(M,g)$ and  $(M',g')$ be  two Riemannian $n$-manifolds. Assume that both $M$  and $M'$ satisfy the flatness assumptions ${\rm FA_\ell}(M,g,1,p_0, \varphi)$ and ${\rm FA_\ell}(M',g,1,p_0', \varphi')$ respectively, and suppose that $g_{ij}\equiv g'_{ij}$ in $\B_1(0)$ in the coordinates induced by $\varphi^{-1}$ and $(\varphi')^{-1}$. 

Then, letting $H(x,y,t): = H_M(\varphi(x),\varphi(y),t)$ and $H'(x,y,t): = H_{M'} (\varphi'(x),\varphi'(y),t)$, we have that the difference $(H-H')(x,y,t)$ is of class $C^\ell$ in  $\B_{1/2}(0) \times \B_{1/2}(0)\times [0,\infty)$ and 
\[
\bigg|\frac{\partial^{|\alpha|}}{\partial x^{\alpha}} \frac{\partial^{|\beta|}}{\partial y^{\beta}} (H-H')(x,y,t)\bigg| \le C \exp(-c/t)\quad \mbox{for } (x,y,t)\in \B_{1/2}(0) \times \B_{1/2}(0)\times [0,\infty),
\]
whenever  $\alpha$ and $\beta$ are multi-indices satisfying $|\alpha|+|\beta| \le \ell$, 
with $C,c>0$ depending only on $n$ and $\ell$.
\end{lemma}
Moreover, defining $E_{M'}:=\frac{1}{(4\pi t)^{n/2}}e^{-\frac{{\rm dist}_{M'}^2(p,q)}{4t}}$ and $E'(x,y,t):=E_{M'}(\varphi(x),\varphi(y),t)$, there is $r=r(n)>0$ such that
\begin{align}\label{eepb}
        |\nabla_y (E-E')|(x,y,t)=0
\end{align}
for $(x,y,t)\in\B_{r}(0) \times \B_{r}(0)\times (0,\infty)$, simply because $E\equiv E'$ in this domain. Indeed, by a trivial comparison with the Euclidean metric, the flatness assumptions imply that there is some small $r=r(n)>0$ for which if $(x,y)\in\B_{r}(0) \times \B_{r}(0)$, then the distances ${\rm dist}_M(\varphi(x),\varphi(y))$ and ${\rm dist}_{M'}(\varphi'(x),\varphi'(y))$ are completely determined by the metrics $g_{ij}\equiv g'_{ij}$ in $\B_1$, and therefore they are actually identical. Since $E$ and $E'$ depend only on ${\rm dist}_M$ and ${\rm dist}_{M'}$, we conclude \eqref{eepb}.\\

Combining \eqref{hhpb} and \eqref{eepb}, we deduce that
\begin{align}\label{hepb}
        |\nabla_y (H-E)|(x,y,t)\leq |\nabla_y (H'-E')|(x,y,t)+ Ce^{-c/t}
\end{align}
for $(x,y,t)\in\B_{r}(0) \times \B_{r}(0)\times (0,\infty)$, thus reducing our problem to showing the proximity of $H'$ and $E'$.\\

\textbf{Step 2.} Closeness between $\nabla_y H'$ and $\nabla_y E'$.\\
Now, the closeness between $ H_{M'}$ and $ E_{M'}$ is precisely the content of the heat kernel expansion on a compact manifold. A precise statement that will serve for our purposes is the following:
\begin{proposition}[\textbf{Heat kernel expansion}]\label{prop:hkexp}
    Let $(N,h)$ be a closed, $n$-dimensional Riemannian manifold. Consider the new manifold $(N,h')$, and assume that $\frac{1}{C}h\le h' \le Ch$ and $\| h'\|_{C^k((N,h))}\leq C$ for all $1\le k \le l=n+2$. The latter means that, given any pair $X,Y$ of vector fields on $(N,h)$ with $C^l$ norm bounded by $1$, $h'(X,Y)$ is a function on $(N,h)$ with $C^l$ norm bounded by $C$.\\
    Then, there is
    a constant $C'=C'(h,C,l)$ such that there exist smooth functions $u_0,...,u_l$ on $M\times M$ with $C^2$ norms bounded by $C'$ such that
    $$\|H_{N,h'}-E_{N,h'}(u_0+tu_1+...+t^lu_l)\|_{C^2(M\times M)}\leq C't^{l+1}$$
    for all $t\in (0,1]$. Moreover, $u_0(p,p)\equiv 1$ and $\nabla_y u_0(p,p)\equiv 0$.
\end{proposition}
\begin{proof}[Proof of Proposition \ref{prop:hkexp}]
    For $h'=h$, i.e. without quantifying the result on the metric, this is a standard result obtained by constructing the heat kernel on a closed manifold using some parametrix approach. The fact that the coefficients in the expansion can be quantified in terms of the metric, in the sense that the result holds uniformly for $h'$ satisfying the assumptions in the present proposition, follows from a careful inspection of the version of the approach mentioned above which uses normal coordinates to construct the parametrix. This is a standard way of constructing the parametrix, which the reader can find for example in \cite[Chapter 3]{Ros}. In this approach, one first constructs locally a parametrix working on normal coordinates around every point of the manifold, by explicitly solving certain ODEs in polar coordinates to define an approximate heat kernel (parametrix) of the form
    $$H^{(0)}=E_{N,h'}(u_0+tu_1+...+t^lu_l)$$
    on a neighbourhood of the diagonal of the manifold. The $u_i$ are explicit in the metric in normal coordinates, and since the latter (together with derivatives) is naturally uniformly controlled for all our admissible $h'$, so are the $u_i$. In the precise case of $u_0$, we have that $u_0(p,q)=\mathcal D^{-1/2}$ with
    $$\mathcal D={\rm det} \big[D_{{\rm exp}_p^{-1}(q)}{\rm exp}_p\big]=1+O({\rm dist}_{N,h'}^2(p,q))\,.$$
    This shows that $u_0(p,p)\equiv 1$ and $\nabla_y u_0(p,p)\equiv 0$.
    
    One then proceeds by extending in space the parametrix to a function $H^{(0)}$ on all of $M\times M$ simply by the use of a cutoff which is identically $1$ on a neighbourhood of the diagonal. Of course, our assumptions imply that any fixed choice of such a cutoff function has uniformly bounded $C^l$ norm for all our admissible $h'$.
    
    Finally, one constructs the heat kernel on $M\times M$ iteratively, essentially by considering the error $LH^{(k)}$ of the approximate heat kernel $H^{(k)}$ at a certain step $k$ to vanishing under the heat operator $L:=\partial_t-\Delta$, convolving this error with the approximate heat kernel $H^{(k)}$ itself, and adding this quantity to $H^{(k)}$ to obtain a better approximation $H^{(k+1)}$; see \cite[page 98]{Ros} for the formal computation. This procedure converges to the true (unique) solution $H_{N,h'}$, and the difference/error between $H^{(0)}$ and $H_{N,h'}$ is then given by an explicit infinite sum made of iterated convolutions between $H^{(0)}$ and $LH^{(0)}$. This error difference (together with its first $\lfloor l-n/2\rfloor\geq 2$ spatial derivatives) is then shown to be uniformly bounded by $Ct^{l+1}$, see \cite[Lemma 3.18]{Ros}, for all $t\in(0,1]$, which gives our desired result. Here $C$ is a constant which depends on $H^{(0)}$, $LH^{(0)}$ and the metric tensor of $(N,h')$,  which shows that $C$ is uniformly controlled for all $h'$: since $H^{(0)}$ is explicitly constructed in terms of the metric in normal coordinates (as explained at the beginning), both $H^{(0)}$ and $LH^{(0)}$ (and, of course, the metric tensor itself $h'$ as well) are uniformly controlled for all $h'$ satisfying our assumptions.
\end{proof}
Applying Proposition \ref{prop:hkexp} with $N=\mathbb T^n$, $ h=\delta_{ij}$ and $h'=g'$, and noting that the conditions $\frac{1}{C}h\le h' \le Ch$ and $\| h'\|_{C^k((N,h))}\leq C$ for all $1\le k \le l$ hold for some $C=C(n)$ thanks to the flatness assumptions and the fact that $h=h'$ outside of $\varphi'(\B_2)$, we deduce that
\begin{align*}
        |\nabla_y (H'-E')|(x,y,t)&\leq |\nabla_y (H'-u_0 E')|+|\nabla_y ((u_0-1) E')|\\
        &\leq |\nabla_y (H'-u_0 E')|+|u_0-1||\nabla_y E'|+|E'||\nabla_y (u_0-1)|\\
        &\leq |\nabla_y \big[E'(tu_1+...+t^lu_l)\big]|+Ct^{l+1}+|u_0-1||\nabla_y E'|+|E'||\nabla_y (u_0-1)|\\
        &\leq Ct|\nabla_y E'|+CtE'+Ct^{l+1}+|u_0-1||\nabla_y E'|+E'|\nabla_y (u_0-1)|\,.
\end{align*}
We have bounded the $u_k$ and their derivatives by uniform constants, with the exception of $(u_0-1)$. For this quantity, given that $u_0(p,p)\equiv 1$ and $\nabla^N u_0(p,p)\equiv 0$, bounding the second derivatives of $u_0$ by uniform constants we have the Taylor expansion $|u_0(\varphi(x),\varphi(y))-1|\leq C|x-y|^2$, which leads to
\begin{align}
        |\nabla_y (H'-E')|(x,y,t)&\leq C(t+|x-y|^2)|\nabla_y E'|+C|x-y|E'+Ct^{l+1}\,. \label{grfprb}
\end{align}
From the explicit expression of $E'$ we can bound $|\nabla_y E'|\leq C\frac{|x-y|}{t}E'$. Substituting into \eqref{grfprb} and combining the resulting expression with \eqref{hepb}, together with the fact that $E\equiv E'$ for $(x,y,t)\in\B_{r}(0) \times \B_{r}(0)\times (0,\infty)$ we conclude the proof of Proposition \ref{prop:hkexploc}.
\end{proof}
We can now finally give the proof of Proposition \ref{gradcomp}.
\begin{proof}[Proof of Proposition \ref{gradcomp}]
Define the explicit kernel $\widetilde K_s:= \frac{\alpha_{n,s}}{{\rm dist}^{n+s}(p,q)}$. Observe that, letting $E_t:=\frac{1}{(4\pi t)^{n/2}}e^{-\frac{{\rm dist}(p,q)^2}{4t}}$, the explicit kernel corresponds exactly to $\widetilde K_s=\frac{s/2}{\Gamma(1-s/2)} \int dt\,\frac{E_t}{t^{1+s/2}}$. Therefore, we can write
    \begin{align}
    \nabla_q K_s&=\frac{s/2}{\Gamma(1-s/2)}\nabla_q \int  \frac{H_t}{t^{1+s/2}}dt=\frac{s/2}{\Gamma(1-s/2)}\int \frac{\nabla_q E_t}{t^{1+s/2}}dt+\frac{s/2}{\Gamma(1-s/2)}\int \frac{\nabla_q (H_t-E_t)}{t^{1+s/2}}dt\nonumber\\
    &=\nabla_q \widetilde K_s+\frac{s/2}{\Gamma(1-s/2)}\int \frac{\nabla_q (H_t-E_t)}{t^{1+s/2}}dt\,,\label{kerqthecomp}
    \end{align}
    Now, we can compute
    \begin{align}
        \nabla_q \widetilde K_s(p,q)&=\nabla_q \frac{\alpha_{n,s}}{{\rm dist}_M^{n+s}(p,q)}\nonumber\\
        &=-\alpha_{n,s}(n+s)\frac{\nabla_q \,{\rm dist}(p,q)}{{\rm dist}_M^{n+s+1}(p,q)}\,.\label{kstwgrd}
    \end{align}
    Observe that the gradient $\nabla_q \,{\rm dist}(p,q)$ corresponds to the tangent vector of the geodesic from $p$ to $q$. Recalling that $\varphi={\rm exp}_p$ on its domain of definition, letting $q:=\varphi(y)$ we deduce that $[D_y\varphi]^{-1}(\nabla_q \,{\rm dist}(p,q))=\frac{y}{|y|}$. On the other hand, thanks to this and the Gauss Lemma we can write, for a vector $v\in\R^n$,
    \begin{align*}
        \partial_{v}{\rm dist}(p,\varphi(y))&=[D_q{\rm dist}(p,\cdot)]\circ [D_y\varphi](v)=\langle \nabla_q{\rm dist}(p,q),D_y\varphi(v)\rangle_g\\
        &=\langle D_y({\rm exp}_p)(\frac{y}{|y|}),D_y({\rm exp}_p)(v)\rangle_g=\frac{y}{|y|}\cdot v\,,
    \end{align*}
    which shows that
    \begin{align*}
        \nabla_y{\rm dist}(p,\varphi(y))=\frac{y}{|y|}=[D_y\varphi]^{-1}(\nabla_q \,{\rm dist}(p,q))\,.
    \end{align*}
    We remark that for a function different from the Riemannian distance, the equality between the LHS and RHS would not be true in general.\\
    
    In particular, from \eqref{kstwgrd} we immediately deduce that
    \begin{align*}
        \nabla_y \widetilde K_s(p,\varphi(y))
        &=-\alpha_{n,s}(n+s)\frac{y}{|y|^{n+s+2}}\,,
    \end{align*}
    so that thanks to \eqref{kerqthecomp} and the flatness assumptions we arrive at
    \begin{align}
    \Big|\nabla_y K_s(p,\varphi(y))-[-\alpha_{n,s}(n+s)\frac{y}{|y|^{n+s+2}}]\Big|&=\Big|\nabla_y K_s(p,\varphi(y))-\nabla_y \widetilde K_s(p,\varphi(y))\Big|\nonumber\\
    &\leq C\Big|\nabla_q K_s(p,q)-\nabla_q \widetilde K_s(p,q)\Big|_g\nonumber\\
    &\leq C\frac{s/2}{\Gamma(1-s/2)}\int_0^\infty \frac{|\nabla_q (H_t-E_t)(p,q)|_g}{t^{1+s/2}}dt\nonumber\\
    &\leq C\frac{s/2}{\Gamma(1-s/2)}\int_0^\infty \frac{|\nabla_y (H_t-E_t)(p,\varphi(y))|}{t^{1+s/2}}dt\,.\label{difgheq}
    \end{align}

    We will now bound \eqref{difgheq} by $O\Big(\frac{1}{|y|^{n+s-1}}\Big)$ and conclude.
    
    By Proposition \ref{prop:hkexploc}, we have the estimate
   \begin{align*}
        |\nabla_y (H-E)|(p,y,t)\leq C|y|\Big(1+\frac{|y|^2}{t}\Big)E(p,y,t)+Ct^{l+1}+Ce^{-c/t}
    \end{align*}
    for all $(x,y,t)\in\B_{r}(0) \times \B_{r}(0)\times (0,1]$.\\
    
    Since $l\geq 1$, we can easily bound
    \begin{align*}
    \int_0^1 \frac{t^{l+1}+e^{-c/t}}{t^{1+s/2}}dt=\int_0^1 t^{l-s/2}+\int_0^1\frac{e^{-c/t}}{t^{1+s/2}}dt\leq C\,.
    \end{align*}
    Using that
    $$E_t(p,\varphi(y))=\frac{1}{(4\pi t)^{n/2}}e^{-\frac{{\rm dist}(p,\varphi(y))^2}{4t}}=\frac{1}{(4\pi t)^{n/2}}e^{-\frac{|y|^2}{4t}}$$
    and performing the change of variables $r=\frac{t}{|y|^2}$, we see that
    \begin{align*}
    \int_0^1 \frac{C|y|\Big(1+\frac{|y|^2}{t}\Big)E(p,y,t)}{t^{1+s/2}}\,dt&= C|y|\int_0^1 \frac{\Big(1+\frac{|y|^2}{t}\Big)e^{-\frac{|y|^2}{4t}}}{t^{1+(n+s)/2}}\,dt\\
    &\leq C|y|^{-(n+s-1)}\int_0^{\infty} \frac{1+r^{-1}}{r^{1+(n+s)/2}}e^{-1/r}\,dr\\
    &=C|y|^{-(n+s-1)}\,.
    \end{align*}
    Finally, $\nabla H_t$ and $\nabla E_t$ are uniformly bounded for large times: it is clear for $E_t$ from its explicit formula, and a constant bound for the derivatives of $H_t$ is obtained in the proof of \cite[Theorem 2.13]{FracSobPaper}, by combining Lemma \ref{lemaux2} with \cite[Proposition 2.19]{FracSobPaper}. Therefore,
    \begin{align*}
        \int_1^\infty \frac{|\nabla_y (H_t-E_t)(p,\varphi(y))|}{t^{1+s/2}}dt\leq C\,.
    \end{align*}
    Putting everything together, we conclude the bound for \eqref{difgheq} and the proof of Proposition \ref{gradcomp}.
\end{proof}

\end{document}